\documentclass[a4j]{article}
\usepackage[margin=0.5in]{geometry}
\usepackage{color}
\usepackage{comment}
\usepackage{graphicx}
\usepackage[hang,small,bf]{caption}
\usepackage[subrefformat=parens]{subcaption}
\usepackage{here}
\usepackage{amsmath}
\usepackage{amssymb}
\usepackage{amsthm}
\usepackage{enumerate}
\usepackage{ulem}
\usepackage{physics}

\newtheorem{thm}{Theorem}
\newtheorem{lem}{Lemma}
\newtheorem{prop}{Proposition}

\newtheorem{remark}{Remark}

\makeatletter
\allowdisplaybreaks[1]
\@addtoreset{equation}{section}
\makeatother

\begin{document}
\begin{center}
	{\LARGE  Least squares estimators for discretely observed stochastic processes driven by small fractional noise}
	
	\large
	\vspace{+6mm}
	Shohei Nakajima\footnote{st08m26@akane.waseda.jp \label{tks: nakajima}}; Shota Nakamura\footnote{nakamurashota@akane.waseda.jp \label{tks: nakamura}}${}^{,*}$ and Yasutaka Shimizu\footnote{shimizu@waseda.jp \label{tks: shimizu}} 
	\vspace{+4mm}
	\\
	${}^{\ref{tks: nakajima}, \ref{tks: nakamura},\ref{tks: shimizu}}$Department of Applied Mathematics, Waseda University\\
	\vspace{+5mm}
	December 1, 2021
	\vspace{+2mm}
\end{center}

\begin{abstract}
We study the problem of parameter estimation for discretely observed stochastic differential equations driven by small fractional noise with Hurst index $H\in (0,1)$. Under some conditions, we obtain strong consistency and rate of convergence of the least square estimator(LSE) when small dispersion coefficient $\varepsilon \rightarrow 0$ and $n \rightarrow \infty$. 
\end{abstract}

\vspace{+1mm}
\begin{flushleft}
\end{flushleft}

\section{Introduction}
Let $(\Omega,\mathcal{F}, (\mathcal{F}_t), P)$ be a stochastic basis satisfying the usual conditions. We consider the following  stochastic differential equation $X_t^{\theta,\varepsilon}$ with an unkown parameter $\theta$ 
\begin{eqnarray}
\label{sde1}
\left\{
\begin{split}
X_t^{\theta,\varepsilon}&=\int_0^T b(\theta,X_t^{\theta,\varepsilon})dt+\varepsilon W_t^H \quad(0 \leq t \leq T);\\
X_0^{\theta,\varepsilon}&=x_0. 
\end{split}
\right.
\end{eqnarray}
where $x_0 \in \mathbb{R}$ is a given initial condition, $(W_t^H)$  is a fractional Brownian motion with the given Hurst index $H \in (0,1)$, the true parameter $\theta_0$ lies in a certain set $\Theta \subset \mathbb{R}^d$ which will be specified later on, $\{b(\cdot,\theta);\theta \in \Theta\}$ is a known family of drift coefficients. The main purpose of this paper is to study the least square type estimator for the true drift parameter $\theta_0$ based on the sampling data $( X_{k/n} ^{\theta_0,\varepsilon})_{k=1}^n$ with small disperison  $\varepsilon>0$ and large sample size $n$. 

The asymptotic properties of the maximum likelihood estimator(MLE) for stochastic differential equations driven by the fractional Brownian motion are studied by many authors. (see Kleptsyna and Le Breton \cite{Kleptsyna}; Brouste and Kleptsyna \cite{Brouste}; Tudor and Viens \cite{Tudor}; Rao \cite{Rao}; Kubilius \textit{et al.} \cite{Kubilius}; Hu et al. \cite{Hu1}; Lohvinenko and Ralchenko \cite{Lohvinenko}; Hu \textit{et al.} \cite{Hu2};  Tanaka \textit{et al.} \cite{Tanaka} and Chiba \cite{Chiba}) In addition, the strong consistency of LSE for drift parameters is studied by Neuenkirchi and Tindel \cite{Nuenkirch} and the asymptotic normality for discrete observation from the stochastic differential equation driven by the fractional Brownian motion with the Hurst index $H>\frac{1}{2}$  is studied by Nakajima and Shimizu \cite{Nakajima}.

The parameter estimation problems for diffusion processes with small white noise based o continuous-time observations have been well developed. (see e.g. Kutoyants \cite{Kutoyants1} and \cite{Kutoyants2}; Yoshida \cite{Yoshida1} and \cite{Yoshida2}; Uchida and Yoshida \cite{Uchida}) In the case of the small fractional Brownian motion, Nakajima and Shimizu \cite{Nakajima2} studied the asymptotic properties of MLE based on continuous-time observations. However, to the best of the authors' knowledge,  parameter estimation problems on the discrete observation from the stochastic differential equation driven by small fractional Brownian motion have not been studied. Therefore, in this paper, we studied the asymptotic properties of LSE for discretely observed diffusion processes with small fractional Brownian motion with Hurst index $H \in (0,1)$ when $\varepsilon \rightarrow 0, n \rightarrow \infty$. The main tool to obtain the asymptotic properties of estimators is an inequality evaluation of the stochastic integral by Young inequality for the Young integral. (see Young \cite{Young}) Since we can evaluate the stochastic integral at each point by Young inequality, we obtain strong consistency and asymptotic normality in the sense of almost surely convergence. 

This paper is organized as follows. In Section 2, we state our main results and prepare some auxiliary results for the proofs of our main results. All the proofs are given in Section 3. In Section 4, we provide some simulation studies to confirm our main results.

\section{Main results}
In this section,  we investigate our main results. Using high frequency data $(X_{t_i}^{\theta,\varepsilon})_{i=1}^n$ with $t_{i+1}-t_i=\frac{1}{n}$, the contrast function $L_{\varepsilon, n}(\theta)$ is defined as follows
\begin{align*}
Q_{\varepsilon, n}(\theta)&:=n \sum_{k=0}^{n-1}  \left| \delta X_{t_kt_{k+1}}^{\theta_0, \varepsilon} -\frac{1}{n} b(X_{t_k}^{\theta_0, \varepsilon}, \theta) \right|^2 ,\\
L_{\varepsilon,n} (\theta)&:=Q_{\varepsilon, n}(\theta)-Q_{\varepsilon, n} (\theta_0),
\end{align*}
where $\delta X_{t_kt_{k+1}}^{\theta_0, \varepsilon}=X_{t_{k+1}}^{\theta_0, \varepsilon}-X_{t_k}^{\theta_0, \varepsilon}$. Then, the LSE $\hat{\theta}_{n,\varepsilon}$ is defined as
\begin{align*}
\hat{\theta}_{n,\varepsilon}:=\underset{\theta \in \Theta}{\text{argmin}}   L_{n,\varepsilon}^{\alpha} (\theta).
\end{align*}
Our interest in this paper is to obtain the asymptotic properties of $\hat{\theta}_{n,\varepsilon}$. 

To state our main results, we make some notations. Let $(x_t^{\theta})_{0\leq t \leq T}$ be the solution to the underlying ordinary differential equation (ODE) under the true value of the drift parameter:
\begin{eqnarray*}
\left\{
\begin{split}
dx_t^{\theta}&=b \left( x_t^{\theta},\theta \right)dt,\\
x_0^{\theta}&=x_0.
\end{split}
\right.
\end{eqnarray*}
We denote the space of all functions $f: \mathbb{R} \times \Theta \rightarrow \mathbb{R}$ which is $k$ and $l$ times continuous differentiable by $C^{k,l}(\mathbb{R}, \Theta)$. Moreover $C_{\uparrow}^{k,l}(\mathbb{R},\Theta)$ is a class of $f\in C^{k,l}(\mathbb{R},\Theta)$ satisfying that $\underset{\theta \in \Theta}{\sup} \left| \partial_x \partial_{\theta_i} f(x,\theta) \right| \leq C(1+\left|x\right|^{\lambda})$ for universal positive constants $C$ and $\lambda$, where $\theta_i$ is the $i$-th projection of $\theta$. 

We introduce the following assumptions.
\begin{itemize}
\item[(A1)] $\Theta$ is an open bounded convec subset of $\mathbb{R}^d$.
\item[(A2)] $\theta \neq \theta_0$ $\leftrightarrow$ $b(x_t^{\theta_0}, \theta) \neq b(x_t^{\theta_0},\theta_0)$ for at least one value of $t \in [0,1]$.
\item[(A3)] There exists a constant $c>0$ such that
\begin{align*}
\underset{\theta \in \Theta}{\sup} \left| b(x,\theta)- b(y,\theta)\right| &\leq c\left| x-y \right|,\\
\underset{\theta \in \Theta}{\sup} \left| \partial_x b(x,\theta)- \partial_x b(y,\theta)\right| &\leq c\left| x-y \right|,\\
\left| b(x,\theta_1)- b(x, \theta_2) \right| &\leq c \left(1+\left| x \right|^N \right) \left| \theta_1- \theta_2\right|,\\
\underset{\theta \in \Theta}{\sup} \left|b(x,\theta) \right| &\leq c \left(1+|x| \right) ,
\end{align*}
for each $x \in \mathbb{R}$ , $N \in \mathbb{N}$ and $\theta_1,\theta_2 \in \Theta$.
\item[(A4)] $b(\cdot, \cdot) \in C_{\uparrow}^{2,3}\left(\mathbb{R}, \Theta \right)$.
\item[(A5)] $I(\theta_0)=\left( I^{i,j}(\theta_0)\right)_{1\leq i,j \leq d}$ is positive definite, where
\begin{align*}
I^{i,j}(\theta):=\int_0^1 \left( \partial_{\theta_i}b\left(x_s^{\theta_0},\theta\right) \right)^{\top} \partial_{\theta_j}b\left(x_s^{\theta_0},\theta\right) ds.
\end{align*}
\end{itemize}

\begin{thm}
\label{consistency}
Suppose that (A1)--(A3) and either of the following conditions hold:
\begin{itemize}
\item[(1)] $H>\frac{1}{2}$.
\item[(2)] $H\leq \frac{1}{2}$ and $\varepsilon^2 n^{1-H} \rightarrow 0 \quad (n \rightarrow \infty, \varepsilon \downarrow 0)$.
\end{itemize}
Then we have $\hat{\theta}_{n,\varepsilon} \rightarrow \theta_0\  \text{a.s.} \quad(n\rightarrow \infty, \varepsilon \downarrow 0)$.
\end{thm}

\begin{thm}
\label{main3}
Suppose that (A1)--(A5) and either of the following conditions holds:
\begin{itemize}
\item[(1)] $H>\frac{1}{2}$ and $\varepsilon n \rightarrow \infty . $
\item[(2)] $H\leq \frac{1}{2}$, $\varepsilon n^{1-H} \rightarrow 0$ and $\varepsilon n \rightarrow \infty . $
\end{itemize}
Then we have 
\begin{align}
\label{main}
\varepsilon^{-1} \left( \hat{\theta}_{n, \varepsilon} - \theta_0 \right) \rightarrow I^{-1}(\theta_0) S(\theta_0) \quad \text{a.s.} \quad(n\rightarrow \infty, \varepsilon \downarrow 0),
\end{align}
where
\begin{align*}
I(\theta_0)&:=\left( \int_0^T \partial_{\theta_i} b\left(x_s^{\theta_0}, \theta_0 \right) \partial_{\theta_j} b\left(x_s^{\theta_0}, \theta_0 \right) ds \right)_{1\leq i,j \leq d},\\
S(\theta_0)&:=\left( \int_0^T \partial_{\theta_i} b\left( x_s^{\theta_0}, \theta_0 \right) dW_s^H \right)_{1\leq i,j \leq d}.
\end{align*}
\end{thm}

\begin{remark}
The stochastic integral $S(\theta_0):=\int_0^T \partial_{\theta_i} b \left( x_s^{\theta_0}, \theta_0 \right) dW_s^H$  is defined  as Young integral. (see Young \cite{Young})  When $H = 1/2$, $S(\theta_0)$ is coincide with the It\^{o} integral. 
\end{remark}

\section{Proof}

\subsection{Auxiliary results}
We shall prepare a stochastic integral for fractional Brownian motion and its properties. Let $C_\lambda(A)$ be the space of all functions $f:A \rightarrow \mathbb{R}$ which is $\lambda$-H\"{o}lder continuous function. It is well known that the Riemann-Stieltjes integral $\int_a^b f(t) dg(t)$ exists for any $f \in C_{p}([a,b])$ and $g \in C_{q}([a,b])$ with $p+q >1$. (see Young \cite{Young}) Also,  the  the change  of variable is valid as the following. (see Z\"{a}hle \cite{Zahle}) 
\begin{lem}
Let $f \in C_{p}([a,b])$ with $p>\frac{1}{2}$ and $F \in C^1(\mathbb{R})$. Then we have
\begin{align*}
F\left( f(y) \right) -F \left(f(a) \right)=\int_a^y F'\left( f(t) \right) df(t),
\end{align*}
for any $y \in [a,b]$.
\end{lem} 

For estimating the Young integral, we note the following estimate, which can be found e.g. in Young \cite{Young}.
\begin{lem}
\label{Young}
Let $f \in C_{\lambda}([a,b])$ and  $g \in C_{\mu}([a,b])$ with $\lambda + \mu>1$. Then, there exists a constant $c_{\lambda, \mu}$ such that
\begin{align*}
\left| \int_a^b \left( f(s)- f(a) \right) dg(s) \right| \leq c_{\lambda, \mu} \left\| f \right\|_{\lambda;[a,b]} \left\| g \right\|_{\mu;[a,b]} |b-a|^{\lambda ; \mu},
\end{align*}
where $\| \cdot \|_{\lambda;[a,b]}$ is the H\"{older} norm on $C_{\lambda}([a,b])$.
\end{lem}

\subsection{ Proof of the strong consistency}
We consider the following proposition which is the convergence of the contrast function $L_{n,\varepsilon}(\theta)$ uniformly with respect to $\theta \in \Theta$.

\begin{prop}
\label{prop}
Suppose that (A1)--(A3) and $ \varepsilon^2 n^{\left( 1-H \right)} \rightarrow 0 \quad(n \rightarrow \infty, \varepsilon \downarrow 0)$ hold.Then we have
\begin{align*}
\underset{\theta \in \Theta}{\sup} \left| Q_{n,\varepsilon} (\theta) - Q_{n,\varepsilon} (\theta_0) - \int_0^T \left| b(x_s^{\theta_0},\theta)- b(x_s^{\theta_0}, \theta_0) \right|^2 ds \right| \rightarrow 0 \quad \left(n \rightarrow \infty, \varepsilon \downarrow 0\right).
\end{align*}
\end{prop}

\begin{lem}
Under a assumption (A3), for any $\theta \in \Theta$ and $0 \leq t \leq T$, we have
\begin{align*}
\underset{\substack{\theta \in \Theta \\ 0 \leq t \leq T }}{\sup} \left| X_t^{\theta, \varepsilon} - x_t^{\theta, \varepsilon} \right| \leq \varepsilon e^{cT} \underset{0 \leq t \leq T}{\sup} \left|W_t^H \right| .
\end{align*}
In particular, 
\begin{align*}
\underset{\substack{\theta \in \Theta \\ 0 \leq t \leq T }}{\sup}  \left| X_t^{\theta, \varepsilon} - x_t^{\theta, \varepsilon} \right| \rightarrow 0 \quad(\varepsilon \downarrow 0).
\end{align*}
\end{lem}

\begin{proof}
For any $\theta \in \Theta$ and $0 \leq t \leq T$, we have
\begin{align*}
\left| X_t^{\theta, \varepsilon}-x_t^{\theta, \varepsilon} \right| &\leq \int_0^t \left| b\left( X_s^{\theta, \varepsilon}, \theta \right) - b\left( x_s^{\theta, \varepsilon}, \theta \right) \right| ds+ \varepsilon \left| W_t^H \right|\\
&\leq c\int_0^t \left| X_s^{\theta, \varepsilon} -x_s^{\varepsilon} \right| ds + \varepsilon \underset{0 \leq t \leq T}{\sup} \left| W_t^H \right|.
\end{align*}
By Gronwall's inequality, it follows that 
\begin{align*}
\underset{ \theta \in \Theta}{\sup} \left| X_t^{\theta, \varepsilon}-x_t^{\theta, \varepsilon} \right| \leq \varepsilon e^{ct}  \underset{0 \leq t \leq T}{\sup} \left| W_t^H \right|.
\end{align*}
Thus we obtain the first estimate and this proof is completed.
\end{proof}

\begin{lem}
Under assumptions (A1) and (A3), we have
\label{1}
\begin{align}
\label{Q}
Q_{n,\varepsilon}(\theta)-Q_{n,\varepsilon}(\theta_0)=Q_{n,\varepsilon}^{(1)}(\theta)-2Q_{n,\varepsilon}^{(2)}(\theta)+R_{n, \varepsilon}(\theta),
\end{align}
where
\begin{align*}
Q_{n,\varepsilon}^{(1)}(\theta)&:=\frac{1}{n}\sum_{k=0}^{n-1}\left| b(X_{t_k}^{\theta_0, \varepsilon}, \theta) - b(X_{t_k}^{\theta_0, \varepsilon}, \theta_0)\right|^2,\\
Q_{n,\varepsilon}^{(2)}(\theta)&:=\sum_{k=0}^{n-1}\left( b(X_{t_k}^{\theta_0, \varepsilon}, \theta) - b(X_{t_k}^{\theta_0, \varepsilon}, \theta_0) \right) \left( \varepsilon \left(W_{t_{k+1}}^H- W_{t_k}^H \right) \right),\\
R_{n, \varepsilon}(\theta)&:=-2\sum_{k=0}^{n-1}\left( b(X_{t_k}^{\theta_0, \varepsilon}, \theta) - b(X_{t_k}^{\theta_0, \varepsilon}, \theta_0) \right) \int_{t_k}^{t_{k+1}} \left( b(X_s^{\theta_0 ,\varepsilon},\theta_0)-  b(X_{t_k}^{\theta_0, \varepsilon},\theta_0 ) \right) ds.
\end{align*}
In particular,  we have
\begin{align*}
\underset{\theta  \in \Theta}{\sup} \left| R_{n,\varepsilon}(\theta) \right| \rightarrow 0 \quad \text{a.s.} \quad(n\rightarrow \infty, \varepsilon \downarrow 0).
\end{align*}
\end{lem}

\begin{proof}
(\ref{Q}) is follows that
\begin{align*}
&\frac{1}{n} \left(Q_{n,\varepsilon}^{\alpha}(\theta)-Q_{n,\varepsilon}^{\alpha}(\theta_0)\right)\\
&=\sum_{k=0}^{n-1} \left\{ \left| \delta X_{t_k t_{k+1}}^{\theta_0, \varepsilon}-\frac{1}{n} b( X_{t_k}^{\theta_0, \varepsilon}, \theta)\right|^2- \left| \delta X_{t_k t_{k+1}}^{\theta_0, \varepsilon}-\frac{1}{n} b( X_{t_k}^{\theta_0, \varepsilon}, \theta_0)\right|^2 \right\} \\
&=-2\frac{1}{n} \sum_{k=0}^{n-1}\left( b(X_{t_k}^{\theta_0, \varepsilon}, \theta) - b(X_{t_k}^{\theta_0, \varepsilon}, \theta_0) \right)\delta X_{t_k t_{k+1}}^{\theta_0, \varepsilon}+\left(\frac{1}{n}\right)^2 \sum_{k=0}^{n-1}\left( b(X_{t_k}^{\theta_0, \varepsilon}, \theta)^2 - b(X_{t_k}^{\theta_0, \varepsilon}, \theta_0)^2 \right)\\
&=-2\frac{1}{n} \sum_{k=0}^{n-1}\left( b(X_{t_k}^{\theta_0, \varepsilon}, \theta) - b(X_{t_k}^{\theta_0, \varepsilon}, \theta_0) \right)\left( \delta X_{t_k t_{k+1}}^{\theta_0, \varepsilon}-\frac{1}{n} b(X_{t_k}^{\theta_0, \varepsilon}, \theta_0) \right)+\left(\frac{1}{n}\right)^2 \sum_{k=0}^{n-1}\left( b(X_{t_k}^{\theta_0, \varepsilon}, \theta) - b(X_{t_k}^{\theta_0, \varepsilon}, \theta_0) \right)^2\\
&=\frac{1}{n} \left(Q_{n,\varepsilon}^{\alpha,(1)}(\theta)+Q_{n,\varepsilon}^{\alpha,(2)}(\theta)+R_{n+\varepsilon}(\theta) \right).
\end{align*}
Since  the path of the fractional Brownian motion $(W_t^H)_{t\in[0,T]}$ is $\nu$-H\"{o}lder continuous for any $\nu \in (0,H)$, we obtain
\begin{align}
\label{ine}
\left| X_t^{\theta_0,\varepsilon} -X_s^{\theta_0, \varepsilon} \right| &\leq \int_s^t \left| b\left( X_u^{\theta_0, \varepsilon}, \theta_0 \right) \right| du + \varepsilon \left( W_t^H- W_s^H \right) \nonumber\\
&\leq \left(t-s \right) \left\{1+  \left( \varepsilon e^{cT}\underset{0 \leq t \leq T}{\sup}\left| W_t^H \right| +\underset{0\leq t \leq T}{\sup} \left| x_t^{\theta_0} \right|  \right)^N \right\} + \varepsilon \left( t-s \right)^{\lambda} ,
\end{align}
for any $\lambda \in (0,H)$. From (\ref{ine}), we have
\begin{align*}
&\underset{\theta \in \Theta }{\sup}\left| R_{n, \varepsilon}(\theta) \right| \\
&=\underset{\theta \in \Theta}{\sup} \left|  2 \sum_{k=0}^{n-1} \left( b\left(X_{t_k}^{\varepsilon,\theta_0}, \theta \right)- b\left( X_{t_k}^{\theta_0, \varepsilon}, \theta_0 \right) \right) \left( \int_{t_k}^{t_{k+1}} b\left( X_s^{\theta_0, \varepsilon} ,\theta_0 \right)- b\left(X_{t_k}^{\theta_0, \varepsilon}, \theta_0  \right)ds  \right)     \right|\\
&\lesssim  \left\{ 1+ \left( \varepsilon e^{cT}\underset{0\leq t \leq T}{\sup}\left| W_t^H \right| + \underset{0 \leq t \leq T}{\sup} \left| x_t^{\theta_0} \right| \right)^N   \right\}\\
&\times  \sum_{k=0}^{n-1} \left[ \int_{t_k}^{t_{k+1}}  \left(t_{k+1}-t_k \right) \left\{1+  \left( \varepsilon e^{cT}\underset{0 \leq t \leq T}{\sup}\left| W_t^H \right| +\underset{0\leq t \leq T}{\sup} \left| x_t^{\theta_0} \right| \right)^N \right\} + \varepsilon \left( t_{k+1}-t_k \right)^{\lambda}ds \right]\\
&=  \left\{ 1+ \left( \varepsilon e^{cT}\underset{0\leq t \leq T}{\sup}\left| W_t^H \right| + \underset{0 \leq t \leq T}{\sup} \left| x_t^{\theta_0} \right| \right)^N   \right\} \left[ \frac{1}{n}  \left\{1+  \left( \varepsilon e^{cT}\underset{0 \leq t \leq T}{\sup}\left| W_t^H \right| +\underset{0\leq t \leq T}{\sup} \left| x_t^{\theta_0} \right| \right)^N \right\} + \varepsilon \left(\frac{1}{n}\right)^{\lambda} \right]\\
&\rightarrow 0 \quad( n \rightarrow \infty, \varepsilon \downarrow 0).
\end{align*}
\end{proof}

\begin{lem}
\label{riemann}
Let $H \in (0,1)$ and $f \in C_{\uparrow}^{1,1}(\mathbb{R}, \Theta)$. Then we have
\begin{align*}
\underset{\theta \in \Theta}{\sup}\left|\frac{1}{n}\sum_{k=0}^{n-1}f(X_{t_k}^{\theta_0, \varepsilon}, \theta) - \int_0^T f(x_s^{\theta_0},\theta) \right| \rightarrow 0  \quad \text{a.s.} \quad(n\rightarrow \infty, \varepsilon \downarrow 0).
\end{align*}
In particular, under assumptions (A1) and (A3), we obtain
\begin{align*}
\underset{\theta \in \Theta }{\sup} \left| Q_{n,\varepsilon}^{\alpha,(1)}(\theta) - \int_0^T \left| b(x_s^{\theta_0}, \theta)- b(x_s^{\theta_0}, \theta_0)\right|^2 ds\right| \rightarrow 0  \quad \text{a.s.}  \quad(n\rightarrow \infty , \varepsilon \downarrow 0).
\end{align*}
\end{lem}
\begin{proof}
For proof, see Hu \textit{et al.} \cite{Shimizu}, Lemma 3.3.
\end{proof}

\begin{lem}
\label{stochastic integral1}
Let $H \in (0,1)$ and $f \in C_{\uparrow}^{1,1}(\mathbb{R}, \Theta)$ such that $\underset{\theta \in \Theta}{\sup} \left| f(x,\theta)- f(y,\theta)\right| \leq c\left| x-y \right|$. Suppose that $ \varepsilon n^{\left( 1-H \right)} \rightarrow 0 \quad(n \rightarrow \infty, \varepsilon \downarrow 0)$ holds, then we have
\begin{align*}
\underset{\theta \in \Theta}{\sup}  \left|  \sum_{k=0}^{n-1} f(X_{t_k}^{\theta_0,\varepsilon},\theta)(W_{t_{k+1}}^H-W_{t_k}^H) - \int_0^T f(x_s^{\theta_0},\theta) dW_s^H  \right| \rightarrow 0  \quad \text{a.s.} \quad(n\rightarrow \infty, \varepsilon \downarrow 0).
\end{align*}

In addition, (A3) and $ \varepsilon^2 n^{\left( 1-H \right)} \rightarrow 0 \quad(n \rightarrow \infty, \varepsilon \downarrow 0)$ hold , then we obtain

\begin{align*}
\underset{\theta \in \Theta}{\sup} \left| Q_{n,\varepsilon}^{\alpha,(2)}(\theta) \right| \rightarrow 0  \quad \text{a.s.}  \quad(n\rightarrow \infty, \varepsilon \downarrow 0).
\end{align*}
\end{lem}
\begin{proof}
For any $\lambda \in (1-H,1), \lambda' \in (0,H)$ s.t. $\lambda+\lambda'>1$, we have
\begin{align*}
&\underset{\theta \in \Theta}{\sup}  \left|  \sum_{k=0}^{n-1} f(X_{t_k}^{\theta_0,\varepsilon},\theta)(W_{t_{k+1}}^H-W_{t_k}^H) - \int_0^T f(x_s^{\theta_0},\theta) dW_s^H  \right|\\
&\leq \underset{\theta \in \Theta}{\sup}  \left|  \sum_{k=0}^{n-1} f(X_{t_k}^{\theta_0,\varepsilon},\theta)(W_{t_{k+1}}^H-W_{t_k}^H) - \sum_{k=0}^{n-1} f(x_{t_k}^{\theta_0},\theta)(W_{t_{k+1}}^H-W_{t_k}^H)  \right| \\
&+\underset{\theta \in \Theta}{\sup}  \left|  \sum_{k=0}^{n-1} f(x_{t_k}^{\theta_0},\theta)(W_{t_{k+1}}^H-W_{t_k}^H) - \int_0^T f(x_s^{\theta_0},\theta) dW_s^H  \right| \\
&\leq   \underset{\theta \in \Theta}{\sup} \left\{ \left(\sum_{k=0}^{n-1}\left| f\left(X_{t_k}^{\theta_0,\varepsilon},\theta \right)-f\left(x_{t_k}^{\theta_0},\theta \right)\right|^{\frac{1}{1-H}} \right)^{1-H} \left(\sum_{k=0}^{n-1} \left| W_{t_{k+1}}^H-W_{t_k}^H \right|^{\frac{1}{H}} \right)^H \right\} \\
&+ \sum_{k=0}^{n-1} \underset{\theta \in \Theta}{\sup} \left| \int_{t_k}^{t_{k+1}}  f(x_{t_k}^{\theta_0},\theta)-f(x_s^{\theta_0},\theta)  dW_s^H\right| \\
&\uwave{<}  n^{1-H} \underset{0 \leq t \leq T}{\sup} \left| X_t^{\theta_0,\varepsilon}-x_t^{\theta_0} \right| \left(\sum_{k=0}^{n-1} \left| W_{t_{k+1}}^H-W_{t_k}^H \right|^{\frac{1}{H}} \right)^H+ \sum_{k=0}^{n-1} \underset{\theta \in \Theta}{\sup} \left\| f\left(x_{\cdot}^{\theta_0},\theta \right) \right\|_{\lambda; [t_k, t_{k+1}]} \left\| W_{\cdot}^{H}\right\|_{\lambda'; [t_k, t_{k+1}]} \left(\frac{1}{n}\right)^{\lambda+\lambda'}  \\
&\lesssim \varepsilon n^{1-H}\left(\sum_{k=0}^{n-1} \left| W_{t_{k+1}}^H-W_{t_k}^H \right|^{\frac{1}{H}} \right)^H +   n \left(\frac{1}{n} \right)^{1-\lambda} \left(\frac{1}{n} \right)^{H-\lambda'}\left(\frac{1}{n} \right)^{\lambda+\lambda'}\\
&=\varepsilon n^{1-H}\left(\sum_{k=0}^{n-1} \left| W_{t_{k+1}}^H-W_{t_k}^H \right|^{\frac{1}{H}} \right)^H +   \left(\frac{1}{n} \right)^H\\
&\rightarrow 0 \quad(n\rightarrow \infty, \varepsilon \downarrow 0).
\end{align*}
by Lemma \ref{Young}, Neuenkirch and Tindel \cite{Nuenkirch}, Lemma2.5 and Lemma 2.6.
\end{proof}

\begin{lem}
\label{stochastic integral2}
Let $H>\frac{1}{2}$ and $f \in C_{\uparrow}^{1,1}(\mathbb{R} \Theta)$ such taht
\begin{align*}
\underset{\theta \in \Theta}{\sup} \left|  f\left( x, \theta \right) -  f\left(y, \theta \right) \right| &\leq c \left|x-y \right|,\\
\underset{\theta \in \Theta}{\sup} \left| \partial_x f\left( x, \theta \right) - \partial_x f\left(y, \theta \right) \right| &\leq c \left|x-y \right|.
\end{align*}
Then we have 
\begin{align}
\label{lem8}
\underset{\theta \in \Theta}{\sup}\left|\sum_{k=0}^{n-1} f \left( X_{t_k}^{\theta_0,\varepsilon} , \theta \right) \left( W_{t_{k+1}}^H - W_{t_k}^H \right)  -  \int_0^T f(x_t^{\theta_0} \theta ) dW_t^H \right| \rightarrow 0 \quad \text{a.s.} \quad (n \rightarrow \infty, \varepsilon \downarrow 0).
\end{align}
In particular, under the assumptions (A1) and (A4), we obtain
\begin{align*}
\underset{\theta \in \Theta}{\sup} \left| Q_{n,\varepsilon}^{\alpha,(2)}(\theta) \right| \rightarrow 0  \quad \text{a.s.}  \quad(n\rightarrow \infty, \varepsilon \downarrow 0).
\end{align*}
\end{lem}

\begin{proof}
Let $\lambda \in (0,H)$ such that $2\lambda>1$.
\begin{align}
\label{proof}
&\underset{\theta \in \Theta}{\sup} \left| \sum_{k=0}^{n-1} f\left(X_{t_k}^{\theta_0,\varepsilon},\theta \right) \left( W_{t_{k+1}}^H - W_{t_k}^H \right)- \int_0^T f\left(x_t^{\theta_0},\theta \right) dW_t^H \right| \nonumber\\
&\leq \underset{\theta \in \Theta}{\sup} \left| \sum_{k=0}^{n-1} f\left(X_{t_k}^{\theta_0,\varepsilon},\theta \right) \left( W_{t_{k+1}}^H - W_{t_k}^H \right) -  \int_0^T f\left(X_t^{\theta_0, \varepsilon}, \theta \right) dW_t^H\right| + \underset{\theta \in \Theta}{\sup} \left|  \int_0^T f\left(X_t^{\theta_0, \varepsilon}, \theta \right) dW_t^H -  \int_0^T f\left(x_t^{\theta_0}, \theta \right) dW_t^H \right| \nonumber\\
&= \underset{\theta \in \Theta}{\sup} \left| \sum_{k=0}^{n-1} \int_{t_k}^{t_{k+1}} \left\{ f\left( X_{t_k}^{\theta_0,\varepsilon}, \theta \right)-f\left(X_t^{\theta_0 , \varepsilon}, \theta \right)\right\} dW_t^H\right|\nonumber\\
&+\underset{\theta \in \Theta}{\sup} \left| \int_0^T \left\{ f\left(X_t^{\theta_0, \varepsilon},\theta \right)- f\left( x_t^{\theta_0}, \theta \right) -\left(f\left(x_0^{\theta_0}, \theta\right) -f\left(x_0^{\theta_0}, \theta\right) \right) \right\} dW_t^H \right|\nonumber\\
&\leq \sum_{k=0}^{n-1} \underset{\theta \in \Theta}{\sup} \left\| f\left( X_{\cdot}^{\theta_0, \varepsilon}, \theta\right) \right\|_{\lambda;[t_k,t_{k+1}]} \left\|W_{\cdot}^H\right\|_{\lambda;[t_k,t_{k+1}]}\left(\frac{1}{n}\right)^{2\lambda} + \underset{\theta \in \Theta}{\sup} \left\| f\left(X_{\cdot}^{\theta_0,\varepsilon},\theta \right)- f\left(x_{\cdot}^{\theta_0}, \theta \right)\right\|_{\lambda;[0,T]}\left\| W_{\cdot}^H \right\|_{\lambda; [0,T]}T^{2\lambda}\nonumber\\
&\lesssim \sum_{k=0}^{n-1}  \left[ \left\{1+  \left(\varepsilon e^{cT}\underset{0 \leq t \leq T}{\sup}\left| W_t^H \right| +\underset{0\leq t \leq T}{\sup} \left| x_t^{\theta_0} \right| \right)^N \right\} \left(\frac{1}{n}\right)^{1-\lambda}+\varepsilon \left\|W_{\cdot}^H \right\|_{\lambda; [t_k,t_{k+1}]} \right]\left\|W_{\cdot}^H \right\|_{\lambda; [t_k,t_{k+1}]} \left(\frac{1}{n} \right)^{2\lambda}\nonumber\\
&+\underset{\theta \in \Theta}{\sup} \left\| f\left(X_{\cdot}^{\theta_0,\varepsilon},\theta \right)- f\left(x_{\cdot}^{\theta_0}, \theta \right)\right\|_{\lambda;[0,T]}\left\| W_{\cdot}^H \right\|_{\lambda; [0,T]}T^{2\lambda}\nonumber\\
&\lesssim \left[  \left\{1+  \left(\varepsilon e^{cT}\underset{0 \leq t \leq T}{\sup}\left| W_t^H \right| +\underset{0\leq t \leq T}{\sup} \left| x_t^{\theta_0} \right| \right)^N \right\} \left(\frac{1}{n}\right)^H +\varepsilon  \left(\frac{1}{n}\right)^{2H-1} \right]\nonumber\\
&+\underset{\theta \in \Theta}{\sup} \left\| f\left(X_{\cdot}^{\theta_0,\varepsilon},\theta \right)- f\left(x_{\cdot}^{\theta_0}, \theta \right)\right\|_{\lambda;[0,T]}\left\| W_{\cdot}^H \right\|_{\lambda; [0,T]}T^{2\lambda}.
\end{align}
Note that 
\begin{align*}
f\left(X_t^{\theta_0,\varepsilon}, \theta \right)-f\left(x_0, \theta \right)&=\int_0^t \partial_x f\left(X_s^{\theta_0,\varepsilon},\theta \right) dX_s^{\theta_0, \varepsilon}\\
&=\int_0^t \partial_x f\left(X_s^{\theta_0,\varepsilon},\theta \right) b\left(X_s^{\theta_0, \varepsilon},\theta_0\right) ds+\varepsilon \int_0^t \partial_x f\left(X_s^{\theta_0,\varepsilon},\theta \right) dW_s^H,
\end{align*}
and
\begin{align*}
f\left(x_t^{\theta_0}, \theta \right)-f\left(x_0, \theta \right)=\int_0^t \partial_x f\left(x_s^{\theta_0},\theta \right) b\left(x_s^{\theta_0},\theta_0\right) ds,
\end{align*}
hold from Lemma \ref{Young}. Thus, we have
\begin{align*}
&\left| f \left( X_t^{\theta_0, \varepsilon} , \theta \right) - f \left( x_t^{\theta_0}, \theta \right) - \left( f \left( X_s^{\theta_0, \varepsilon},\theta \right) - f\left( x_s^{\theta_0} ,\theta \right) \right) \right|\\
&\leq \left| \int_s^t \left\{ \partial_x f\left( X_u^{\theta_0, \varepsilon}, \theta \right) b\left( X_u^{\theta_0, \varepsilon}, \theta_0 \right) -  \partial_x f\left( x_u^{\theta_0}, \theta \right) b\left( x_u^{\theta_0}, \theta_0 \right)  \right\}du \right|+\varepsilon \left| \int_s^t f\left(X_u^{\theta_0, \varepsilon}, \theta \right) dW_u^H \right|\\
&\leq \left| \int_s^t \left\{ \partial_x f\left( X_u^{\theta_0, \varepsilon}, \theta \right) b\left( X_u^{\theta_0, \varepsilon}, \theta_0 \right) -  \partial_x f\left( x_u^{\theta_0}, \theta \right) b\left( X_u^{\theta_0,\varepsilon}, \theta_0 \right)  \right\}du \right| \\
&+\left| \int_s^t \left\{ \partial_x f\left( x_u^{\theta_0}, \theta \right) b\left( X_u^{\theta_0, \varepsilon}, \theta_0 \right) -  \partial_x f\left( x_u^{\theta_0}, \theta \right) b\left( x_u^{\theta_0}, \theta_0 \right)  \right\}du \right|\\
&+\varepsilon \left| \int_s^t f\left( X_u^{\theta_0, \varepsilon},\theta \right) - f\left( X_s^{\theta_0,\varepsilon},\theta \right)dW_u^H \right| + \varepsilon \left| f\left( X_s^{\theta_0, \varepsilon}, \theta \right)\left( W_t^H- W_s^H \right) \right|\\
&\uwave{<} \left(t-s \right) \left( 1+ \underset{0\leq t \leq T}{\sup} \left| X_t^{\theta_0,\varepsilon}\right|^N \right) \underset{0\leq t \leq T}{\sup} \left| X_t^{\theta_0, \varepsilon} -x_t^{\theta_0} \right| + (t-s) \left( 1+\underset{0\leq t \leq T}{\sup} \left| x_t^{\theta_0} \right|^N \right) \underset{0\leq t \leq T}{\sup} \left| X_t^{\theta_0, \varepsilon} -x_t^{\theta_0} \right|\\
&+ \varepsilon \left\| f \left( X_{\cdot}^{\theta_0, \varepsilon} , \theta \right) \right\|_{\lambda;[s,t]} \left\| W_{\cdot}^H \right\|_{\lambda;[s,t]} (t-s)^{2\lambda} + \varepsilon \left( 1+ \underset{0\leq t \leq T}{\sup} \left| X_t^{\theta_0, \varepsilon} \right\|^N \right) (t-s)^{\lambda}\\
&\uwave{<} \varepsilon \left\{ 1+ \left( \underset{0 \leq t \leq T}{\sup} \left| x_t^{\theta_0} \right| + \varepsilon e^{cT} \underset{0 \leq t \leq T}{\sup} \left| W_t^H \right| \right)^N \right\} (t-s) + \varepsilon \left(1+ \underset{0\leq t \leq T}{\sup} \left| x_t^{\theta_0} \right|^N \right)(t-s) \\
&+ \varepsilon \left[ \left\{ 1+ \left( \varepsilon e^{cT} + \underset{0 \leq t \leq T}{\sup} \left| x_t^{\theta_0} \right| \right)^N \right\} (t-s)^{1-\lambda} + \varepsilon \right](t-s)^{H-\lambda}(t-s)^{2\lambda} \\
&+  \varepsilon \left\{ 1+ \left( \underset{0 \leq t \leq T}{\sup} \left| x_t^{\theta_0} \right| + \varepsilon e^{cT} \underset{0 \leq t \leq T}{\sup} \left| W_t^H \right| \right)^N \right\} (t-s)^{\lambda}\\
&= \varepsilon \left\{ 1+ \left( \underset{0 \leq t \leq T}{\sup} \left| x_t^{\theta_0} \right| + \varepsilon e^{cT} \underset{0 \leq t \leq T}{\sup} \left| W_t^H \right| \right)^N \right\} (t-s) + \varepsilon \left(1+ \underset{0\leq t \leq T}{\sup} \left| x_t^{\theta_0} \right|^N \right)(t-s)\\
&+ \varepsilon  \left\{ 1+ \left( \varepsilon e^{cT} + \underset{0 \leq t \leq T}{\sup} \left| x_t^{\theta_0} \right| \right)^N \right\}(t-s)^{1+H} + \varepsilon^2 +\varepsilon \left\{ 1+ \left( \underset{0 \leq t \leq T}{\sup} \left| x_t^{\theta_0} \right| + \varepsilon e^{cT} \underset{0 \leq t \leq T}{\sup} \left| W_t^H \right| \right)^N \right\} (t-s)^{\lambda}\\
&\therefore \underset{\theta \in \Theta}{\sup} \left\| f\left(X_{\cdot}^{\theta_0,\varepsilon},\theta \right)- f\left(x_{\cdot}^{\theta_0}, \theta \right)\right\|_{\lambda;[0,T]} \rightarrow 0 \quad \text{a.s.} \quad(\varepsilon \downarrow 0).\\
&\therefore \underset{\theta \in \Theta}{\sup} \left| \sum_{k=0}^{n-1} f\left(X_{t_k}^{\theta_0,\varepsilon},\theta \right) \left( W_{t_{k+1}}^H - W_{t_k}^H \right)- \int_0^T f\left(x_t^{\theta_0},\theta \right) dW_t^H \right| \rightarrow 0 \quad \text{a.s.} \quad(\varepsilon \downarrow 0, n \rightarrow \infty).
\end{align*}
Therefore, from (\ref{proof}), we have (\ref{lem8}).

\end{proof}

The Proposition \ref{prop} follows immediately from Lemma \ref{1}--\ref{stochastic integral2}. From the identifiability condition (A2) and Proposition \ref{prop}, which implies the convergence of the contrast function $L_{n,\varepsilon}(\theta)$, Theorem \ref{consistency} can be shown using the following lemma. (see Frydman \cite{Frydman} and Kasonga \cite{Kasonga})

\begin{lem}
Assume that the family of random variables $L_{n}(\theta), n \in \mathbb{N}, \theta \in \Theta$, satisfies:
\begin{itemize}
\item[(1)]$ L_n(\theta) \rightarrow L(\theta) \quad \text{a.s.} \quad \text{uniformly in $\theta \in \Theta$} \quad(n \rightarrow \infty).$
\item[(2)] The limit $L$ is non-random and $L(\theta_0) \leq L(\theta)$ for all $\theta  \in \Theta$.
\item[(3)] It holds $L(\theta)=L(\theta_0)$ if and only if $\theta= \theta_0$.
\end{itemize}
Then, we have
\begin{align*}
\hat{\theta}_n \rightarrow \theta_0 \quad \text{a.s.} \quad(n\rightarrow \infty),
\end{align*}
where $\hat{\theta}_n=\underset{\theta \in \Theta }{\text{argmin}}L_n(\theta)$.
\end{lem}

\subsection{Proof of asymptotic normality}
We set the following notation:
\begin{align*}
G_{n, \varepsilon} (\theta) &:=\left( G_{n,\varepsilon}^1(\theta), \cdots , G_{n,\varepsilon}^d(\theta) \right)^{\mathsf{T}},\\
G_{n, \varepsilon}^i(\theta)&:= \sum_{k=0}^{n-1} \partial_{\theta_i}b \left( X_{t_k}^{\theta_0, \varepsilon}, \theta \right) \left(X_{t_{k+1}}^{\theta_0, \varepsilon} - X_{t_{k}}^{\theta_0, \varepsilon} - \frac{1}{n} b \left( X_{t_k}^{\theta_0, \varepsilon}, \theta \right)\right) \quad \left(1\leq i \leq d\right),\\
K_{n, \varepsilon}(\theta) &:= \left( \partial_{\theta_j} G_{n, \varepsilon}^i(\theta) \right)_{1\leq i,j \leq d},\\
K^{i,j}(\theta)&:= \int_0^T \partial_{\theta_j} \partial_{\theta_i} b \left(x_s^{\theta_0}, \theta \right) \left( b\left(x_s^{\theta_0}, \theta_0 \right)- b\left(x_s^{\theta_0}, \theta \right) \right)ds - \int_0^T \partial_{\theta_i} b \left(x_s^{\theta_0}, \theta \right) \partial_{\theta_j} b \left(x_s^{\theta_0}, \theta \right)ds\\
&=\int_0^T \partial_{\theta_j} \partial_{\theta_i} b \left(x_s^{\theta_0}, \theta \right) \left( b\left(x_s^{\theta_0}, \theta_0 \right)- b\left(x_s^{\theta_0}, \theta \right) \right)ds -I^{i,j}(\theta) \quad(1\leq i,j \leq d).
\end{align*}

From Hu \textit{et al.} \cite{Shimizu}, p.431-432, it is sufficient to show the following lemmas to prove Theorem \ref{main3}.

\begin{lem}
Suppose that the assumptions (A1), (A3)--(A4) and either of the following conditions hold. 
\begin{itemize}
\item[(1)] $H>\frac{1}{2}$ and $\varepsilon n \rightarrow \infty . $
\item[(2)] $H\leq \frac{1}{2}$, $\varepsilon n^{1-H} \rightarrow 0$ and $\varepsilon n \rightarrow \infty . $
\end{itemize}
Then we have
\begin{align}
\label{lem9}
\frac{1}{\varepsilon}G_{n,\varepsilon}^i (\theta_0) \rightarrow \int_0^T \partial_{\theta_i} b\left( x_s^{\theta_0}, \theta_0 \right)dW_s^H \quad \text{a.s.} \quad( n \rightarrow \infty, \varepsilon \downarrow 0).
\end{align}
\end{lem}
\begin{proof}
\begin{align*}
\frac{1}{\varepsilon}G_{n,\varepsilon}^i (\theta_0)&=\frac{1}{\varepsilon}\sum_{k=0}^{n-1} \partial_{\theta_i}b \left( X_{t_k}^{\theta_0, \varepsilon}, \theta \right) \left(X_{t_{k+1}}^{\theta_0, \varepsilon} - X_{t_{k}}^{\theta_0, \varepsilon} - \frac{1}{n} b \left( X_{t_k}^{\theta_0, \varepsilon}, \theta \right)\right) \\
&=\frac{1}{\varepsilon}\sum_{k=0}^{n-1} \partial_{\theta_i}b \left( X_{t_k}^{\theta_0, \varepsilon}, \theta \right)  \int_{t_k}^{t_{k+1}} \left( b\left(X_s^{\theta_0,\varepsilon}, \theta_0 \right)-b\left(X_{t_k}^{\theta_0,\varepsilon}, \theta_0 \right) \right)ds + \sum_{k=0}^{n-1} \partial_{\theta_i}b \left( X_{t_k}^{\theta_0, \varepsilon}, \theta \right) \left( W_{t_{k+1}}^H - W_{t_k}^H \right)\\
&=: H_{n , \varepsilon}^{(1)}(\theta_0) + H_{n , \varepsilon}^{(2)}(\theta_0).
\end{align*}
By Lemma \ref{stochastic integral1} and Lemma \ref{stochastic integral2}, we have 
\begin{align*}
H_{n , \varepsilon}^{(2)}(\theta_0) \rightarrow \int_0^T \partial_{\theta_i} b\left( x_s^{\theta_0}, \theta_0 \right)dW_s^H \quad \text{a.s.} \quad( n \rightarrow \infty, \varepsilon \downarrow 0) \quad \text{a.s.} \quad (n \rightarrow \infty, \varepsilon \downarrow 0).
\end{align*}
\begin{align*}
H_{n , \varepsilon}^{(1)}(\theta_0) &\lesssim \frac{1}{n\varepsilon} \sum_{k=0}^{n-1} \left| \partial_{\theta_i} b\left(X_{t_k}^{\theta_0, \varepsilon}, \theta_0 \right) \right| \underset{t_k \leq s \leq t_{k+1}}{\sup} \left| X_s^{\theta_0, \varepsilon} -X_{t_k}^{\theta_0, \varepsilon} \right|\\
&\lesssim \frac{1}{n\varepsilon} \sum_{k=0}^{n-1} \left| \partial_{\theta_i} b\left(X_{t_k}^{\theta_0, \varepsilon}, \theta_0 \right) \right| \left[ \frac{1}{n} \left\{1+\left(\varepsilon e^{cT} \underset{0 \leq t \leq T}{\sup} \left|W_t^H \right| + \underset{0\leq t \leq T}{\sup} \left|x_s^{\theta_0} \right|  \right)^N \right\} + \varepsilon \underset{t_k \leq s \leq t_{k+1}}{\sup} \left| W_s^H- W_{t_k}^H \right| \right]\\
&=\frac{1}{n^2\varepsilon} \sum_{k=0}^{n-1} \left| \partial_{\theta_i} b\left(X_{t_k}^{\theta_0, \varepsilon}, \theta_0 \right) \right|  \left\{1+\left(\varepsilon e^{cT} \underset{0 \leq t \leq T}{\sup} \left|W_t^H \right| + \underset{0\leq t \leq T}{\sup} \left|x_s^{\theta_0} \right| \right)^N\right\} \\
&+ \frac{1}{n} \sum_{k=0}^{n-1} \left| \partial_{\theta_i} b\left(X_{t_k}^{\theta_0, \varepsilon}, \theta_0 \right) \right| \underset{t_k \leq s \leq t_{k+1}}{\sup} \left| W_s^H- W_{t_k}^H \right|\\
&=:H_{n,\varepsilon}^{(1,1)}(\theta_0) + H_{n, \varepsilon}^{(1,2)}(\theta_0).
\end{align*}
\end{proof}
Since $ \frac{1}{n} \sum_{k=0}^{n-1} \left| \partial_{\theta_i} b\left(X_{t_k}^{\theta_0, \varepsilon}, \theta_0 \right) \right|  \lesssim \left(1+ \underset{0 \leq t \leq T}{\sup} \left|X_t^{\theta_0 , \varepsilon} \right|^N \right) < \infty $ holds, we have $H_{n,\varepsilon}^{(1,1)}(\theta_0) \rightarrow 0 \quad(n\rightarrow \infty , \varepsilon \downarrow 0)$. We also obtain
\begin{align*}
 H_{n, \varepsilon}^{(1,2)}(\theta_0) &\lesssim \frac{1}{n} \sum_{k=0}^{n-1} \left( 1+ \underset{0 \leq t \leq T}{\sup} \left| X_t^{\theta_0, \varepsilon} \right|^N \right) \underset{t_k \leq s \leq t_{k+1}}{\sup} \left| W_s^H- W_{t_k}^H \right|\\
 &\uwave{<} \left( \frac{1}{n} \right)^{\lambda} \left( 1+ \underset{0 \leq t \leq T}{\sup} \left| X_t^{\theta_0, \varepsilon} \right|^N \right)\\
 &\rightarrow 0 \quad(n\rightarrow \infty, \varepsilon \downarrow0).
\end{align*}
Thus, we have $H_{n, \varepsilon}^{(1)}(\theta_0) \rightarrow 0 \quad \text{a.s.} \quad (n\rightarrow \infty, \varepsilon \downarrow 0)$. Therefore, we obtain (\ref{lem9}) and this proof is completed.

\begin{lem}
Suppose that the assumptions (A1) and (A3) and either of the following conditions hold. 
\begin{itemize}
\item[(1)] $H>\frac{1}{2}$. 
\item[(2)] $H\leq \frac{1}{2}$, $\varepsilon^2 n^{1-H} \rightarrow 0$.
\end{itemize}
Then we have
\begin{align*}
\underset{\theta \in \Theta}{\sup} \left\| K_{n, \varepsilon}(\theta) - K(\theta) \right\| \rightarrow 0 \quad \text{a.s.} \quad (n\rightarrow \infty, \varepsilon \downarrow 0).
\end{align*}
\end{lem}
\begin{proof}
\begin{align*}
K_{n,\varepsilon}^{i,j} (\theta)&=\partial_{\theta_j} G_{n,\varepsilon}^i (\theta)\\
&=\sum_{k=0}^{n-1} \partial_{\theta_j} \partial_{\theta_i}b \left( X_{t_k}^{\theta_0, \varepsilon}, \theta \right) \left(X_{t_{k+1}}^{\theta_0, \varepsilon} - X_{t_{k}}^{\theta_0, \varepsilon} - \frac{1}{n} b \left( X_{t_k}^{\theta_0, \varepsilon}, \theta \right)\right) - \frac{1}{n} \sum_{k=0}^{n-1} \partial_{\theta_i}b \left( X_{t_k}^{\theta_0, \varepsilon}, \theta \right) \partial_{\theta_j}b \left( X_{t_k}^{\theta_0, \varepsilon}, \theta \right) \\
&=\sum_{k=0}^{n-1} \partial_{\theta_j} \partial_{\theta_i}b \left( X_{t_k}^{\theta_0, \varepsilon}, \theta \right) \left(X_{t_{k+1}}^{\theta_0, \varepsilon} - X_{t_{k}}^{\theta_0, \varepsilon} - \frac{1}{n} b \left( X_{t_k}^{\theta_0, \varepsilon}, \theta_0 \right)\right)\\
&+\frac{1}{n} \sum_{k=0}^{n-1} \left\{ \partial_{\theta_i} \partial_{\theta_j} b\left( X_{t_k}^{\theta_0 ,\varepsilon}, \theta_0 \right) \left( b\left( X_{t_k}^{\theta_0 ,\varepsilon}, \theta_0 \right)-b\left( X_{t_k}^{\theta_0 ,\varepsilon}, \theta \right) \right) - \partial_{\theta_i} b\left( X_{t_k}^{\theta_0 ,\varepsilon}, \theta \right) \partial_{\theta_j} b\left( X_{t_k}^{\theta_0 ,\varepsilon}, \theta \right) \right\}\\
&=: K_{n,\varepsilon}^{(1)}(\theta) + K_{n,\varepsilon}^{(2)}(\theta).
\end{align*}
By using Lemma \ref{riemann}  and letting $f(x,\theta)=\partial_{\theta_i} \partial_{\theta_j} b\left(x, \theta_0 \right) \left( b\left(x, \theta_0 \right)-b\left( x, \theta \right) \right) - \partial_{\theta_i} b\left( x, \theta \right) \partial_{\theta_j} b\left( x, \theta \right)$, we have $\underset{\theta \in \Theta }{\sup} \left| K_{n,\varepsilon}^{(2)}(\theta)- K^{i,j}(\theta) \right| \rightarrow 0 \quad \text{a.s.} \quad (n \rightarrow \infty, \varepsilon \downarrow 0).$ By Lemma \ref{1}, Lemma \ref{stochastic integral1} and Lemma \ref{stochastic integral2}, we have
\begin{align*}
&\underset{\theta \in \Theta}{\sup} \left|K_{n,\varepsilon}^{(1)}(\theta) \right|\\
&\leq \underset{\theta \in \Theta}{\sup} \left| \sum_{k=0}^{n-1} \partial_{\theta_j} \partial_{\theta_i}b \left( X_{t_k}^{\theta_0, \varepsilon}, \theta \right)  \int_{t_k}^{t_{k+1}} \left( b\left(X_s^{\theta_0, \varepsilon}, \theta_0 \right) - b\left(X_{t_k}^{\theta_0, \varepsilon}, \theta_0 \right)\right) ds\right| + \underset{\theta \in \Theta}{\sup} \left| \varepsilon \sum_{k=0}^{n-1} \partial_{\theta_j} \partial_{\theta_i}b \left( X_{t_k}^{\theta_0, \varepsilon}, \theta \right) \left( W_{t_{k+1}}^H - W_{t_k}^H \right) \right|\\
&\rightarrow 0 \quad \text{a.s.} \quad (n\rightarrow \infty, \varepsilon \downarrow 0).
\end{align*}
\end{proof}

\section{Numerical results}
In this section, we simulate the results of Theorem \ref{main3}. Consider the fractional Ornstein-Uhlenbeck process:
\begin{eqnarray*}
\left\{
\begin{split}
dX_t^{\theta, \varepsilon}&=-\theta X_t^{\theta, \varepsilon}+\varepsilon dW_t^H \quad (0\leq t \leq T),\\
X_0^{\theta,\varepsilon}&=1,
\end{split}
\right.
\end{eqnarray*}
where $\theta$ is a positive constant.  On simulating (\ref{main}), we need to identify the distribution of $S(\theta)$. From Biagini et al. \cite{Biagini}, p.124, Young integral $\int_0^T x_t^{\theta_0} dW_t^H$ coincides with the symmetric integral $\int_0^T x_t^{\theta_0}  d\circ W_t^H$. Moreover, from Biagini et al. \cite{Biagini} et al., p.128 and p.130, the symmetric integral $\int_0^T x_t^{\theta_0}  d\circ W_t^H$ coincides with the Skorohod integral $\delta^{W_H} (x_{\cdot}^{\theta})$. Therefore, since  the Young integral $\int_0^T x_t^{\theta_0} dW_t^H$ coicides with the Skorohod integral $\delta^{W_H} (x_{\cdot}^{\theta})$, it is sufficient to consider the distribution of the Skorohod integral $\delta^{W_H} (x_{\cdot}^{\theta})$. $\delta^{W_H} (x_{\cdot}^{\theta})$ coincides with the It\^{o} integral $\int_0^T x_t^{\theta_0} dW_t$ by Nualart \cite{Nualart}, p.44 and p.288, so the distribution of the Young integral $\int_0^T x_t^{\theta_0} dW_t^H$  is the normal distribution with the meanthe $0$ and the variance $\mathbb{E}\left[ \left(\int_0^T x_t^{\theta_0} dW_t^H \right)^2  \right]$.

In each experiment, we generate a discrete sample $(X_{t_k}^{\theta_0, \varepsilon})_{k=1}^{n}$ by using the Euler scheme (see Nuenkirch and Nourdin \cite{Nuenkirch2}) and compute $\hat{\theta}_{n,\varepsilon}$ from sample by Newton method. This procedure is iterated $1000$ times, and the mean and the standard deviation of $1000$ sampled estimators are computed in each case of $(H, \theta_0, n, \varepsilon)$ to confirm the strong consistency in Theorem \ref{consistency}. We also confirmed the asymptotic normality of Theorem \ref{main3} by creating the Normal Q-Q plot and the histogram using $\frac{I(\theta_0)\left(\hat{\theta}_{n,\varepsilon} -\theta_0 \right)}{\varepsilon \sqrt{\mathbb{E}\left[ \left(\int_0^T x_t^{\theta_0} dW_t^H \right)^2  \right]}}$. Here, $\mathbb{E}\left[ \left(\int_0^T x_t^{\theta_0} dW_t^H \right)^2  \right]$ is computed by a Monte Carlo simulation. Since the balance of convergence speed between $\varepsilon$ and $n$ is required differently depending on the value of the Hurst index$H$, the setting of the value of $(\varepsilon,n)$ is considered separately for cases when the Hurst index $H$ is greater than $\frac{1}{2}$ and less than $\frac{1}{2}$. In the case of $H>\frac{1}{2}$, we need $n\varepsilon \rightarrow 0$, so we adapt $\varepsilon=1.0, 0.5, 0.1$ and $n=100, 500,1000$. On the other hand, for $H\leq \frac{1}{2}$, the convergence order of $\varepsilon$ must be greater than $\frac{1}{n^{1-H}}$ and less than $\frac{1}{n}$ because $n^{1-H}\varepsilon \rightarrow 0$ and $n\varepsilon \rightarrow \infty$ are required. Thus, we adopt $\varepsilon=0.01, 0.005, 0.001$ and $n=100, 500, 1000$ for $H=0.25$. In addition, in order to check the behavior in the case where the assumption of balance of the convergence order for the asymptotic normality in Theorem \ref{main3} is not satisfied, we comute $\hat{\theta}_{n,\varepsilon}$ when $\varepsilon = 0.1$ and create the Normal Q-Q plot and the Histgram for $H=0.25$. The results of those are shown in Table \ref{table1}-\ref{table2} and Figure \ref{fig1}-\ref{fig20}.

\begin{table}[h]
\caption{Mean and standard devitation of the estimator $\hat{\theta}_{n,\varepsilon}$ through 1000 experiments in the case $H=0.75, \theta_0=1.0 T=1.0$}.
\label{table1}
\centering
\begin{tabular}{ccccc}
\hline
$ $ & $n=100$ & $n=500$& $n=1000$ \\
\hline
$\hat{\theta}_{n,\varepsilon=1.0}$ (s.d.)&1.02420(1.04458)&0.96179(1.04415)&0.92395(1.02431)\\

$\hat{\theta}_{n,\varepsilon=0.5}$ (s.d.)&1.08895(0.66375)&1.02890(0.66751)&1.08095(0.66120)\\

$\hat{\theta}_{n,\varepsilon=0.1}$ (s.d.)&1.00828(0.14035)&1.00012(0.15209)&1.00534(0.15335)\\
\hline
\end{tabular}
\end{table}

\begin{table}[h]
\caption{Mean and standard devitation of the estimator $\hat{\theta}_{n,\varepsilon}$ through 1000 experiments in the case  $H=0.25, \theta_0=1.0, T=1.0$.}
\label{table2}
\centering
\begin{tabular}{ccccc}
\hline
$ $ & $n=100$ & $n=500$& $n=1000$ \\
\hline
$\hat{\theta}_{n,\varepsilon=0.1}$ (s.d.)&1.12705(0.17712)&1.25678(0.19022)&1.35958(0.20237)\\

$\hat{\theta}_{n,\varepsilon=0.01}$ (s.d.)&1.00163(0.01630)&1.0022(0.01573)&1.00425(0.01598)\\

$\hat{\theta}_{n,\varepsilon=0.005}$ (s.d.)&1.00061(0.00829)&1.00069(0.00801)&1.00109(0.00773)\\

$\hat{\theta}_{n,\varepsilon=0.001}$ (s.d.)&1.00002(0.00154)&1.00004(0.00168)&1.00003(0.00156)\\
\hline
\end{tabular}
\end{table}

\begin{figure}[H]
  \begin{minipage}[b]{0.45 \linewidth}
    \centering
    \includegraphics[keepaspectratio, scale=0.5]{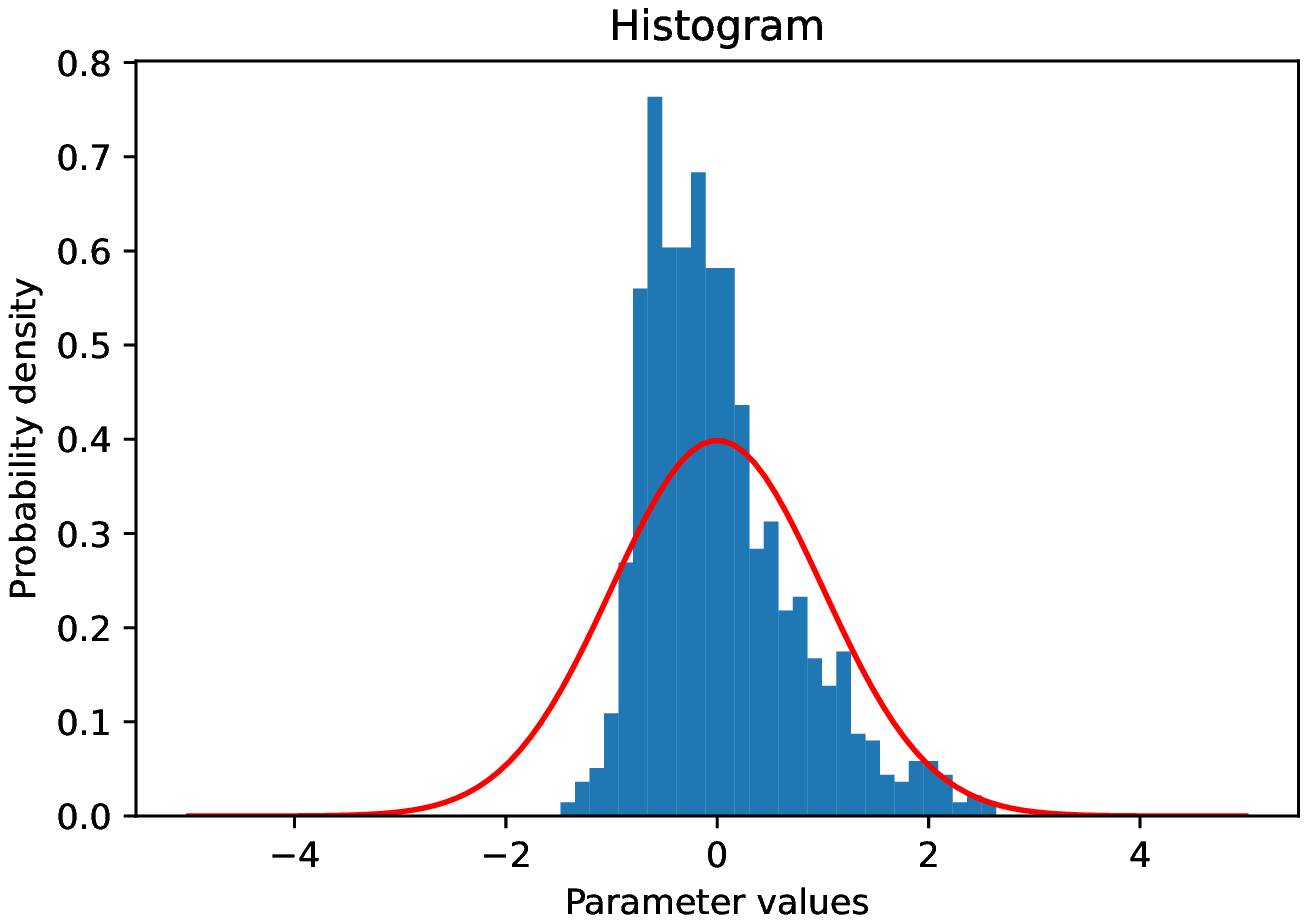}
  \end{minipage}
  \begin{minipage}[b]{0.45\linewidth}
    \centering
    \includegraphics[keepaspectratio, scale=0.5]{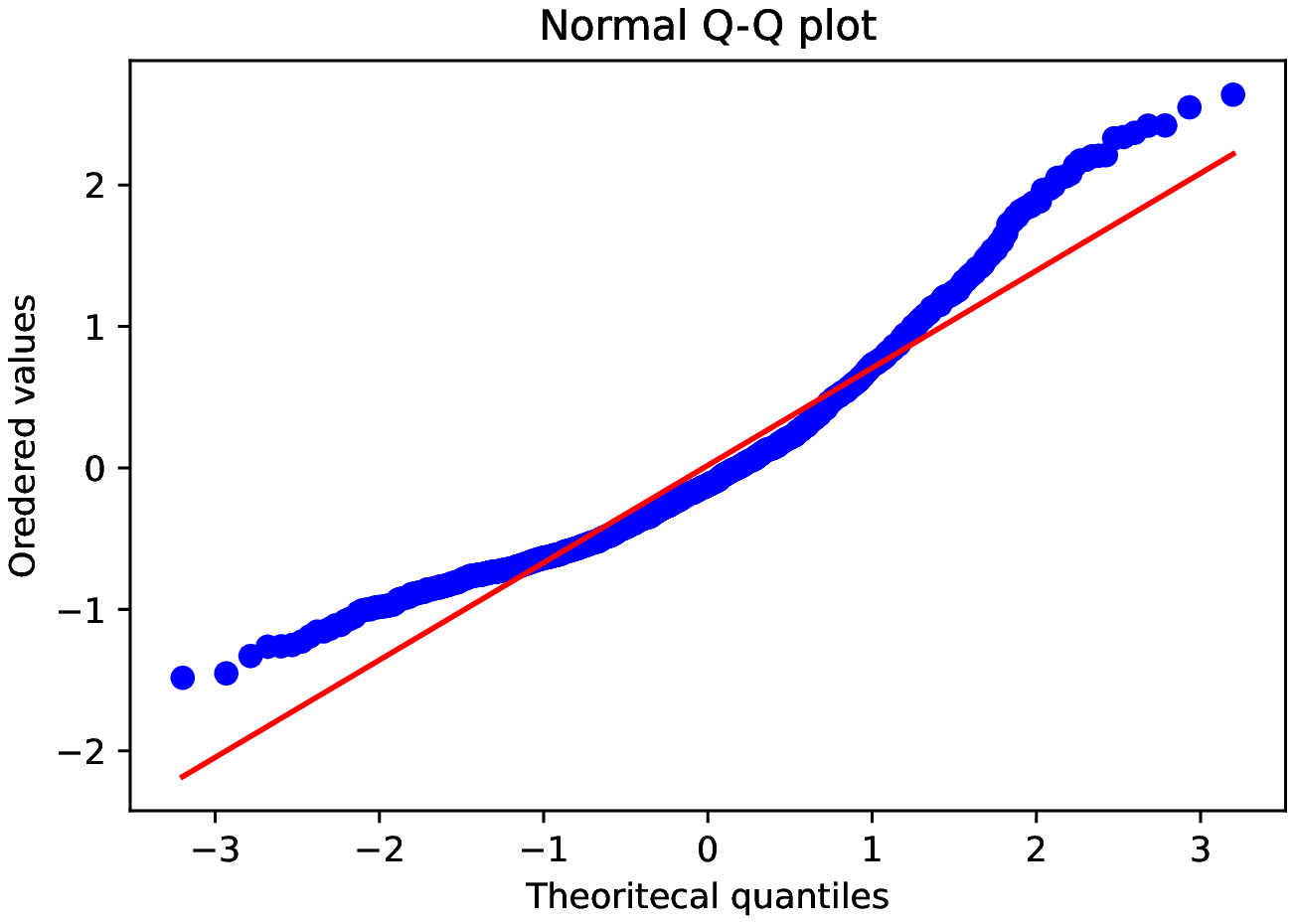}
  \end{minipage}
    \caption{Histogram and density function of a standard normal distribution (left) and Normal Q-Q plot (right) through $1000$ experiments in the case $H=0.75,\theta=1.0,\varepsilon=1.0, T=1.0, n=100$}
  \label{fig1}
\end{figure}

\begin{figure}[H]
  \begin{minipage}[t]{0.45\linewidth}
    \centering
    \includegraphics[keepaspectratio, scale=0.5]{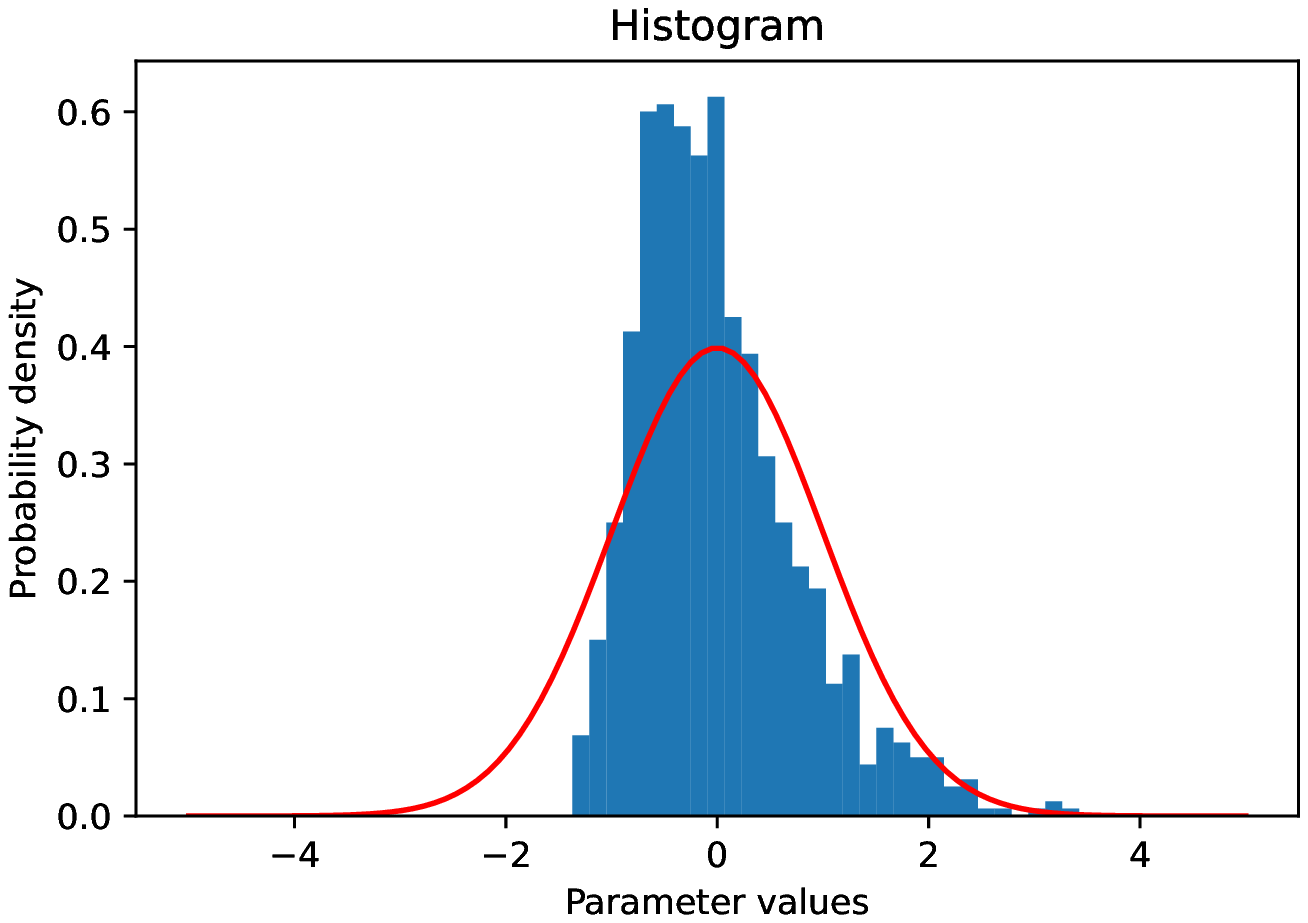}
 
  \end{minipage}
  \begin{minipage}[t]{0.45\linewidth}
    \centering
    \includegraphics[keepaspectratio, scale=0.5]{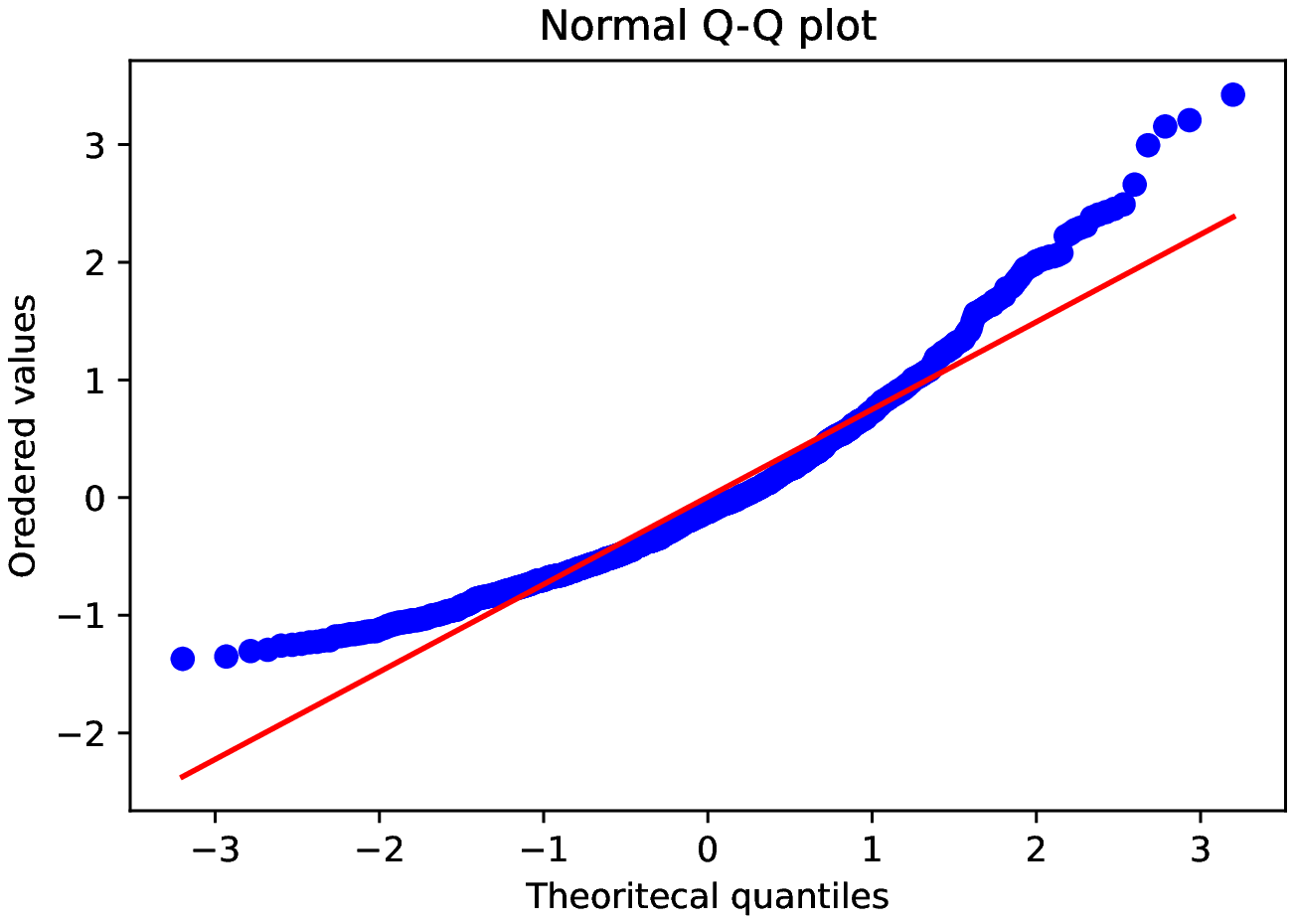}
  
  \end{minipage}
  \caption{Histogram and density function of a standard normal distribution (left) and Normal Q-Q plot (right) through $1000$ experiments in the case $H=0.75,\theta=1.0,\varepsilon=1.0, T=1.0, n=500$}
  \label{fig2}
\end{figure}

\begin{figure}[htbp]

  \begin{minipage}[b]{0.45\linewidth}
    \centering
    \includegraphics[keepaspectratio, scale=0.5]{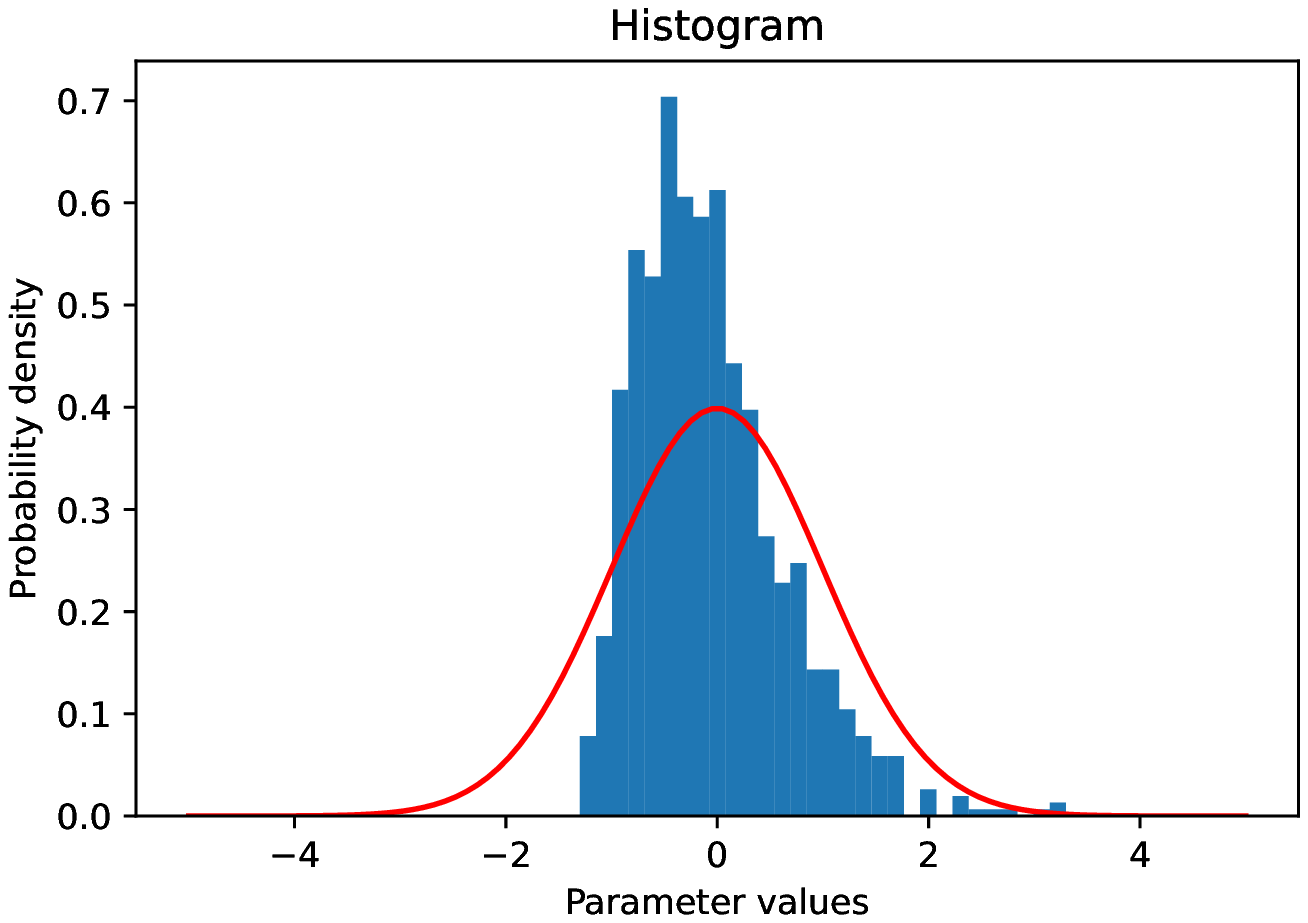}
   
  \end{minipage}
  \begin{minipage}[b]{0.45\linewidth}
    \centering
    \includegraphics[keepaspectratio, scale=0.5]{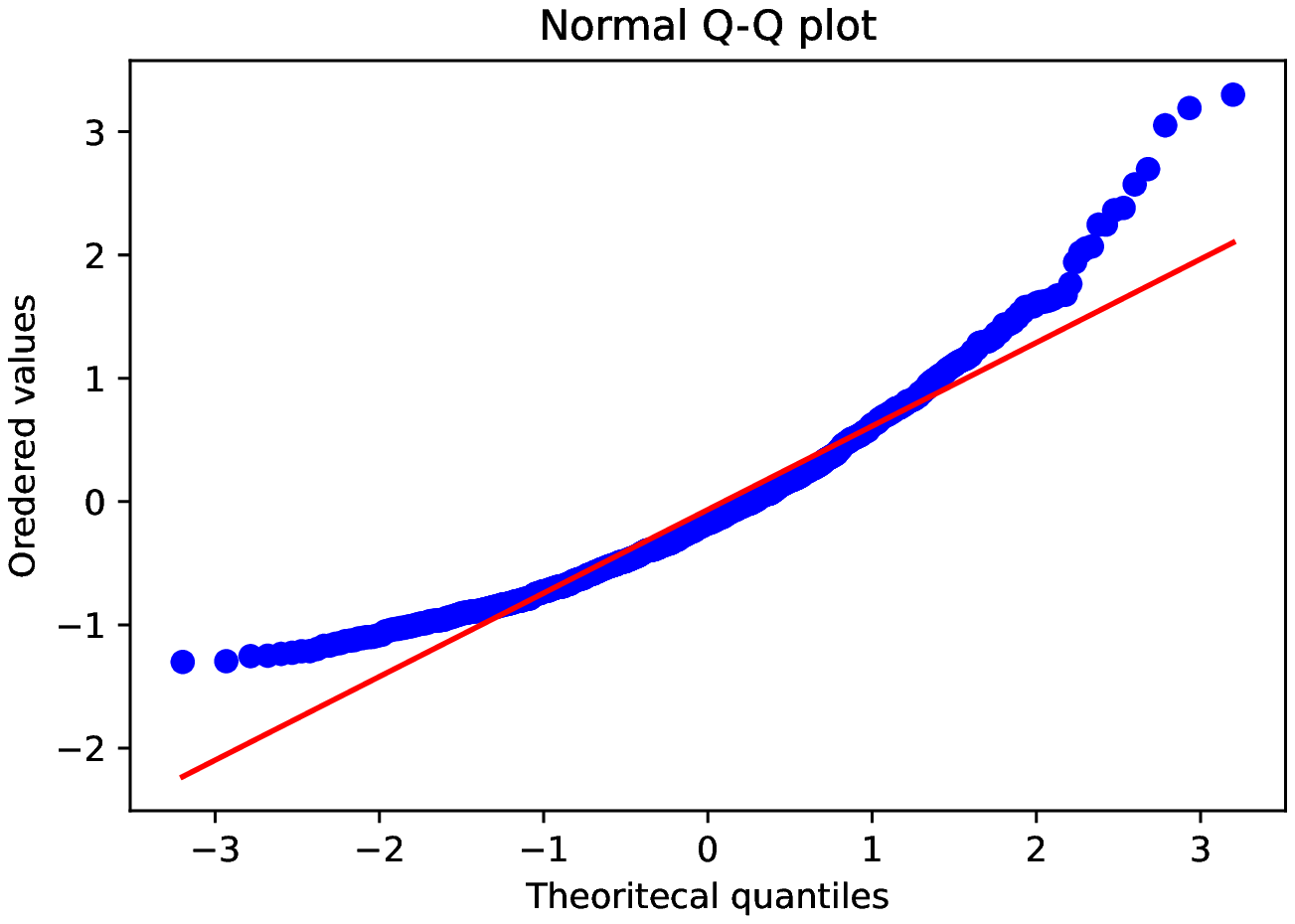}
   
  \end{minipage}
  \caption{Histogram and density function of a standard normal distribution (left) and Normal Q-Q plot (right) through $1000$ experiments in the case $H=0.75,\theta=1.0,\varepsilon=1.0, T=1.0, n=1000$}
  \label{fig3}
\end{figure}

\begin{figure}[htbp]
  \begin{minipage}[b]{0.45\linewidth}
    \centering
    \includegraphics[keepaspectratio, scale=0.5]{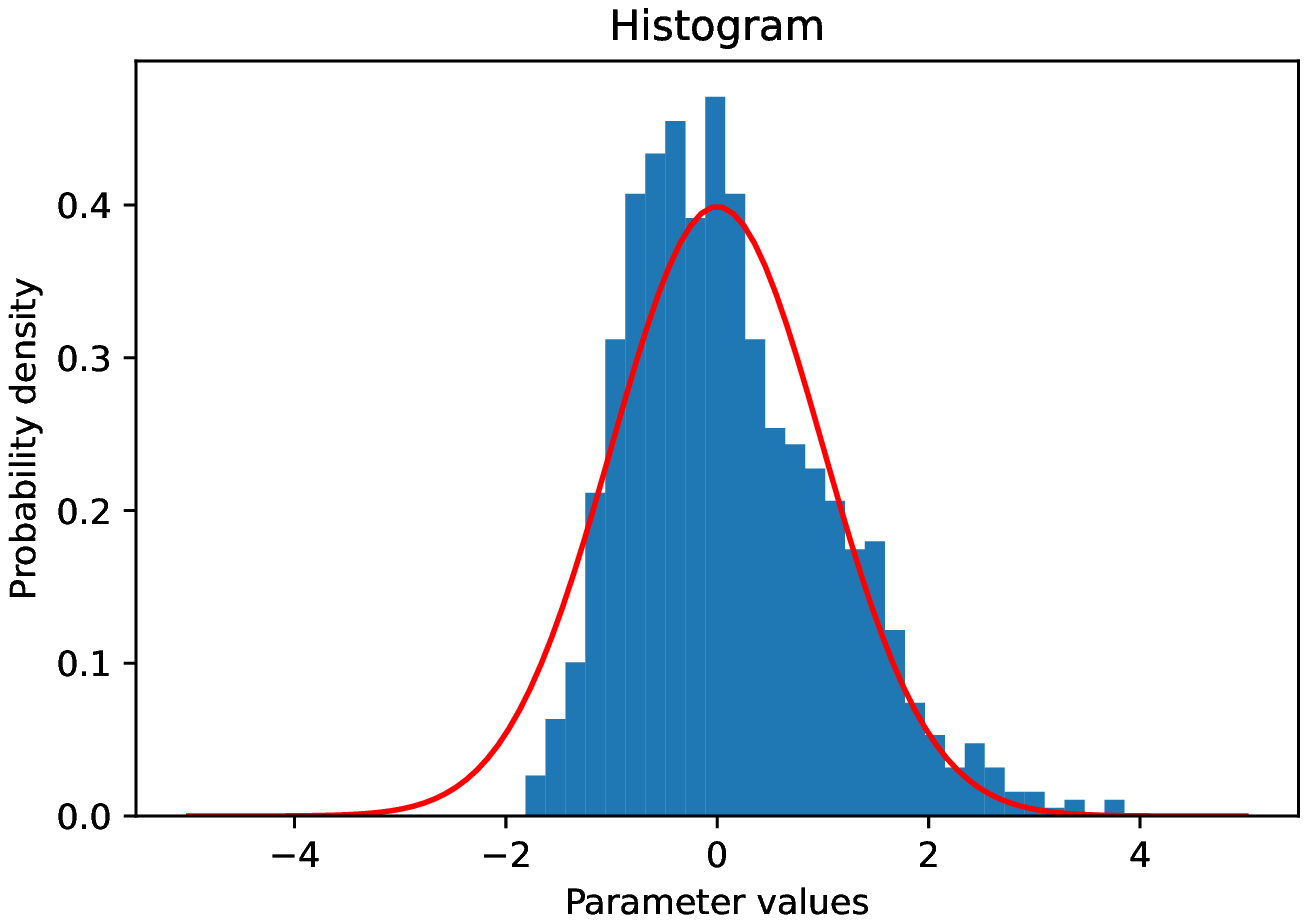}
   
  \end{minipage}
  \begin{minipage}[b]{0.45\linewidth}
    \centering
    \includegraphics[keepaspectratio, scale=0.5]{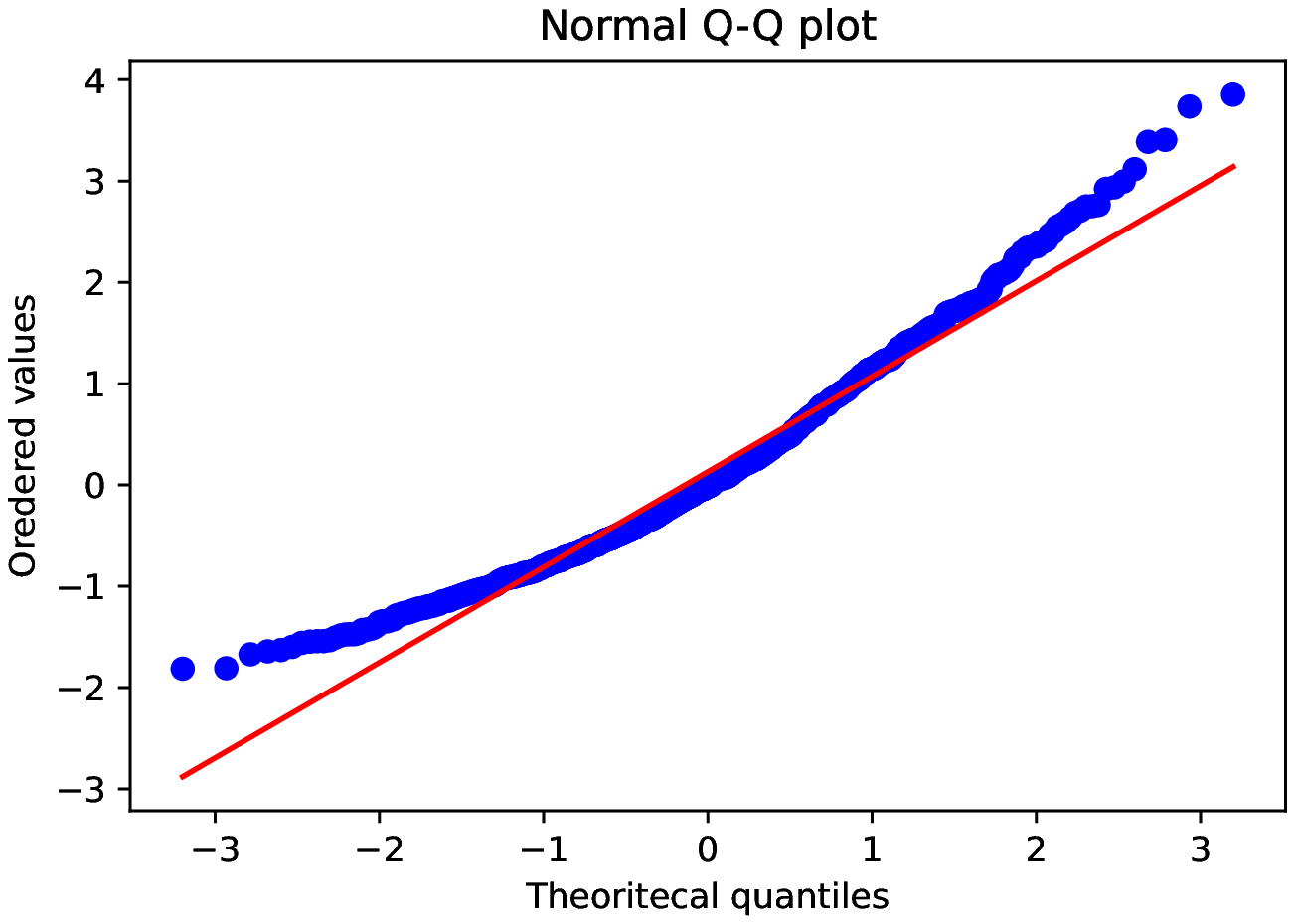}
   
  \end{minipage}
  \caption{Histogram and density function of a standard normal distribution (left) and Normal Q-Q plot (right) through $1000$ experiments in the case $H=0.75,\theta=1.0,\varepsilon=0.5, T=1.0, n=100$}
  \label{fig4}
\end{figure}

\begin{figure}[htbp]
  \begin{minipage}[b]{0.45\linewidth}
    \centering
    \includegraphics[keepaspectratio, scale=0.5]{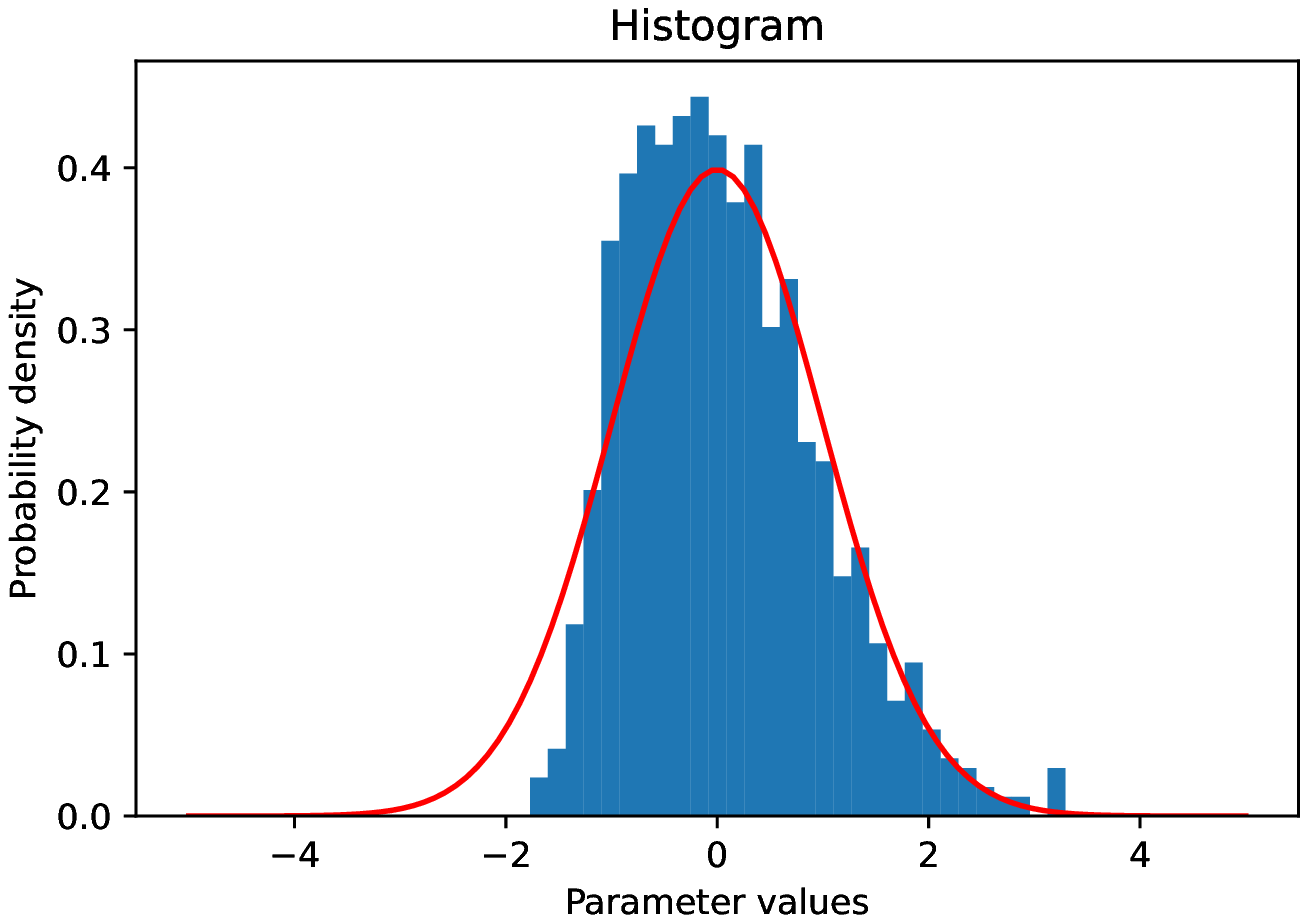}
   
  \end{minipage}
  \begin{minipage}[b]{0.45\linewidth}
    \centering
    \includegraphics[keepaspectratio, scale=0.5]{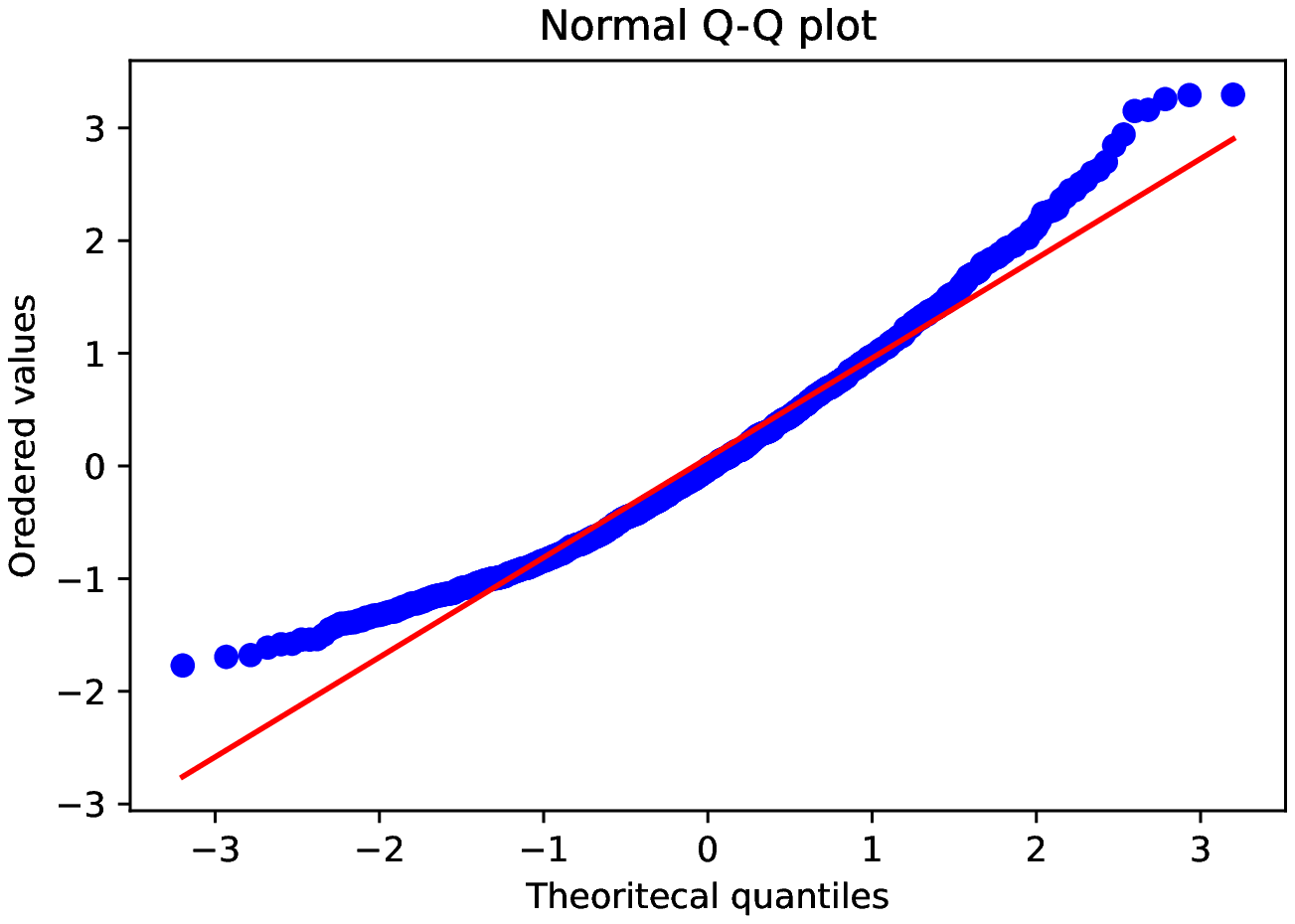}
   
  \end{minipage}
  \caption{Histogram and density function of a standard normal distribution (left) and Normal Q-Q plot (right) through $1000$ experiments in the case $H=0.75,\theta=1.0, \varepsilon=0.5, T=1.0, n=500$}
  \label{fig5}
\end{figure}

\begin{figure}[htbp]
  \begin{minipage}[b]{0.45\linewidth}
    \centering
    \includegraphics[keepaspectratio, scale=0.5]{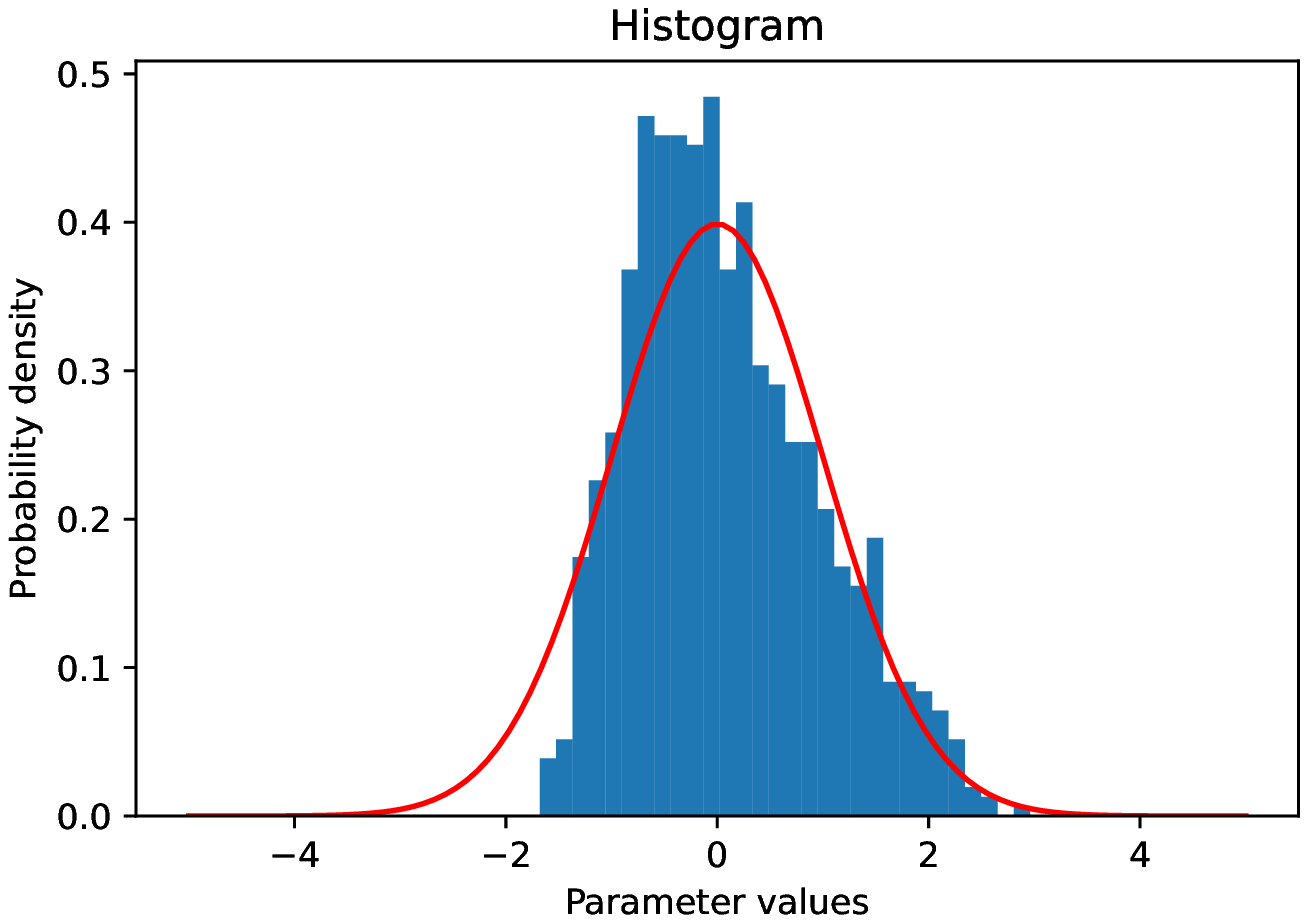}
   
  \end{minipage}
  \begin{minipage}[b]{0.45\linewidth}
    \centering
    \includegraphics[keepaspectratio, scale=0.5]{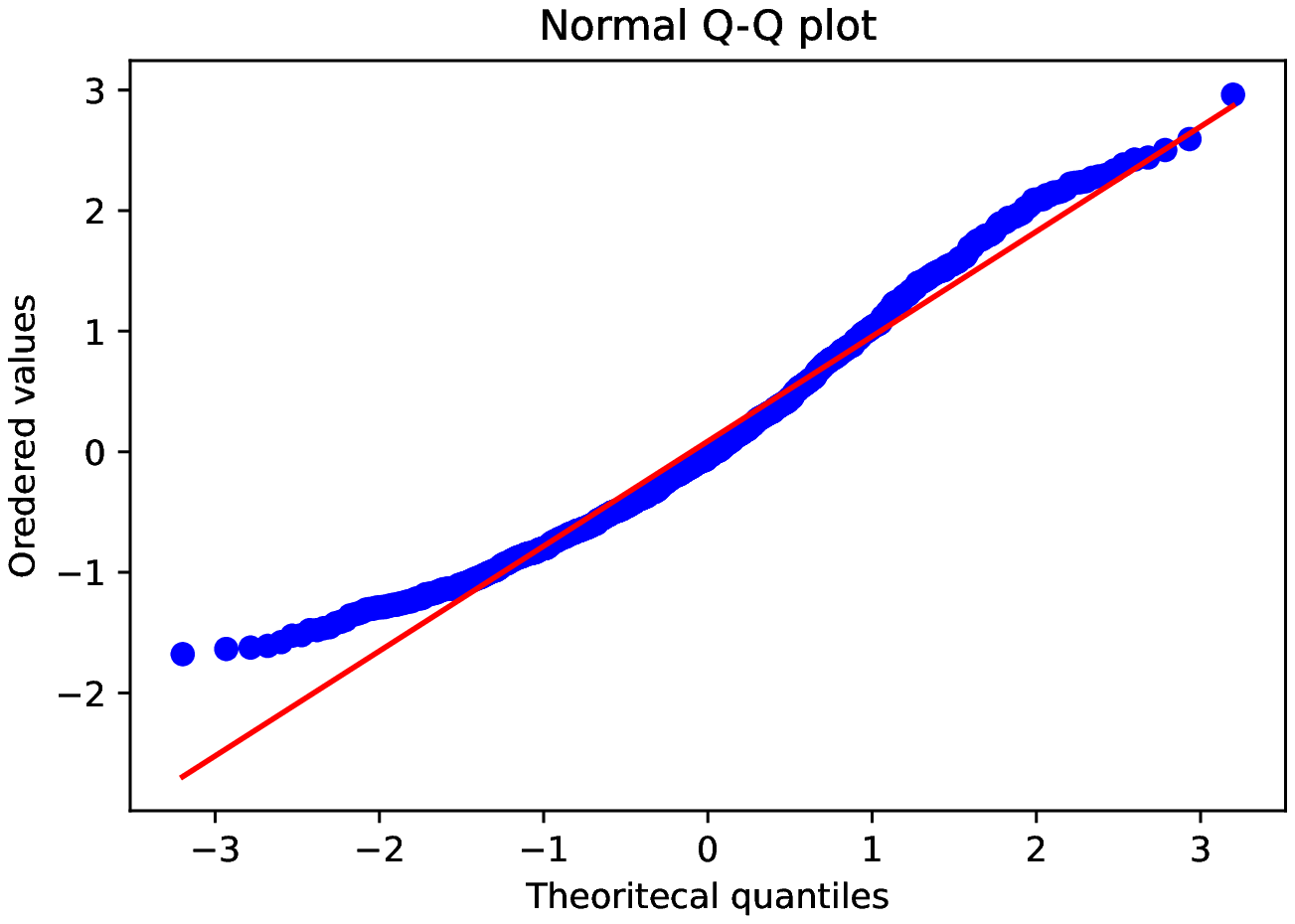}
   
  \end{minipage}
  \caption{Histogram and density function of a standard normal distribution (left) and Normal Q-Q plot (right) through $1000$ experiments in the case $H=0.75,\theta=1.0,\varepsilon=0.5, T=1.0, n=1000$}
  \label{fig6}
\end{figure}

\begin{figure}[htbp]
  \begin{minipage}[b]{0.45\linewidth}
    \centering
    \includegraphics[keepaspectratio, scale=0.5]{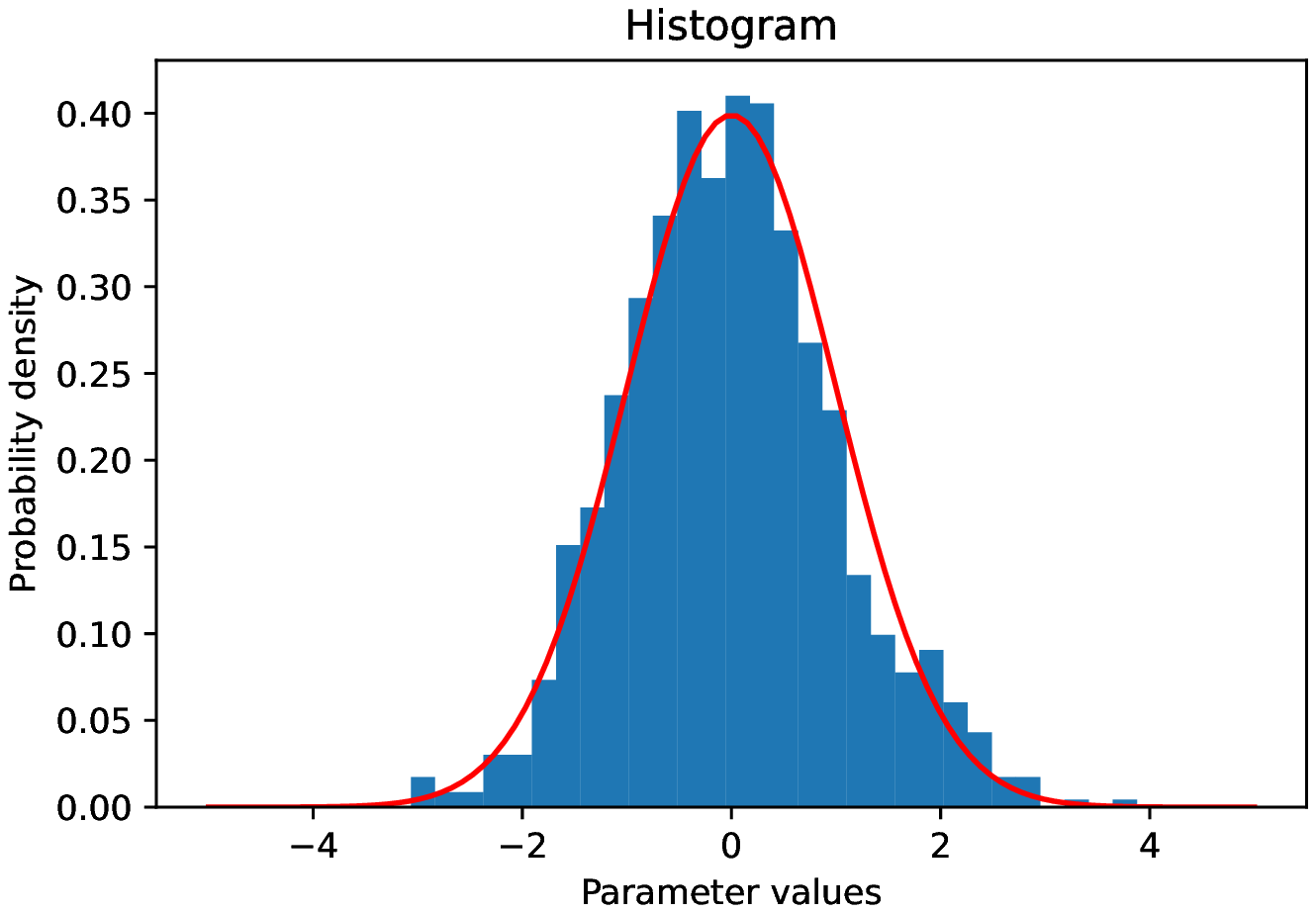}
   
  \end{minipage}
  \begin{minipage}[b]{0.45\linewidth}
    \centering
    \includegraphics[keepaspectratio, scale=0.5]{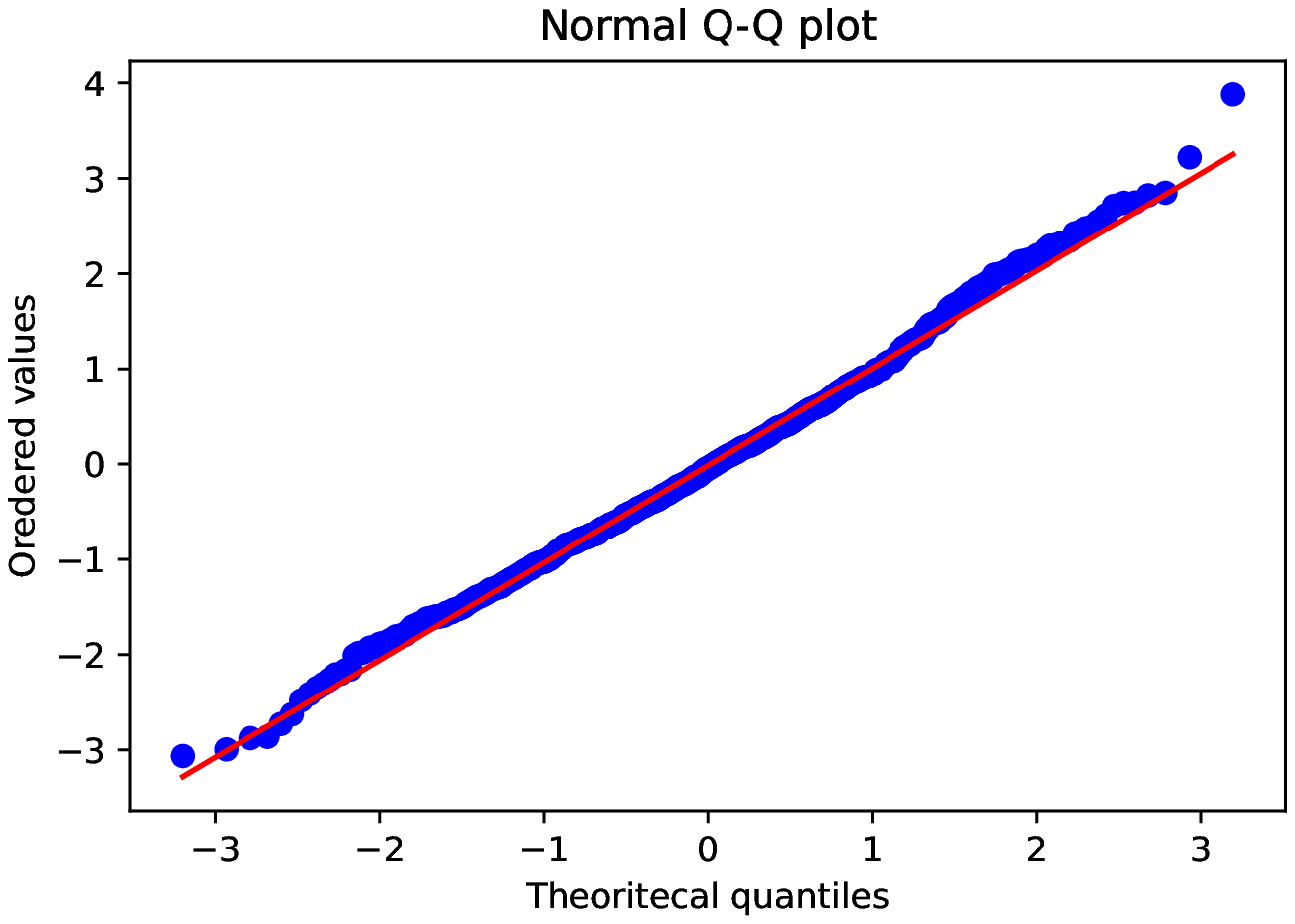}
   
  \end{minipage}
  \caption{Histogram and density function of a standard normal distribution (left) and Normal Q-Q plot (right) through $1000$ experiments in the case $H=0.75,\theta=1.0,\varepsilon=0.1, T=1.0, n=100$}
  \label{fig7}
\end{figure}

\begin{figure}[htbp]
  \begin{minipage}[b]{0.45\linewidth}
    \centering
    \includegraphics[keepaspectratio, scale=0.5]{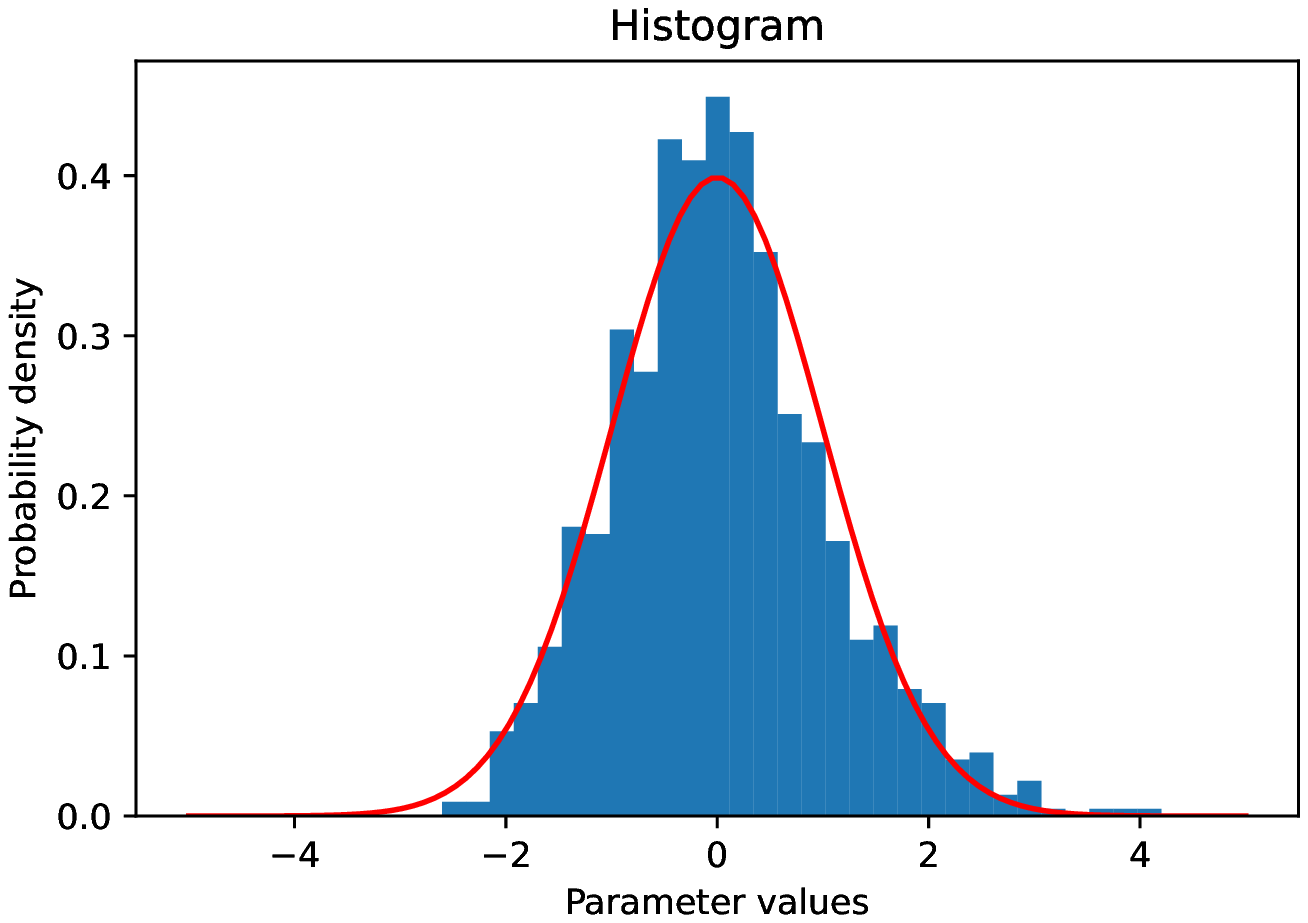}
   
  \end{minipage}
  \begin{minipage}[b]{0.45\linewidth}
    \centering
    \includegraphics[keepaspectratio, scale=0.5]{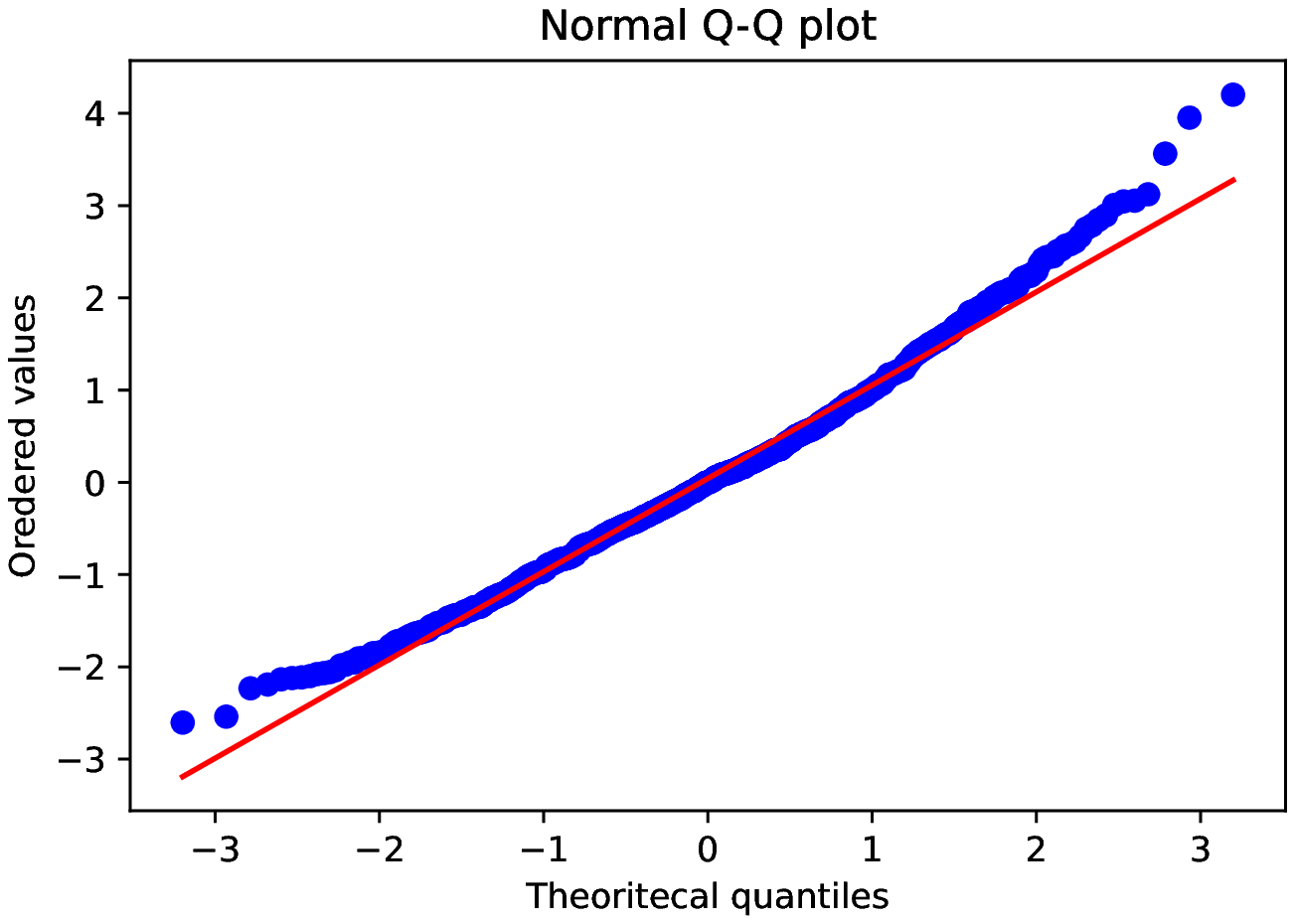}
   
  \end{minipage}
  \caption{Histogram and density function of a standard normal distribution (left) and Normal Q-Q plot (right) through $1000$ experiments in the case $H=0.75,\theta=1.0,\varepsilon=0.1, T=1.0, n=500$}
  \label{fig8}
\end{figure}

\begin{figure}[htbp]
  \begin{minipage}[b]{0.45\linewidth}
    \centering
    \includegraphics[keepaspectratio, scale=0.5]{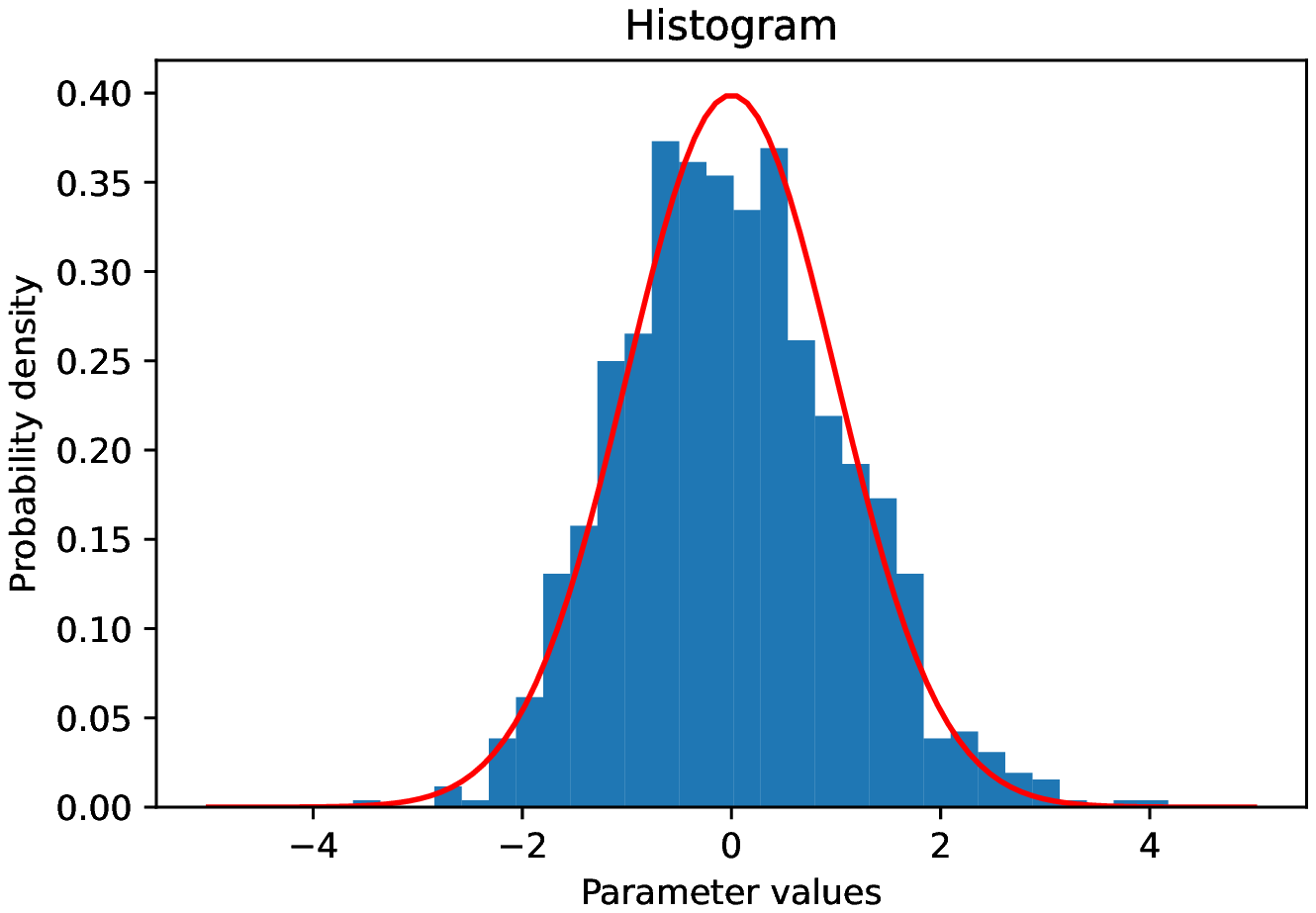}
   
  \end{minipage}
  \begin{minipage}[b]{0.45\linewidth}
    \centering
    \includegraphics[keepaspectratio, scale=0.5]{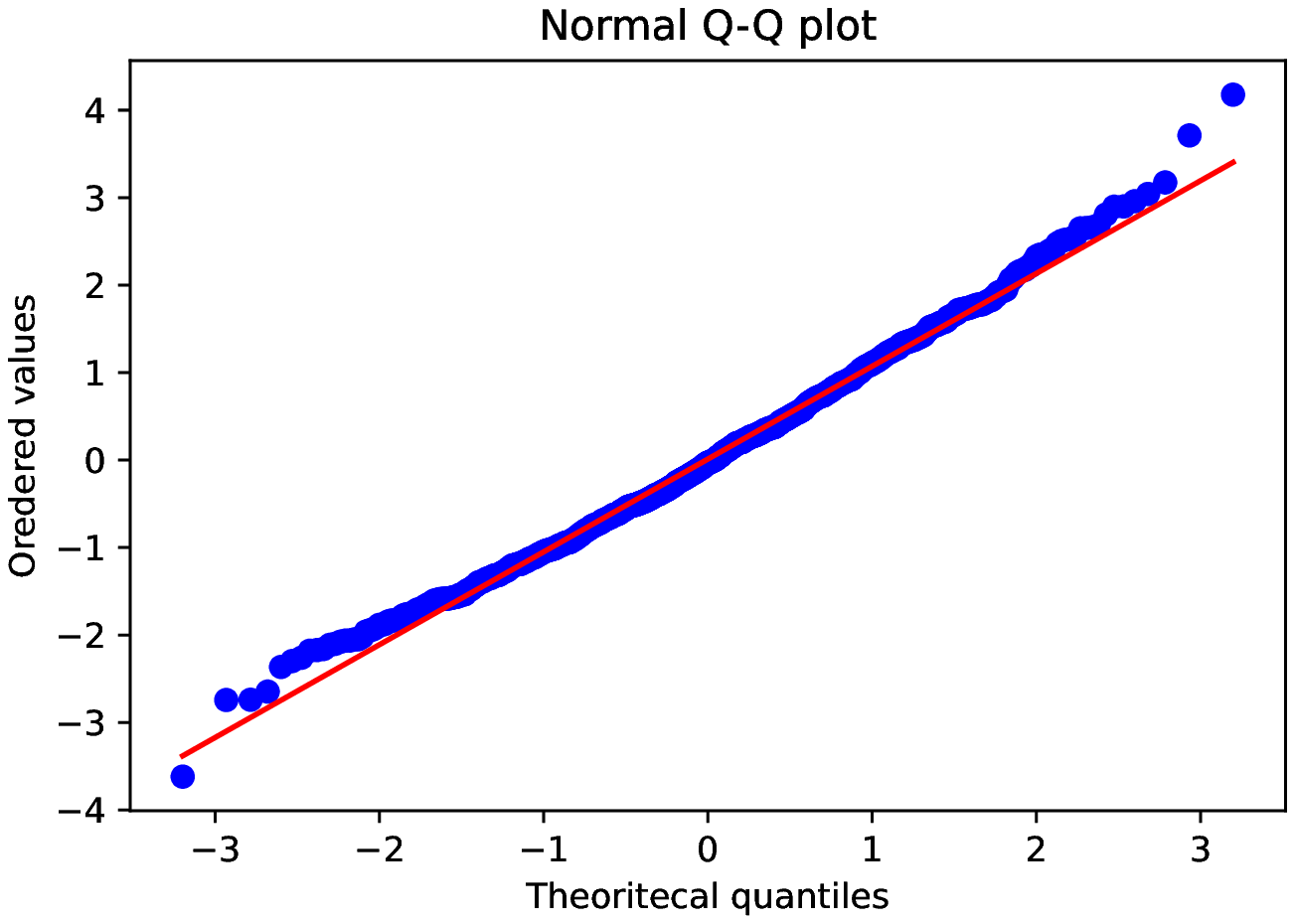}
   
  \end{minipage}
  \caption{Histogram and density function of a standard normal distribution (left) and Normal Q-Q plot (right) through $1000$ experiments in the case $H=0.75,\theta=1.0,\varepsilon=0.1, T=1.0, n=1000$}
  \label{fig9}
\end{figure}

\begin{figure}[htbp]
  \begin{minipage}[b]{0.45\linewidth}
    \centering
    \includegraphics[keepaspectratio, scale=0.5]{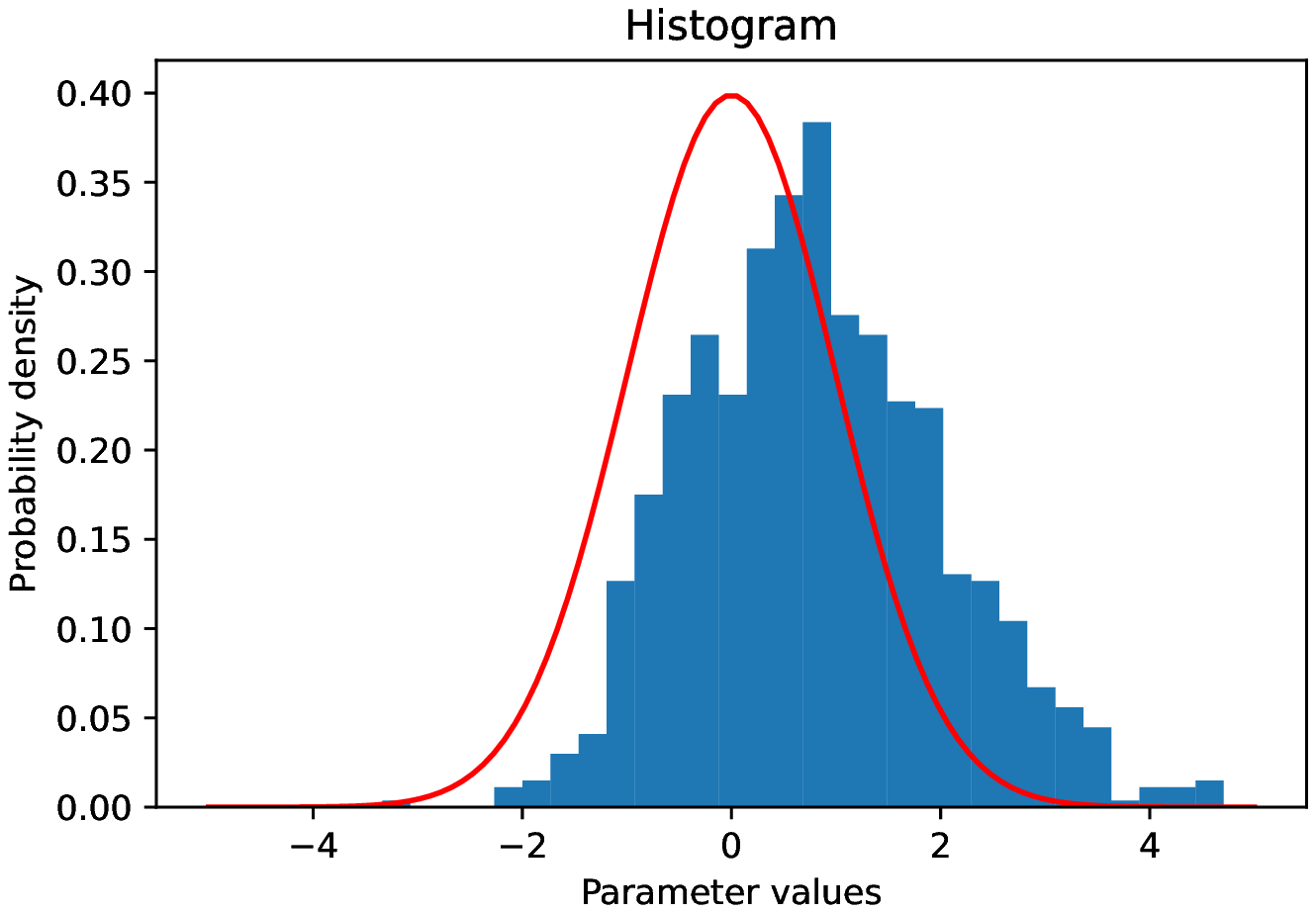}
   
  \end{minipage}
  \begin{minipage}[b]{0.45\linewidth}
    \centering
    \includegraphics[keepaspectratio, scale=0.5]{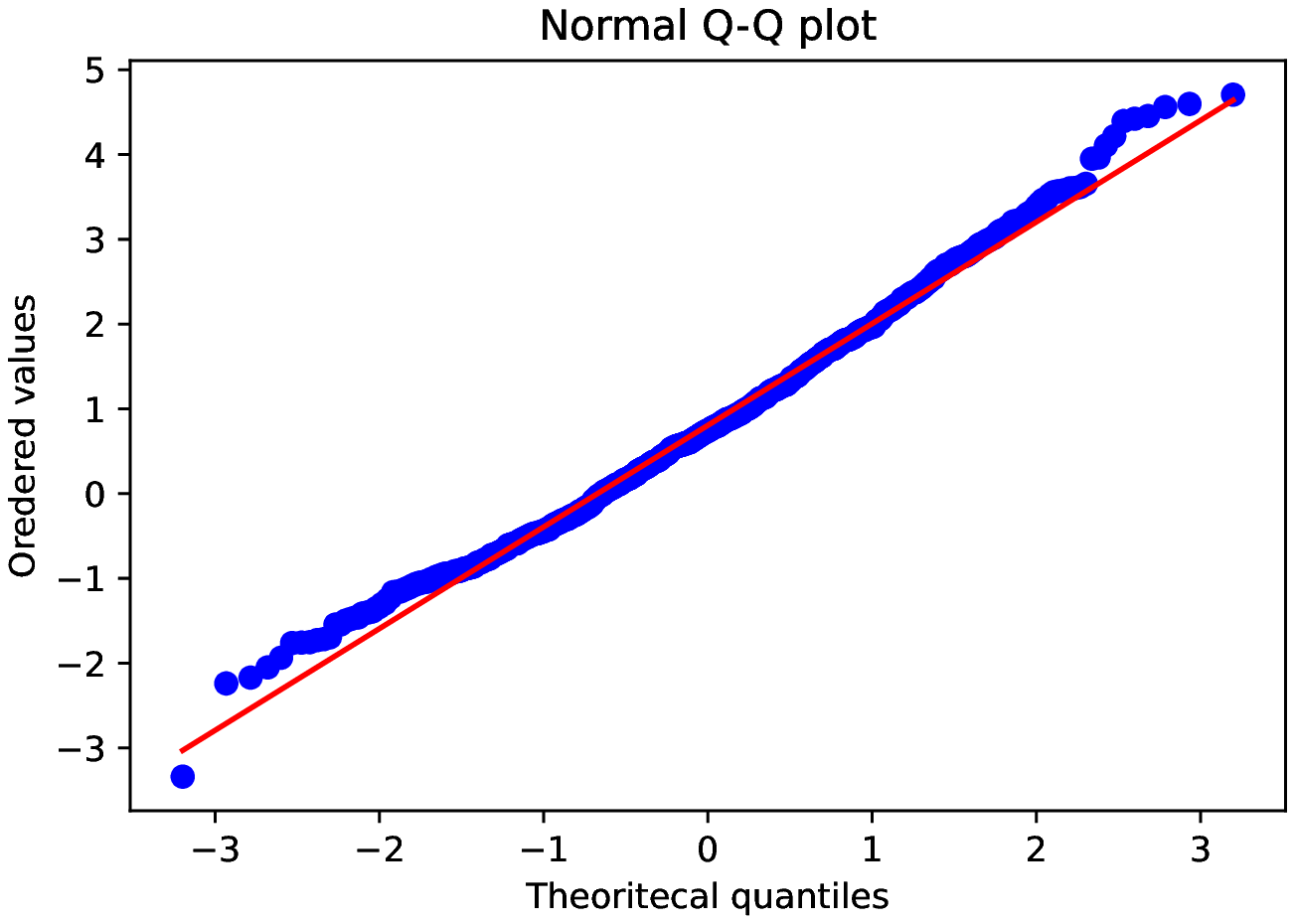}
   
  \end{minipage}
  \caption{Histogram and density function of a standard normal distribution (left) and Normal Q-Q plot (right) through $1000$ experiments in the case $H=0.25,\theta=1.0,\varepsilon=0.1, T=1.0, n=100$}
  \label{fig10}
\end{figure}

\begin{figure}[htbp]
  \begin{minipage}[b]{0.45\linewidth}
    \centering
    \includegraphics[keepaspectratio, scale=0.5]{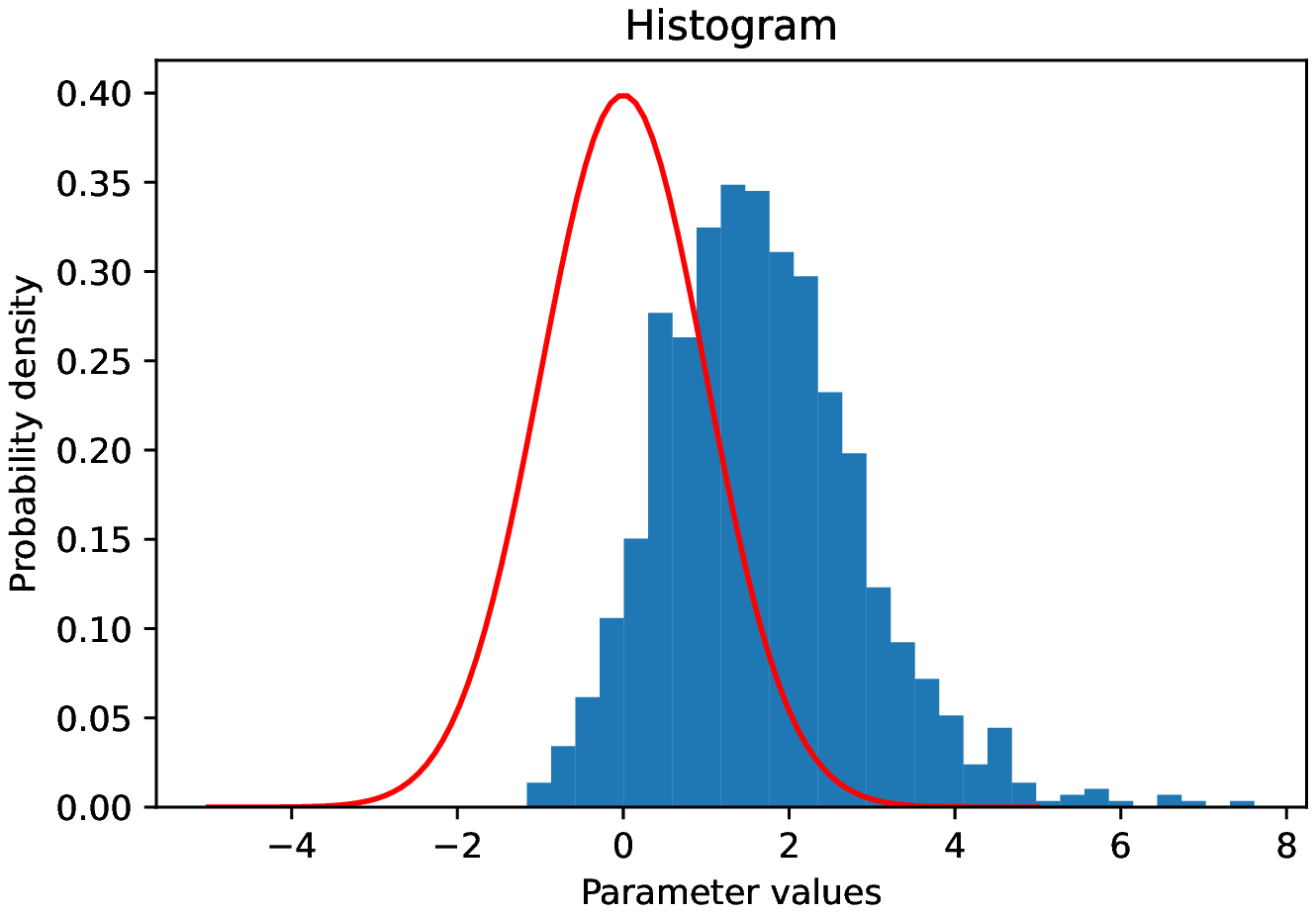}
   
  \end{minipage}
  \begin{minipage}[b]{0.45\linewidth}
    \centering
    \includegraphics[keepaspectratio, scale=0.5]{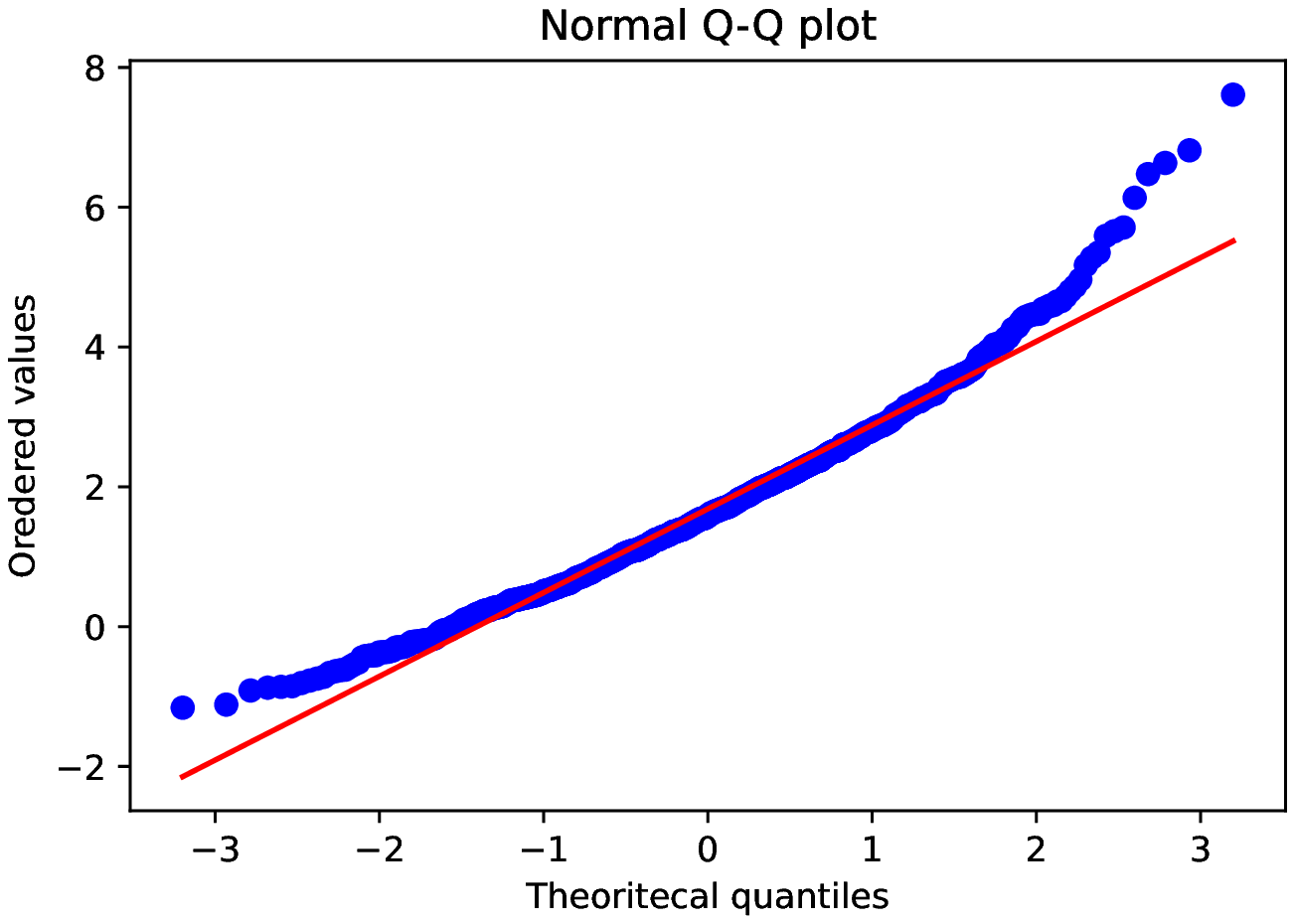}
   
  \end{minipage}
  \caption{Histogram and density function of a standard normal distribution (left) and Normal Q-Q plot (right) through $1000$ experiments in the case  $H=0.25,\theta=1.0,\varepsilon=0.1, T=1.0, n=500$}
  \label{fig11}
\end{figure}

\begin{figure}[htbp]
  \begin{minipage}[b]{0.45\linewidth}
    \centering
    \includegraphics[keepaspectratio, scale=0.5]{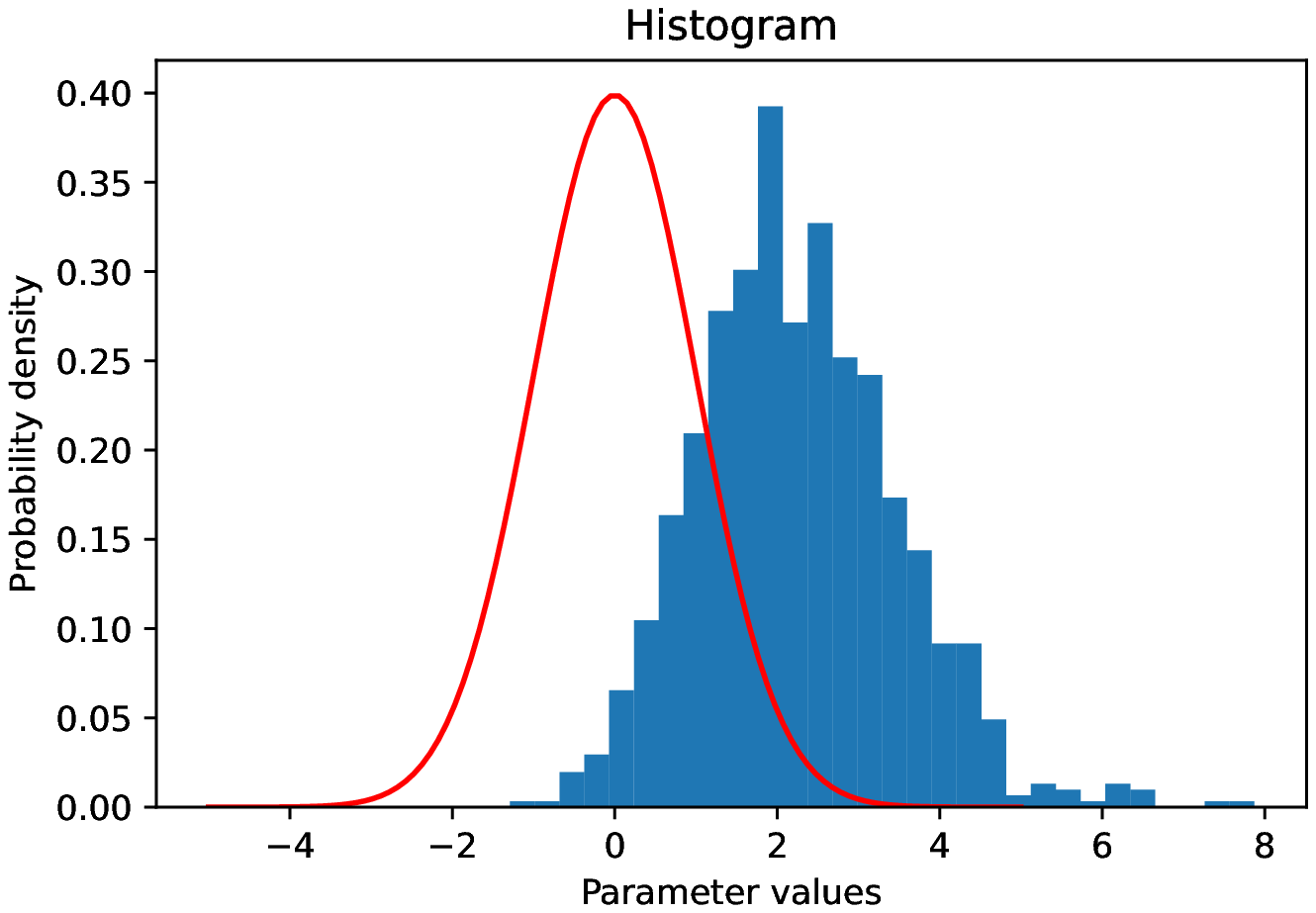}
   
  \end{minipage}
  \begin{minipage}[b]{0.45\linewidth}
    \centering
    \includegraphics[keepaspectratio, scale=0.5]{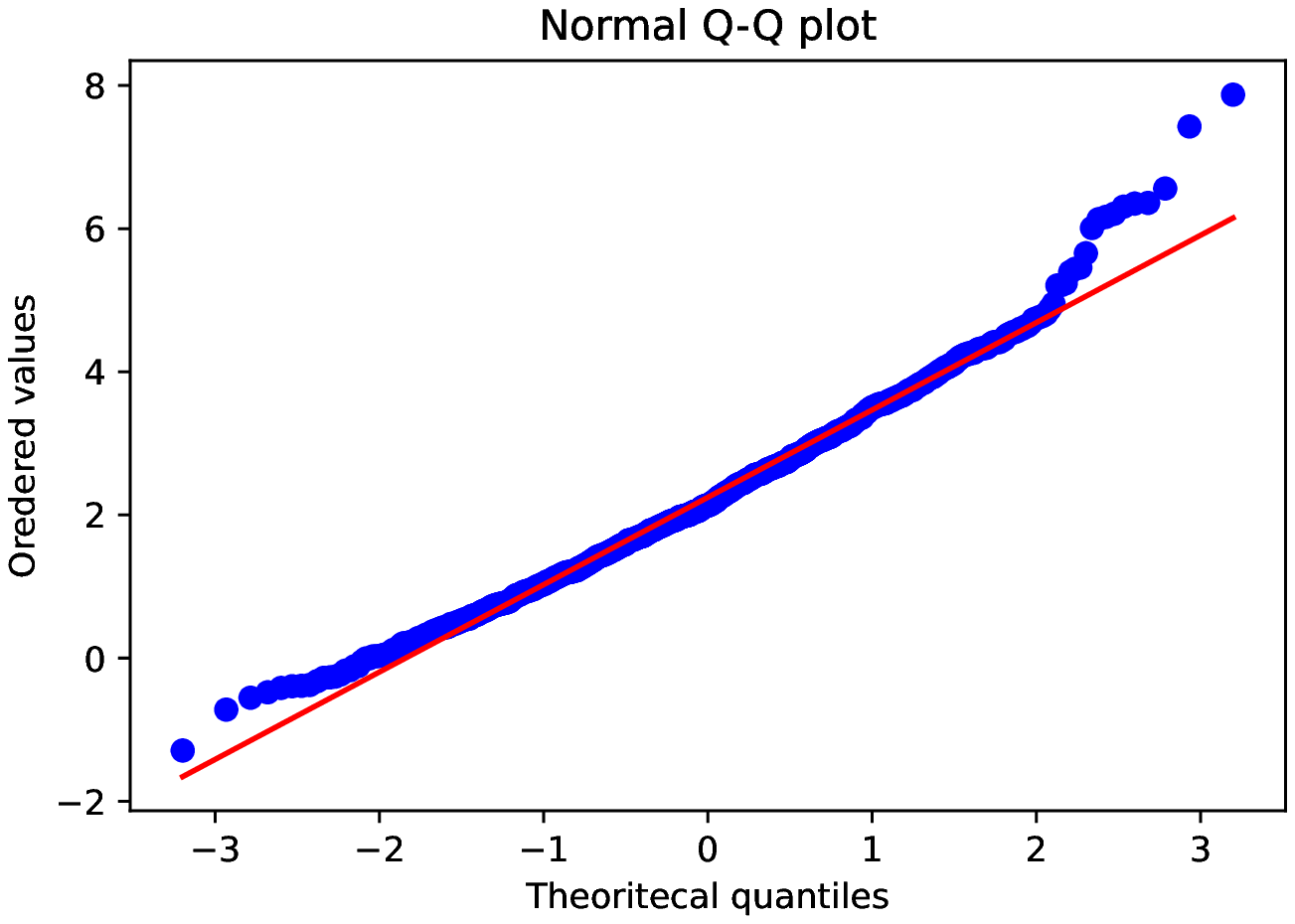}
   
  \end{minipage}
  \caption{Histogram and density function of a standard normal distribution (left) and Normal Q-Q plot (right) through $1000$ experiments in the case  $H=0.25,\theta=1.0,\varepsilon=0.1, T=1.0, n=1000$}
  \label{fig12}
\end{figure}

\begin{figure}[htbp]
  \begin{minipage}[b]{0.45\linewidth}
    \centering
    \includegraphics[keepaspectratio, scale=0.5]{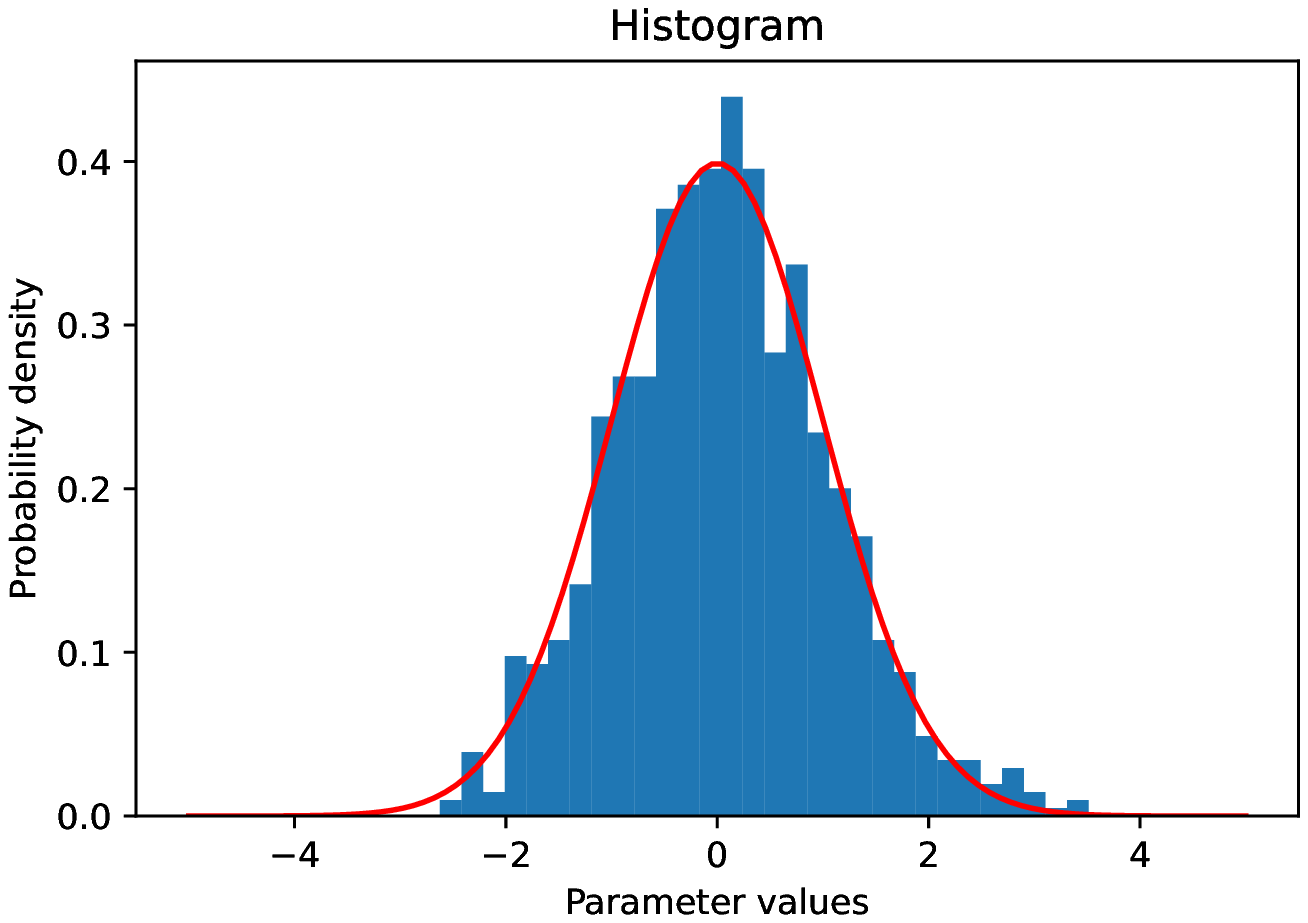}
   
  \end{minipage}
  \begin{minipage}[b]{0.45\linewidth}
    \centering
    \includegraphics[keepaspectratio, scale=0.5]{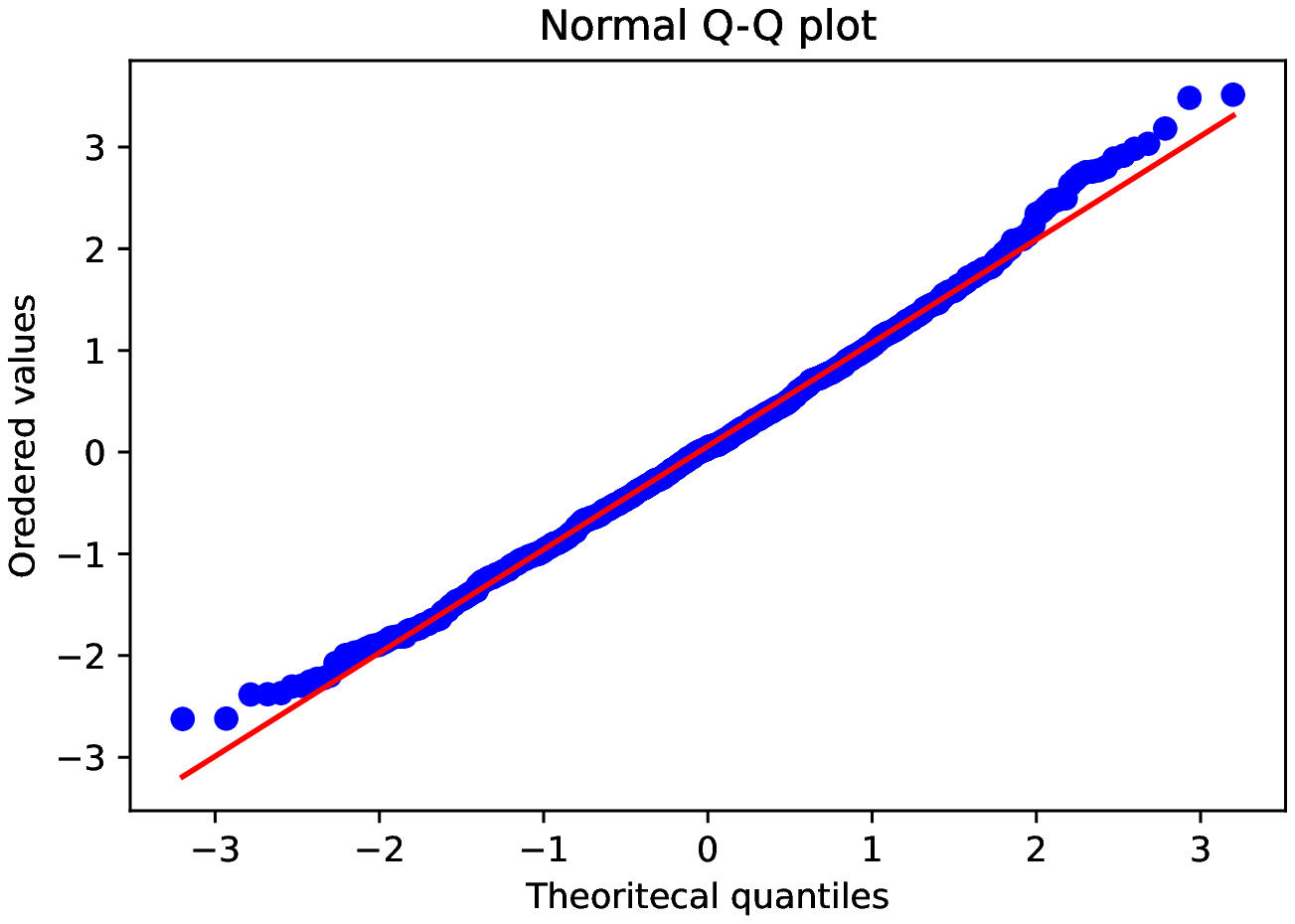}
   
  \end{minipage}
  \caption{Histogram and density function of a standard normal distribution (left) and Normal Q-Q plot (right) through $1000$ experiments in the case $H=0.25,\theta=1.0,\varepsilon=0.01, T=1.0, n=100$}
  \label{fig13}
\end{figure}

\begin{figure}[htbp]
  \begin{minipage}[b]{0.45\linewidth}
    \centering
    \includegraphics[keepaspectratio, scale=0.5]{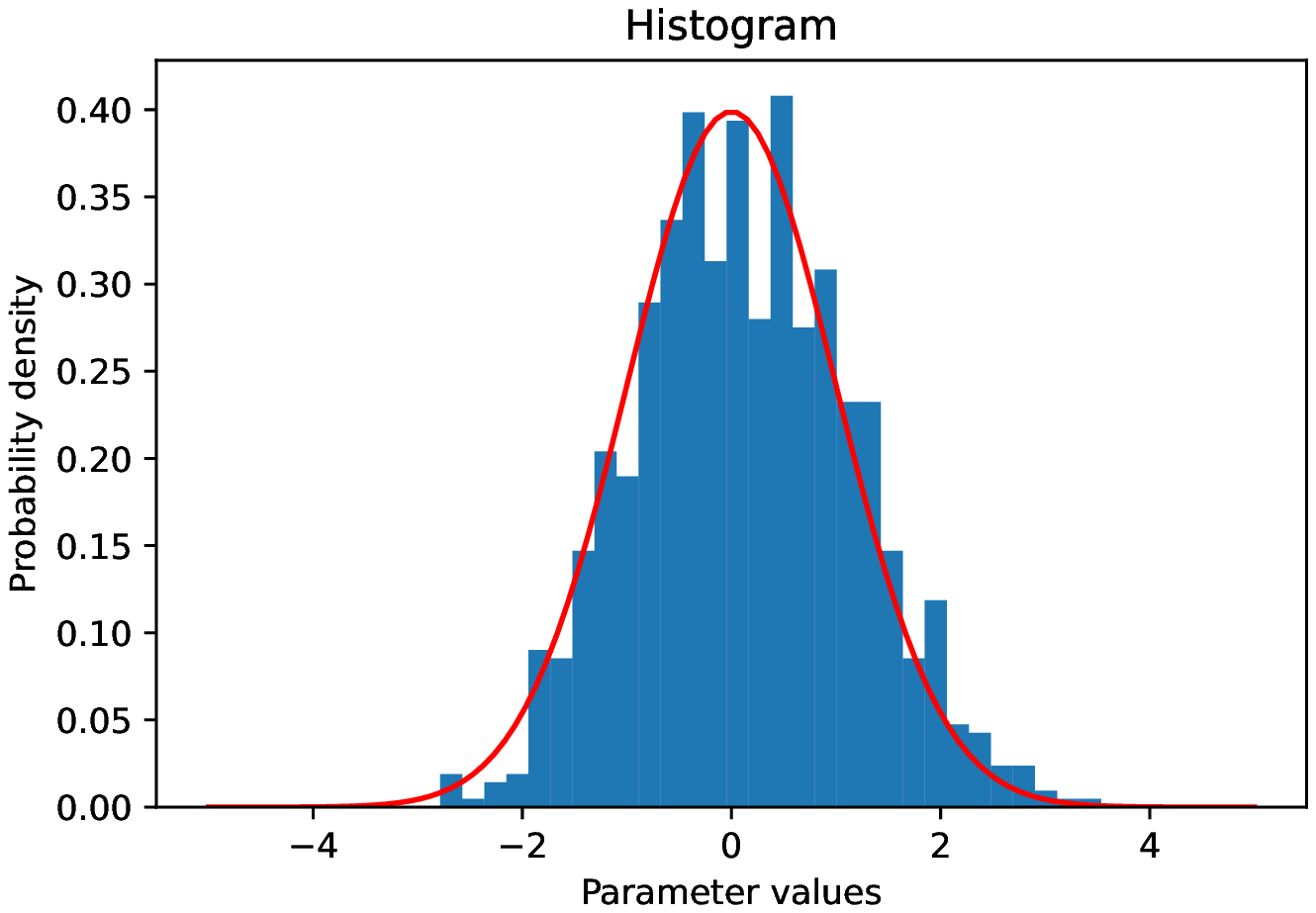}
   
  \end{minipage}
  \begin{minipage}[b]{0.45\linewidth}
    \centering
    \includegraphics[keepaspectratio, scale=0.5]{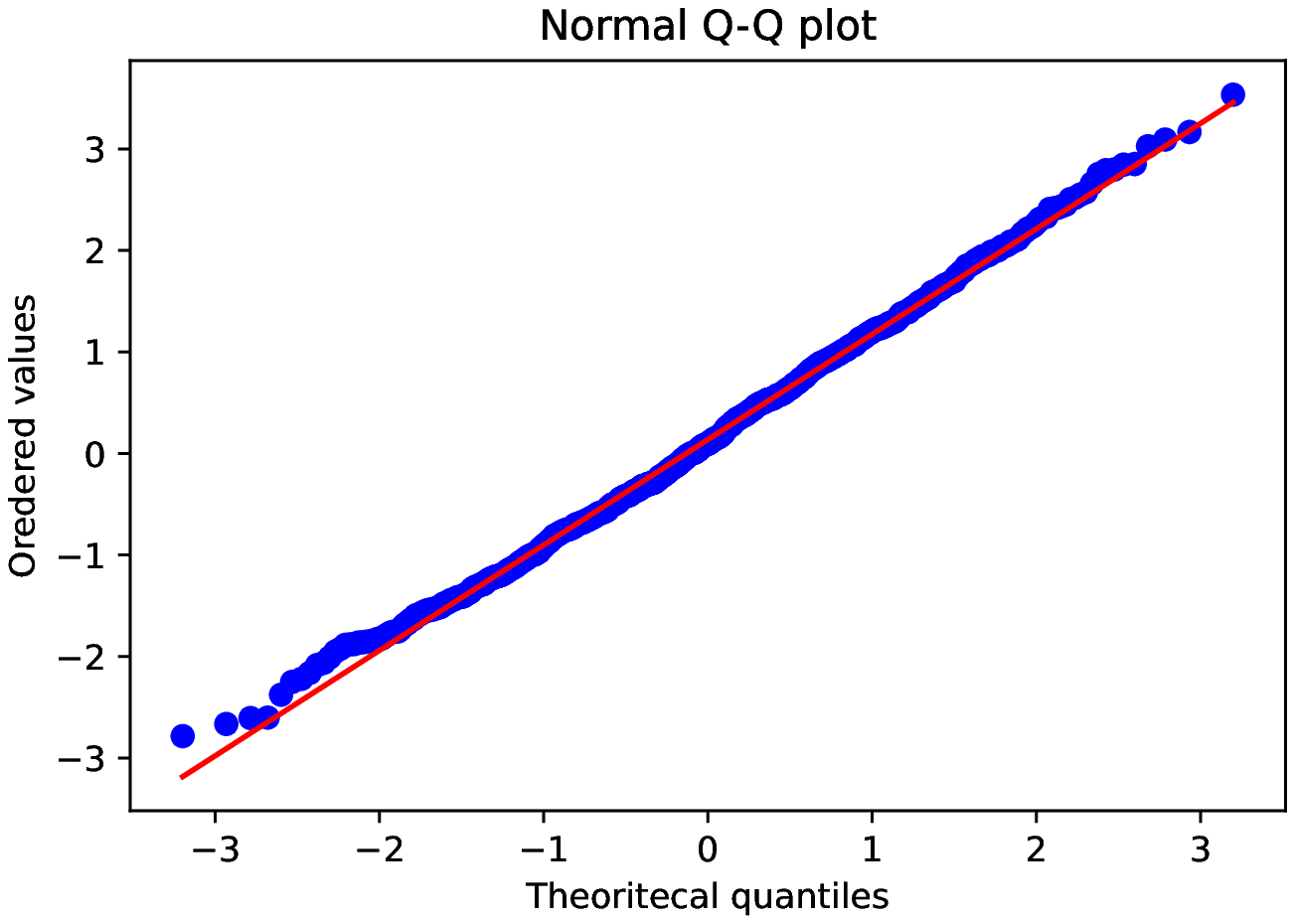}   
  \end{minipage}
  \caption{Histogram and density function of a standard normal distribution (left) and Normal Q-Q plot (right) through $1000$ experiments in the case $H=0.25,\theta=1.0,\varepsilon=0.01, T=1.0, n=500$}
  \label{fig14}
\end{figure}

\begin{figure}[htbp]
  \begin{minipage}[b]{0.45\linewidth}
    \centering
    \includegraphics[keepaspectratio, scale=0.5]{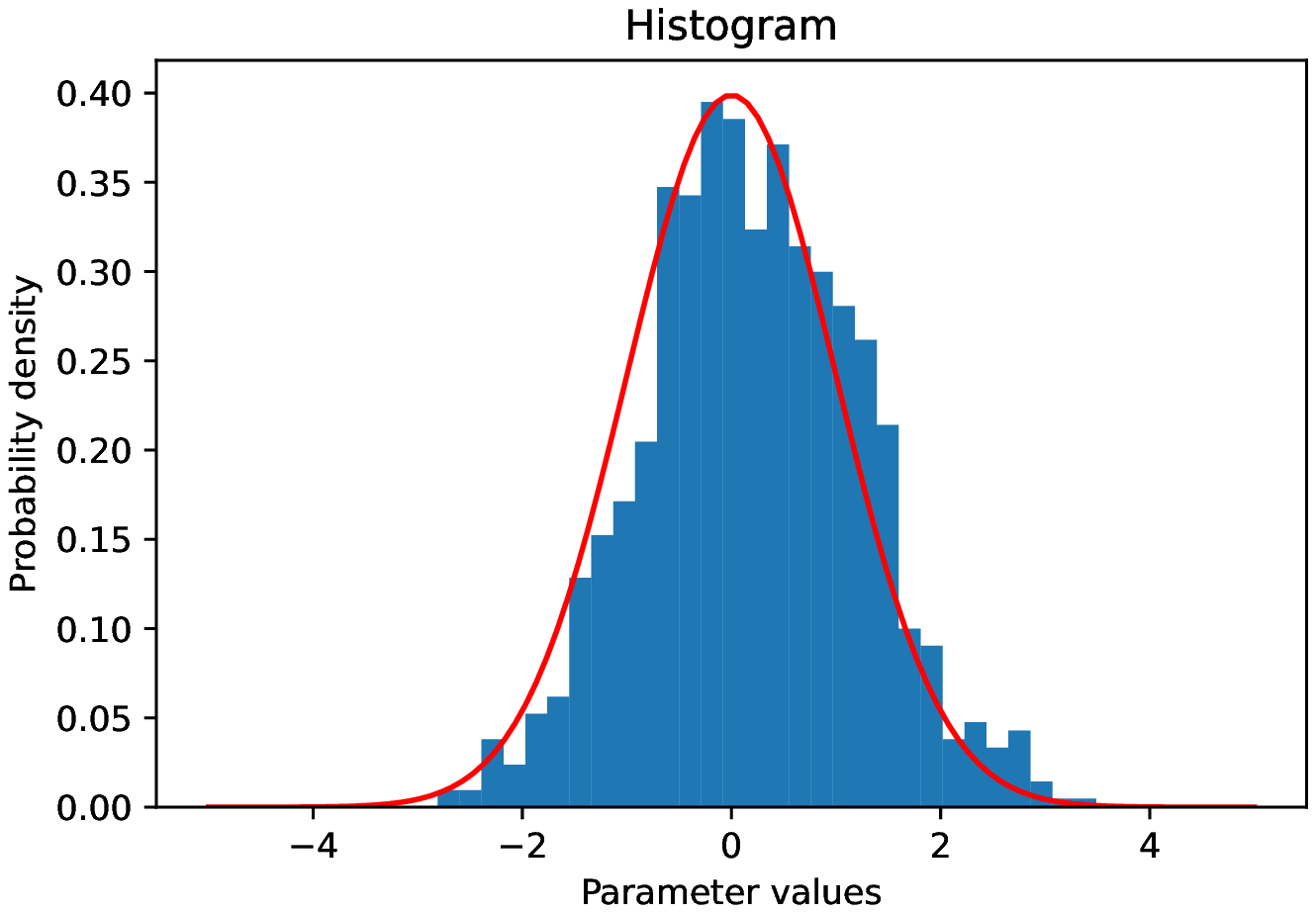}
   
  \end{minipage}
  \begin{minipage}[b]{0.45\linewidth}
    \centering
    \includegraphics[keepaspectratio, scale=0.5]{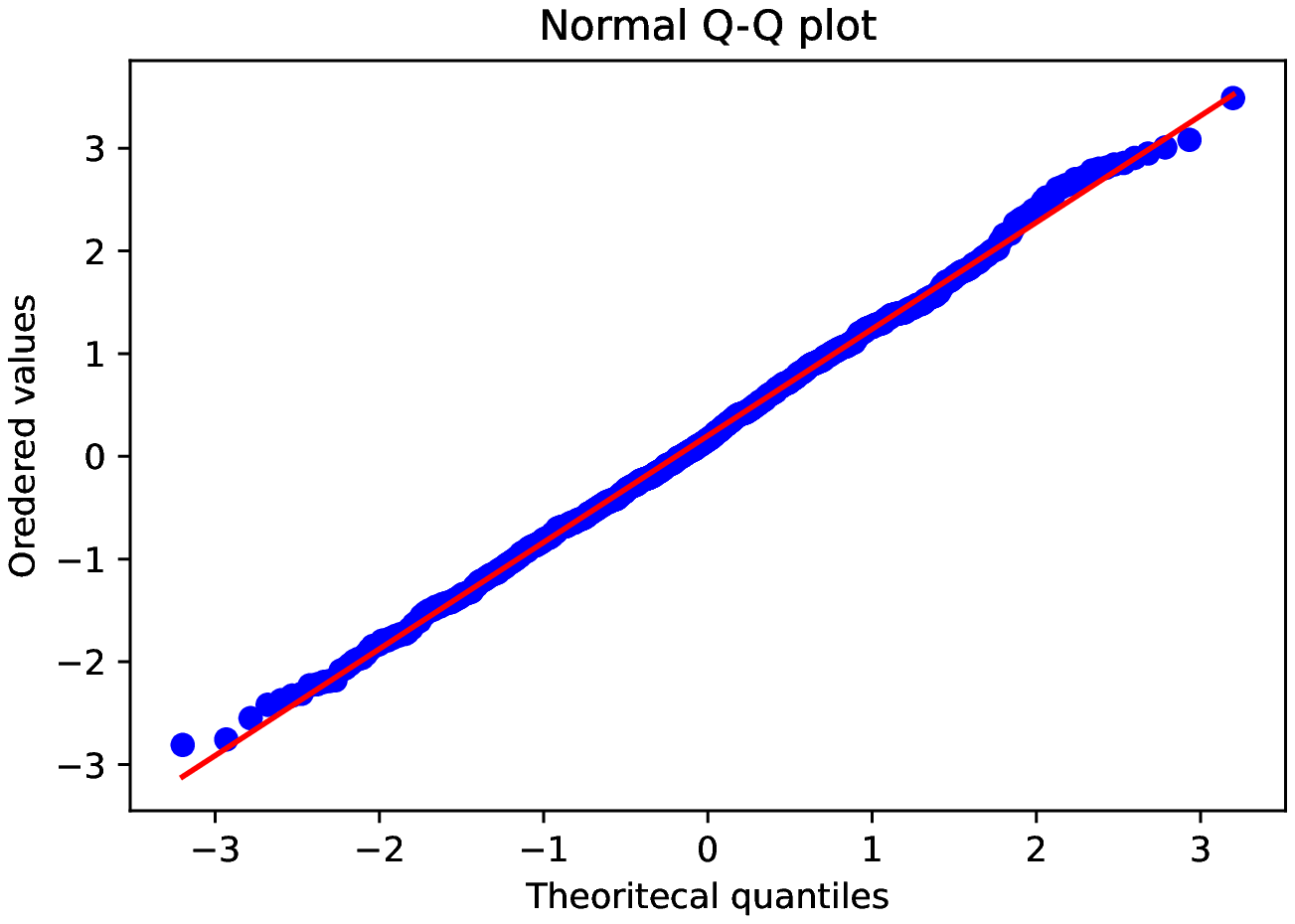}
   
  \end{minipage}
  \caption{Histogram and density function of a standard normal distribution (left) and Normal Q-Q plot (right) through $1000$ experiments in the case $H=0.25,\theta=1.0,\varepsilon=0.01, T=1.0, n=1000$}
  \label{fig15}
\end{figure}

\begin{figure}[htbp]
  \begin{minipage}[b]{0.45\linewidth}
    \centering
    \includegraphics[keepaspectratio, scale=0.5]{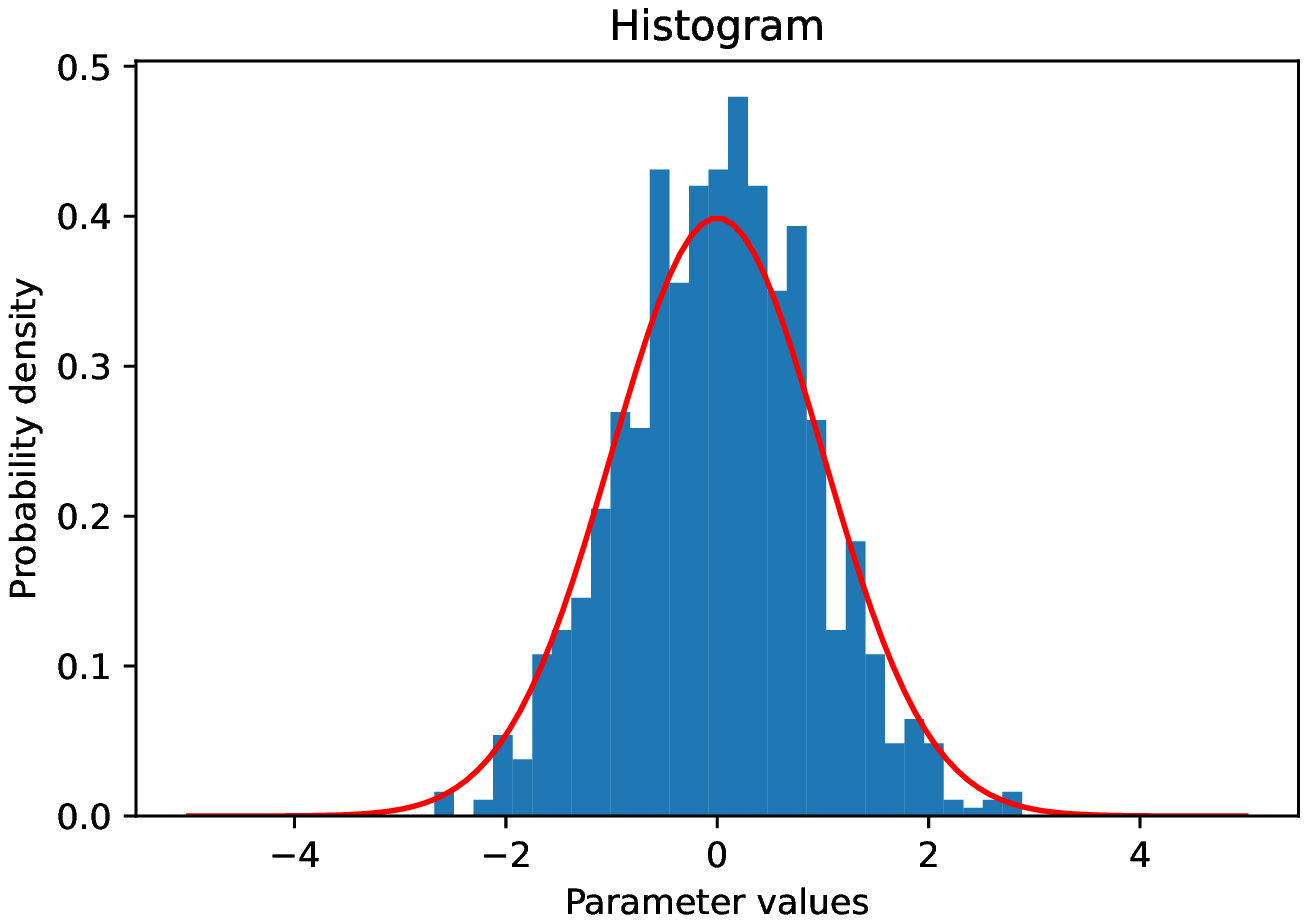}
   
  \end{minipage}
  \begin{minipage}[b]{0.45\linewidth}
    \centering
    \includegraphics[keepaspectratio, scale=0.5]{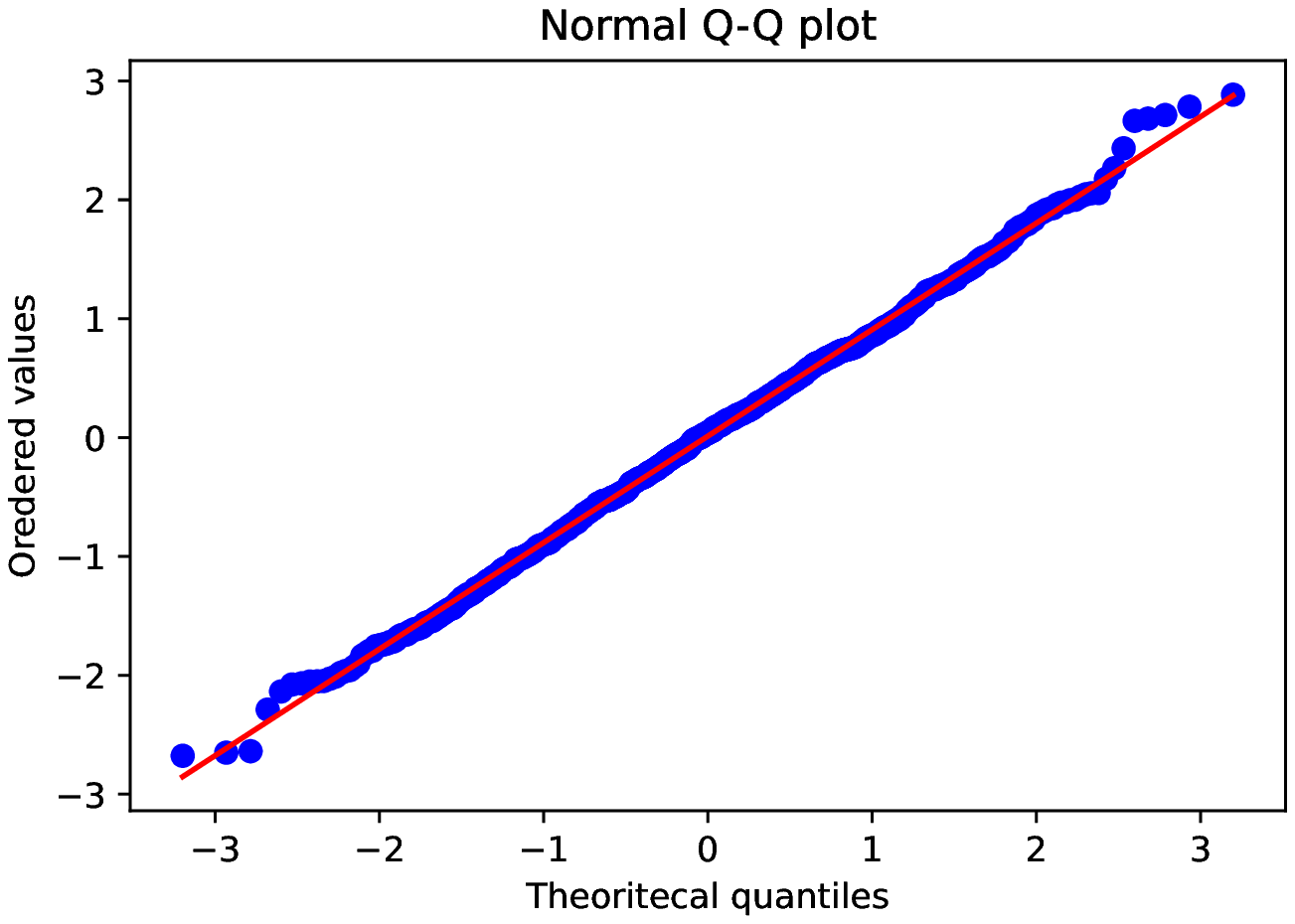}
   
  \end{minipage}
  \caption{Histogram and density function of a standard normal distribution (left) and Normal Q-Q plot (right) through $1000$ experiments in the case $H=0.25,\theta=1.0,\varepsilon=0.005, T=1.0, n=100$}
  \label{fig16}
\end{figure}

\begin{figure}[htbp]
  \begin{minipage}[b]{0.45\linewidth}
    \centering
    \includegraphics[keepaspectratio, scale=0.5]{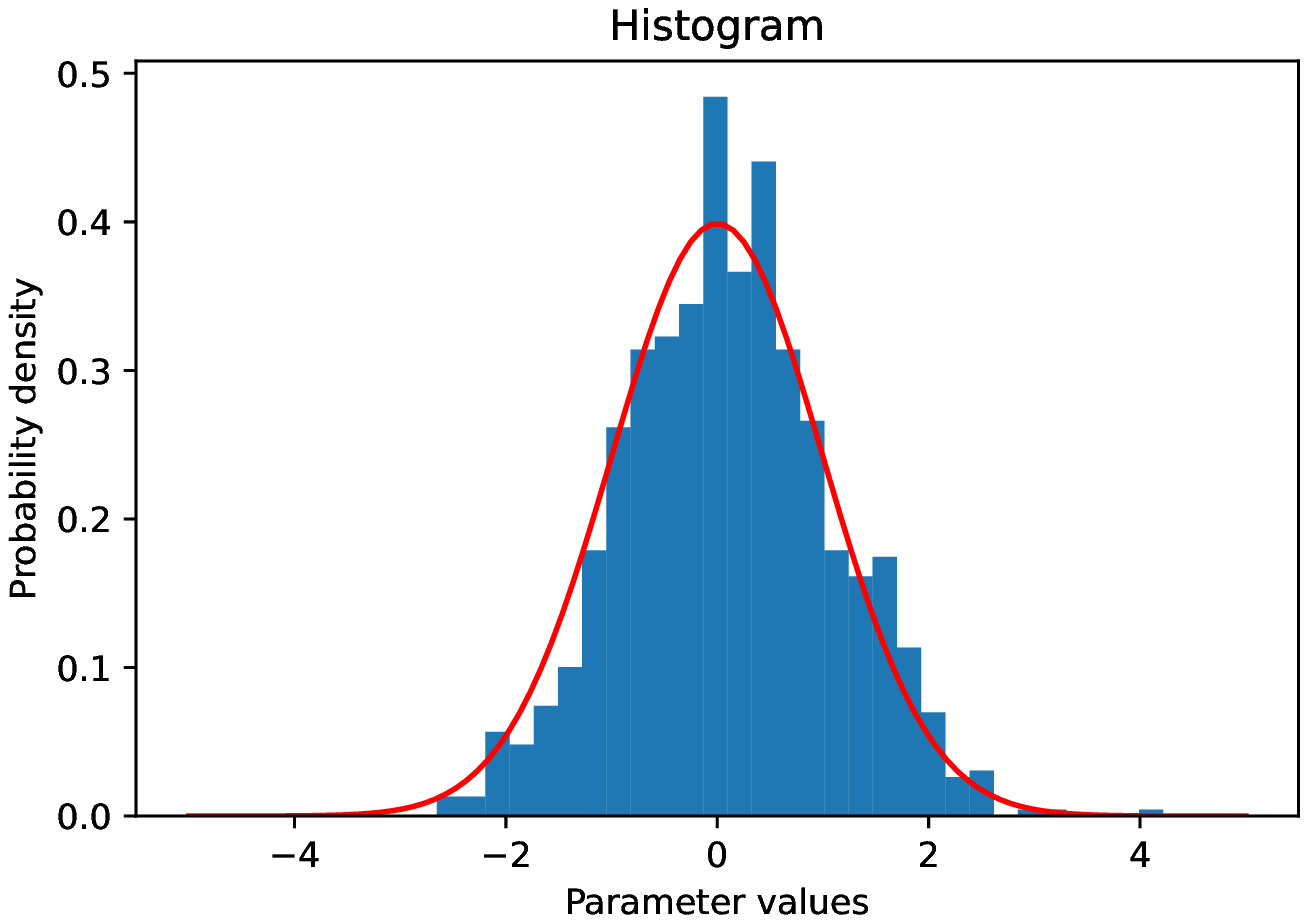}
   
  \end{minipage}
  \begin{minipage}[b]{0.45\linewidth}
    \centering
    \includegraphics[keepaspectratio, scale=0.5]{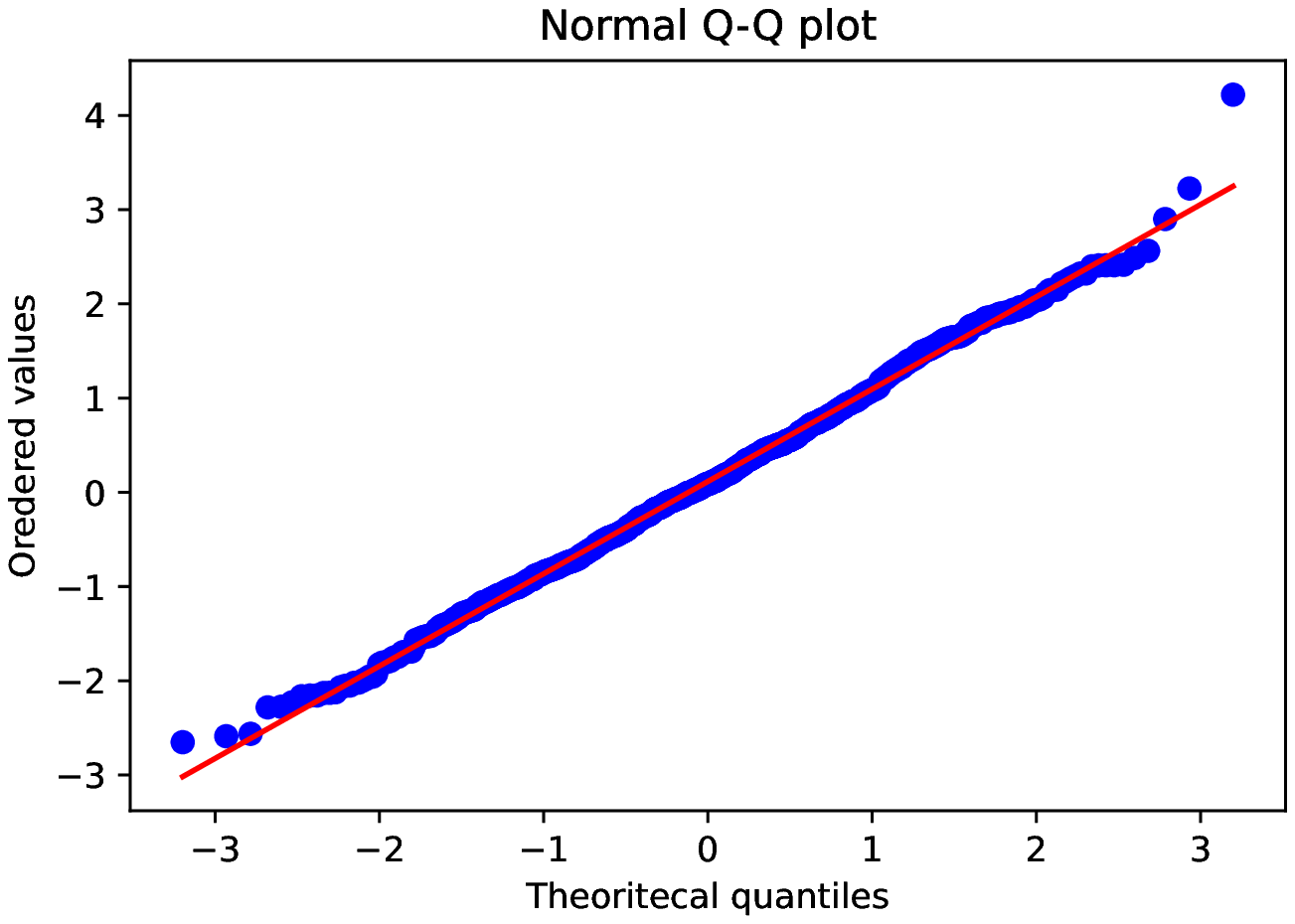}
   
  \end{minipage}
  \caption{Histogram and density function of a standard normal distribution (left) and Normal Q-Q plot (right) through $1000$ experiments in the case $H=0.25,\theta=1.0,\varepsilon=0.005, T=1.0, n=500$}
  \label{fig16}
\end{figure}

\begin{figure}[htbp]
  \begin{minipage}[b]{0.45\linewidth}
    \centering
    \includegraphics[keepaspectratio, scale=0.5]{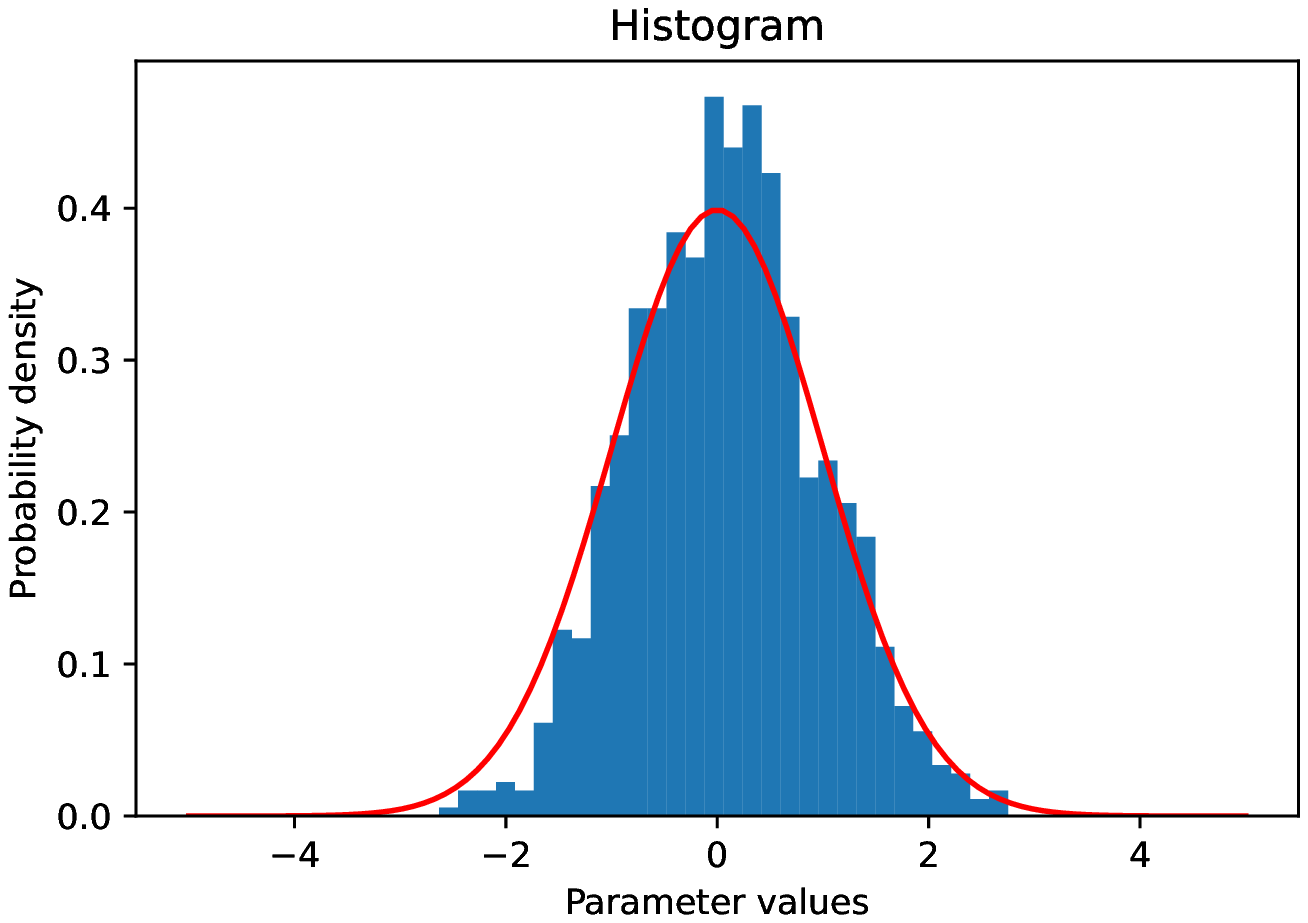}
   
  \end{minipage}
  \begin{minipage}[b]{0.45\linewidth}
    \centering
    \includegraphics[keepaspectratio, scale=0.5]{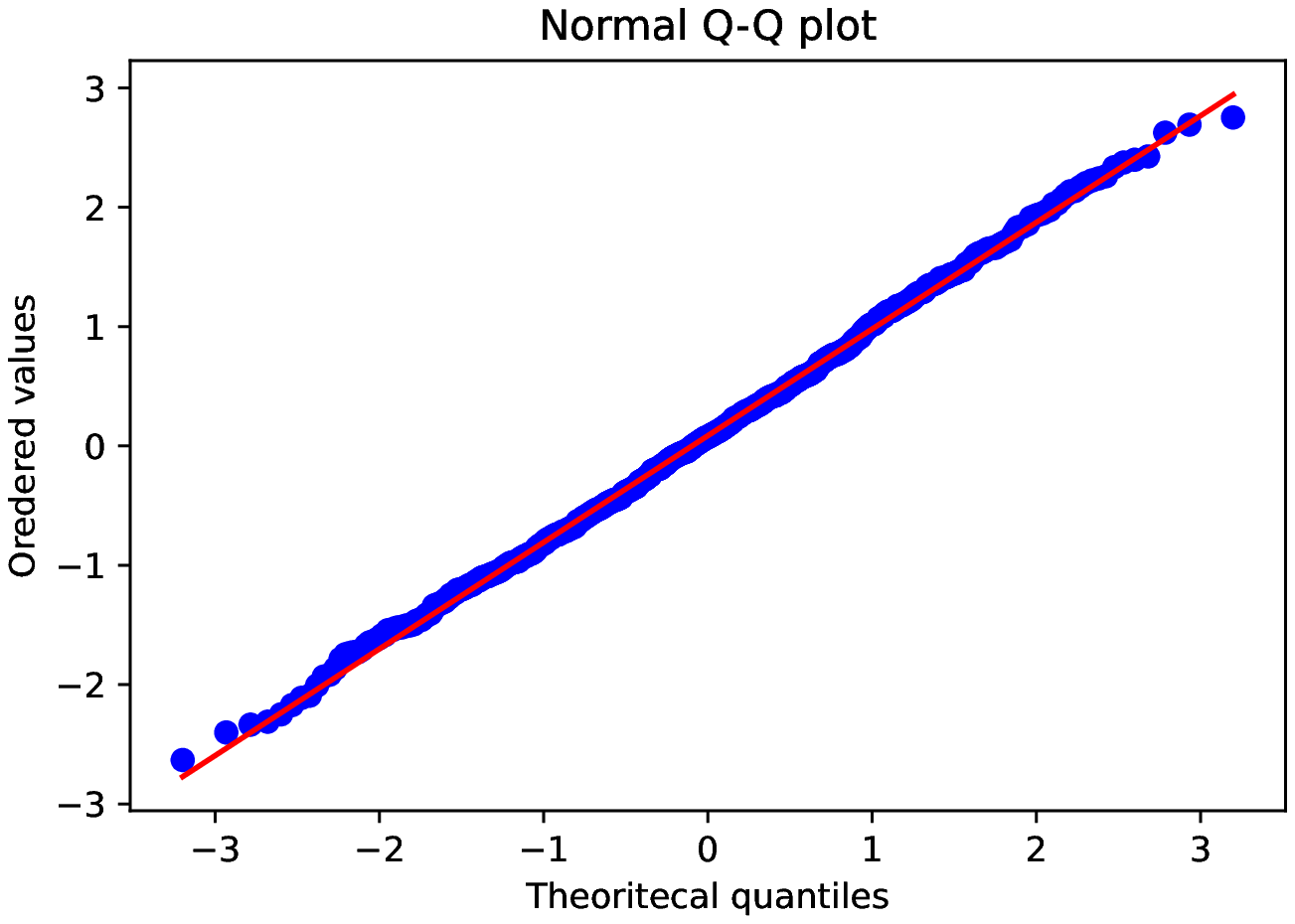}
   
  \end{minipage}
  \caption{Histogram and density function of a standard normal distribution (left) and Normal Q-Q plot (right) through $1000$ experiments in the case $H=0.25,\theta=1.0,\varepsilon=0.005, T=1.0, n=1000$}
  \label{fig17}
\end{figure}

\begin{figure}[htbp]
  \begin{minipage}[b]{0.45\linewidth}
    \centering
    \includegraphics[keepaspectratio, scale=0.5]{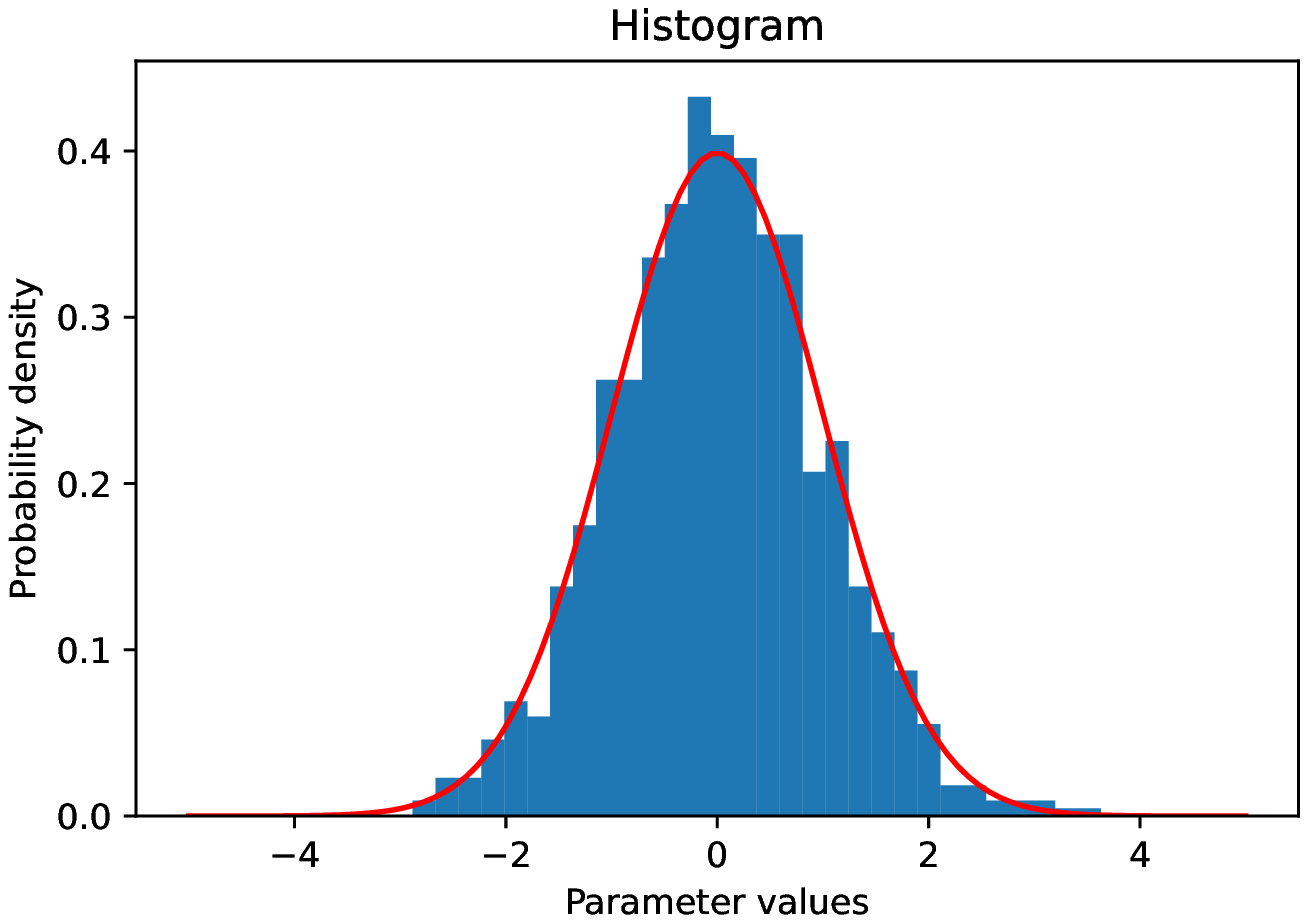}
   
  \end{minipage}
  \begin{minipage}[b]{0.45\linewidth}
    \centering
    \includegraphics[keepaspectratio, scale=0.5]{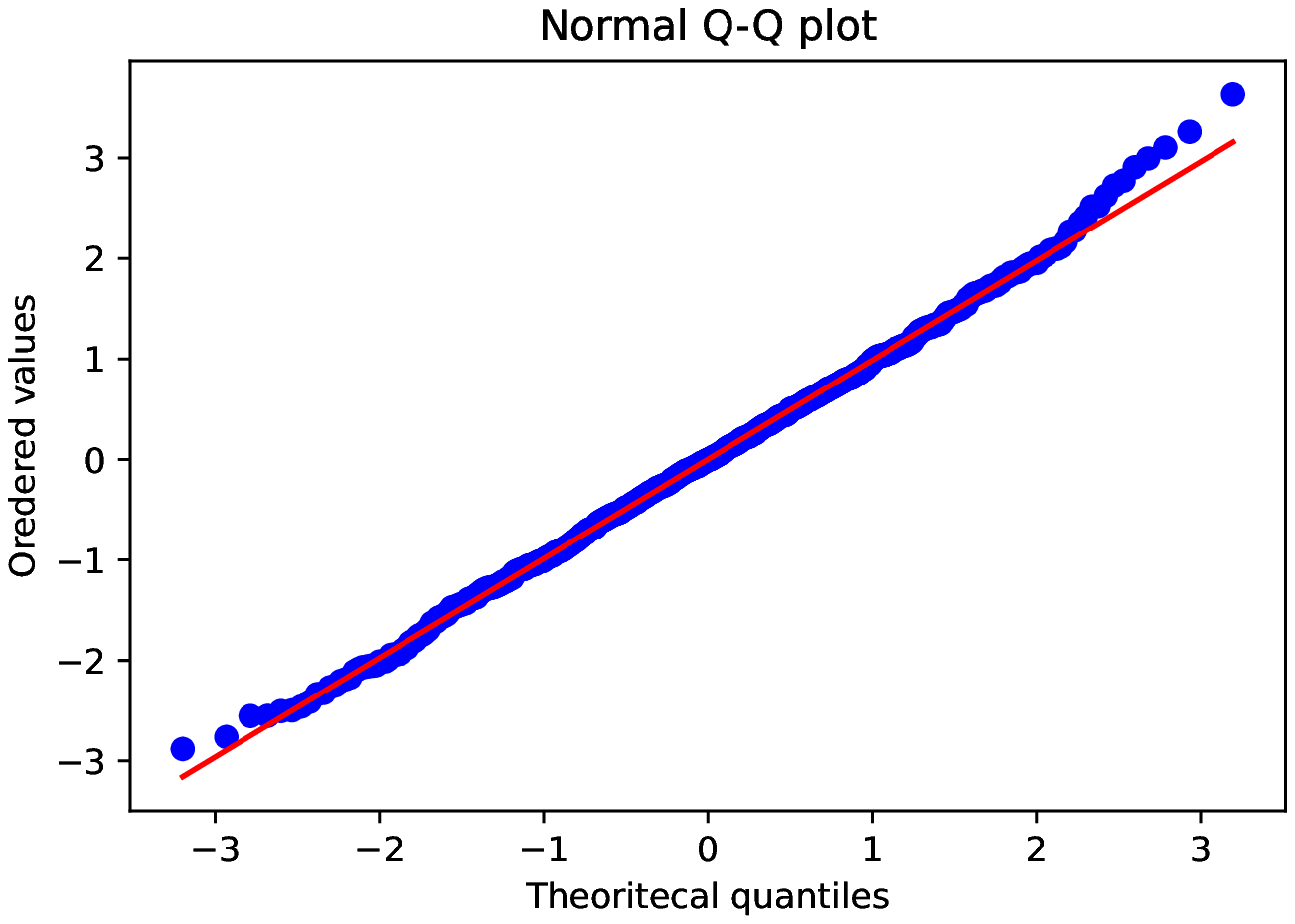}
   
  \end{minipage}
  \caption{Histogram and density function of a standard normal distribution (left) and Normal Q-Q plot (right) through $1000$ experiments in the case $H=0.25,\theta=1.0,\varepsilon=0.001, T=1.0, n=100$}
  \label{fig18}
\end{figure}

\begin{figure}[htbp]
  \begin{minipage}[b]{0.45\linewidth}
    \centering
    \includegraphics[keepaspectratio, scale=0.5]{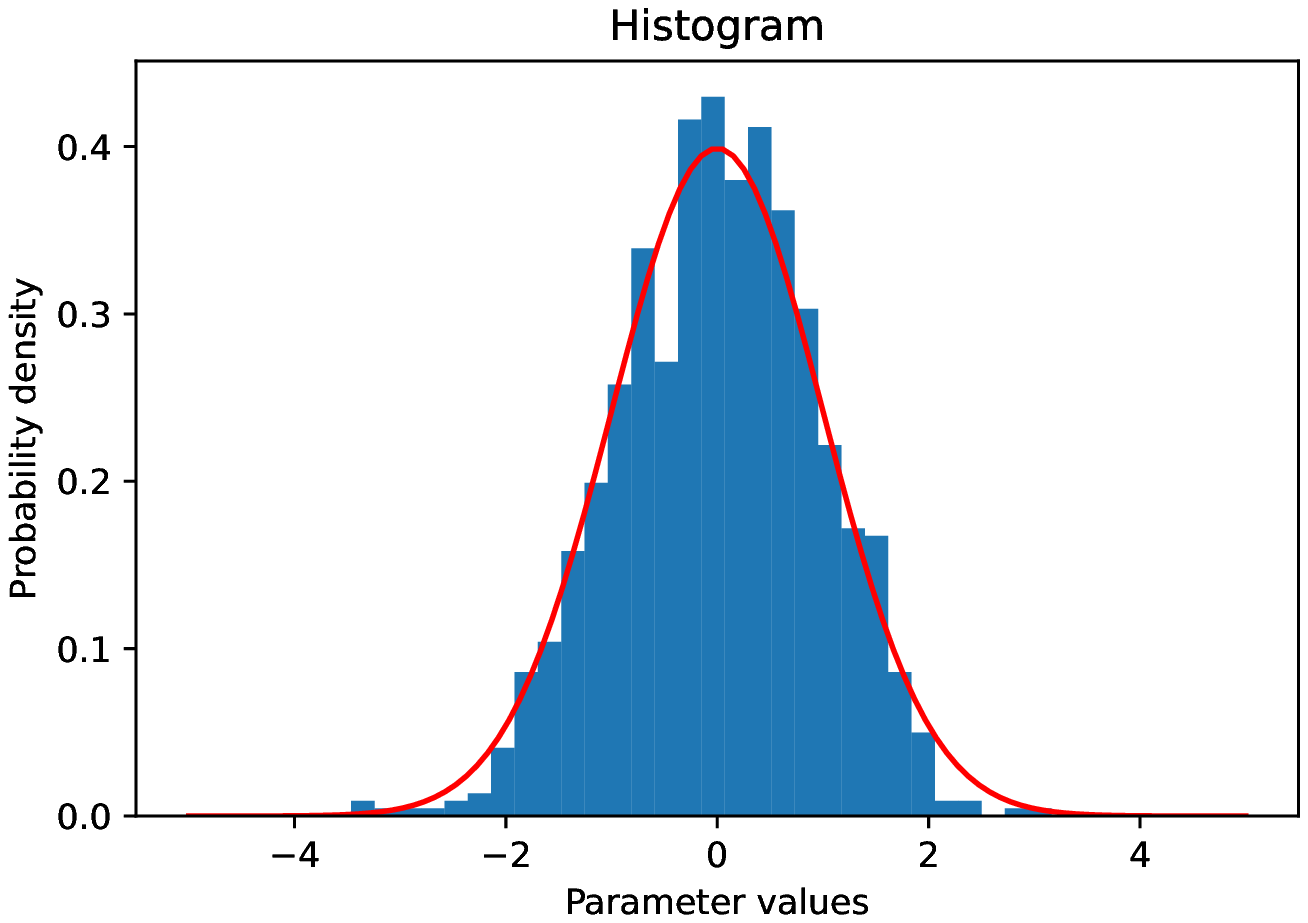}
   
  \end{minipage}
  \begin{minipage}[b]{0.45\linewidth}
    \centering
    \includegraphics[keepaspectratio, scale=0.5]{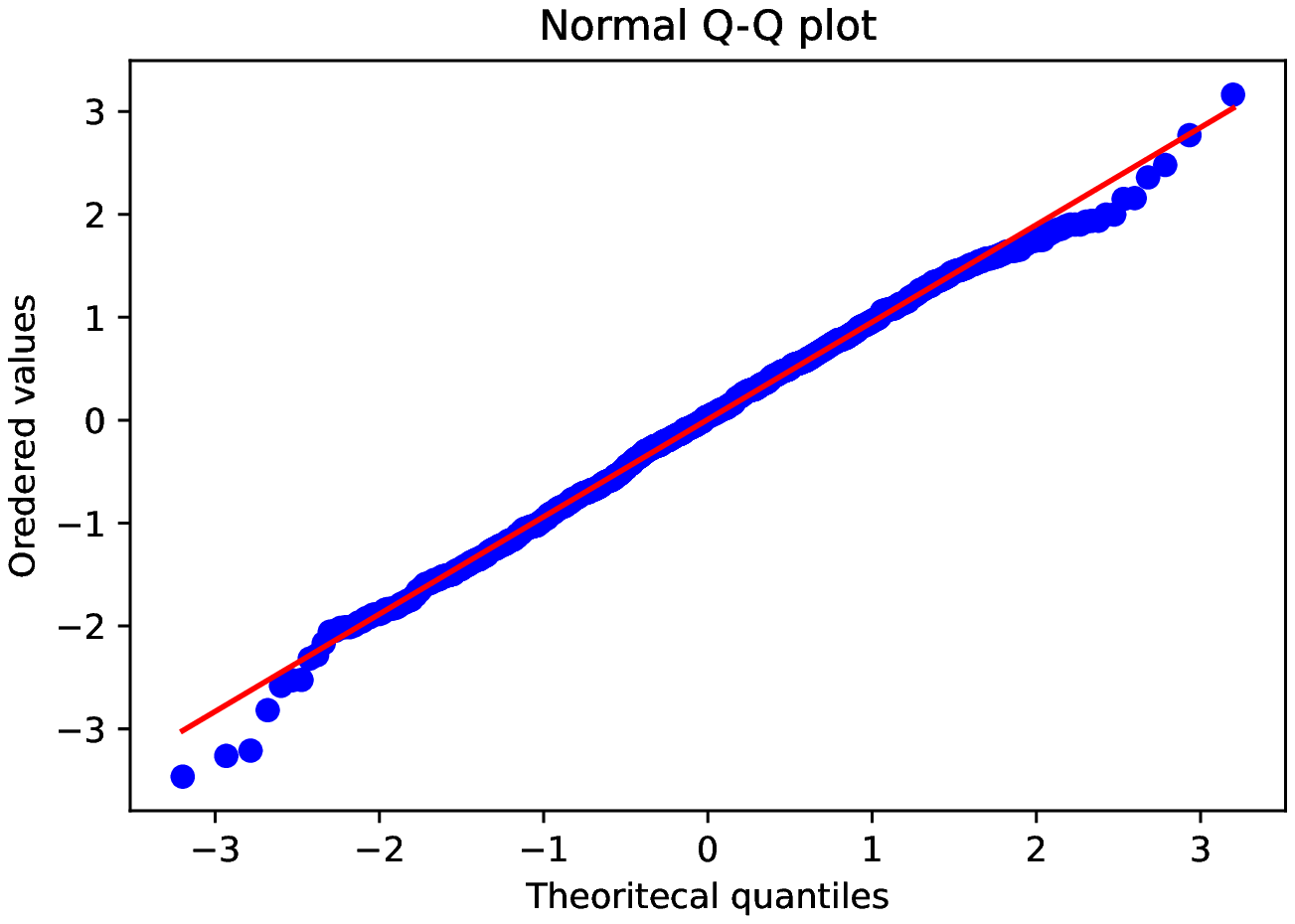}
   
  \end{minipage}
  \caption{Histogram and density function of a standard normal distribution (left) and Normal Q-Q plot (right) through $1000$ experiments in the case $H=0.25,\theta=1.0,\varepsilon=0.001, T=1.0, n=500$}
  \label{fig19}
\end{figure}

\begin{figure}[htbp]
  \begin{minipage}[b]{0.45\linewidth}
    \centering
    \includegraphics[keepaspectratio, scale=0.5]{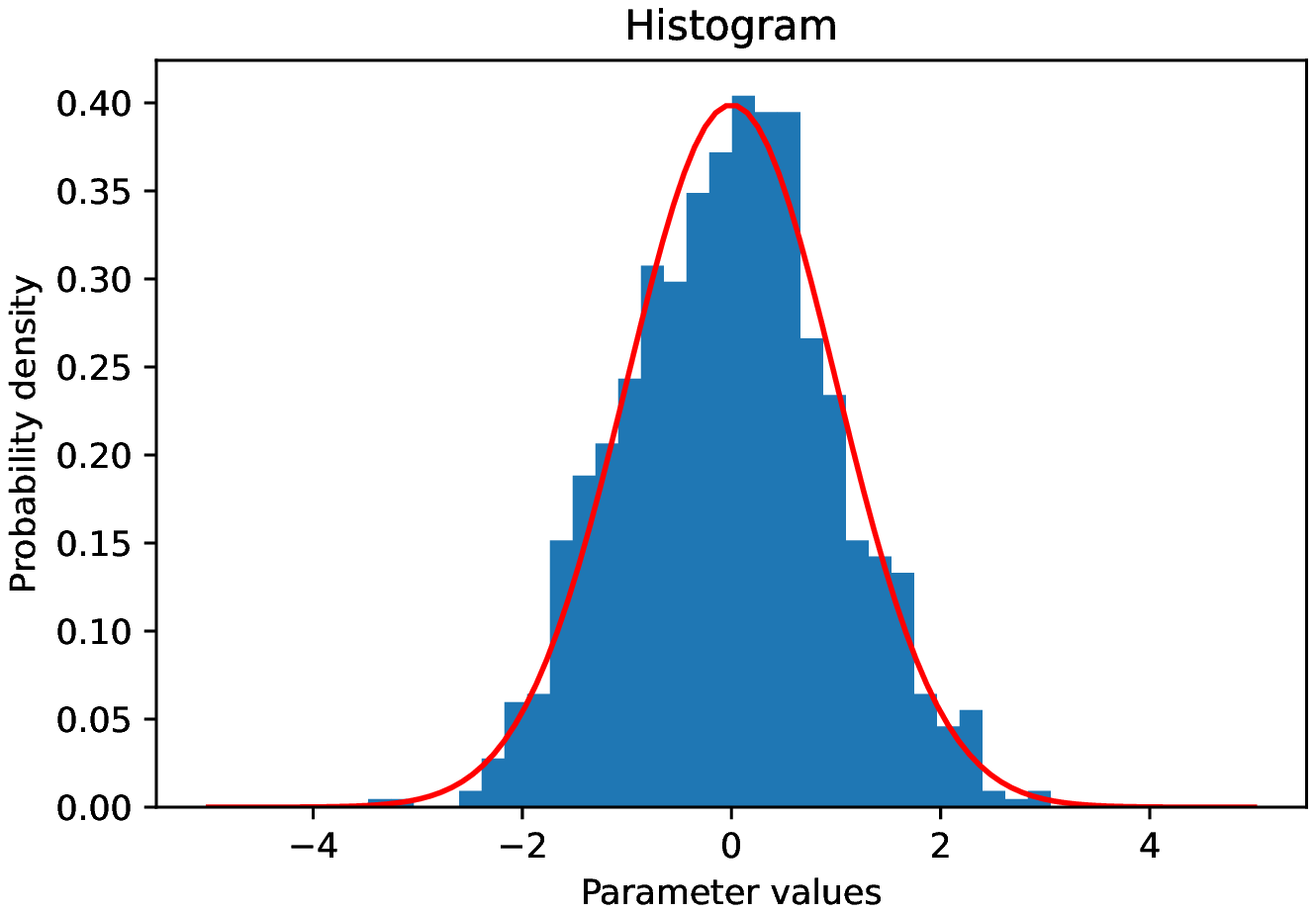}
   
  \end{minipage}
  \begin{minipage}[b]{0.45\linewidth}
    \centering
    \includegraphics[keepaspectratio, scale=0.5]{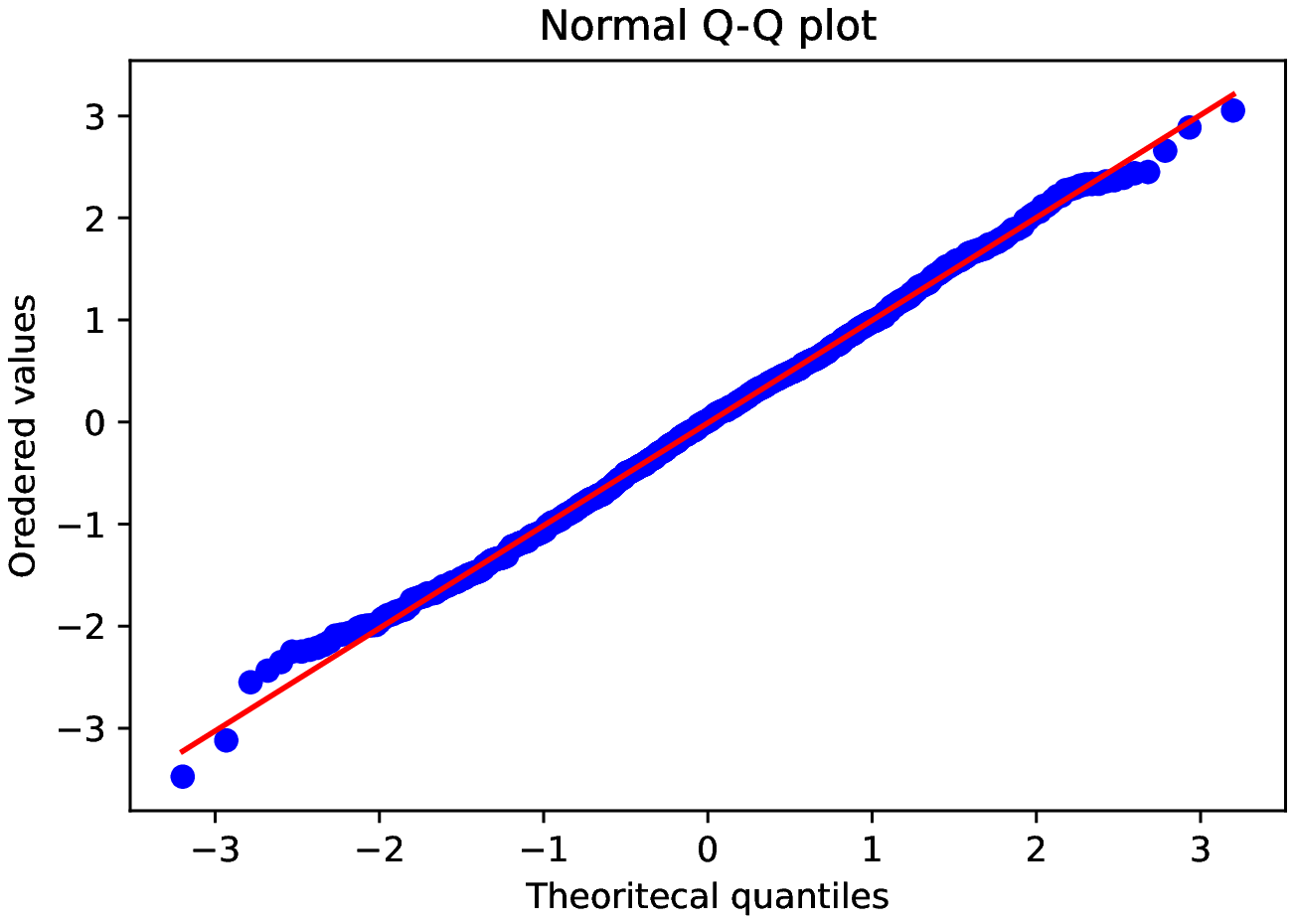}
   
  \end{minipage}
  \caption{Histogram and density function of a standard normal distribution (left) and Normal Q-Q plot (right) through $1000$ experiments in the case $H=0.25,\theta=1.0,\varepsilon=0.001,T=1.0, n=1000$}
    \label{fig20}
\end{figure}

From Table \ref{table1}-\ref{table2}, we can observe the strong consistency results in Theorem \ref{consistency} hold true when $\varepsilon \rightarrow 0 , n \rightarrow \infty$. In addition, we can see  that asymptotic normality in Theorem \ref{main3} hold from Figure \ref{fig1}-\ref{fig9} and Figure \ref{fig13}-\ref{fig20}. When $H = 0.25 $ and $\varepsilon=0.1$ , the convergence order assumptions of Theorem \ref{main3} is not satisfied for $n=100, 500, 1000$, so we can observe that asymptotic normality does not hold, as shown in Figure \ref{fig10}-\ref{fig12}.


\end{document}